\documentclass[12pt,a4paper]{amsart}
\usepackage{amsthm}
\usepackage{amssymb}
\usepackage{amsmath}
\usepackage{mathrsfs}
\usepackage{amscd}
\usepackage[all]{xy}

\usepackage{tikz}
\usepackage{pdflscape}
\usepackage{amsmath}
\usepackage{amsfonts}
\usepackage{amssymb}
\usepackage{mathtools}
\usepackage[all]{xy}
\usepackage{color}
\usetikzlibrary{arrows}
\usepackage{float}
\usepackage[shortlabels]{enumitem}
\usepackage{ifpdf}
\ifpdf
  \usepackage[colorlinks,
            linkcolor=magenta,       
            anchorcolor=magenta,     
            citecolor=magenta,       
            final,
            hyperindex]{hyperref}
\else
  \usepackage[colorlinks,
            linkcolor=magenta,       
            anchorcolor=magenta,     
            citecolor=magenta,       
            final,
            hyperindex,driverfallback=hypertex]{hyperref}
\fi
\newtheorem{theorem}{Theorem}[section]
\newtheorem{lemma}[theorem]{Lemma}
\newtheorem{coro}[theorem]{Corollary}
\newtheorem{problem}[theorem]{Problem}

\newtheorem{prop}[theorem]{Proposition}
\theoremstyle{definition}
\newtheorem{defn}[theorem]{Definition}
\newtheorem{remark}[theorem]{Remark}

\newcommand{\delete}[1]{}

\newcommand {\dd}{{\bf d}}
\newcommand {\rr}{{\bf r}}
\newcommand {\im}{{\rm im }}
\newcommand {\ra}{{\rm ran }}
\newcommand {\dm}{{\rm dom }}

\begin{document}
\title[Chain projection ordered categories and DRC-restriction semigroups]{Chain projection ordered categories and DRC-restriction semigroups}

\author{Yin Die}
\address{School of Mathematics, Yunnan Normal University,  Kunming, Yunnan, 650500,   China.}
\email{204585080@qq.com}

\author{Shoufeng Wang$^{\ast}$}\thanks{*Corresponding author}
\address{School of Mathematics, Yunnan Normal University,  Kunming, Yunnan, 650500,  China.}
\email{wsf1004@163.com}


\begin{abstract}
In this paper we provide a theory of chain projection ordered categories and generalize that of chain projection ordered groupoids developed by  East and  Azeef Muhammed recently.  By using  chain projection ordered categories, we obtain a structure theorem for DRC-restriction semigroups. More specifically, we prove that the category  of DRC-restriction semigroups together with (2,1,1)-homomorphisms is isomorphic to the category of chain projection ordered categories together with chain projection ordered functors.  Moreover, some special cases are also considered. As applications of the main theorem, we demonstrate that the existence of free projection-generated DRC-restriction semigroups associated to any strong two-sided projection algebra and reobtain the structures of projection-fundamental  DRC-restriction semigroups.  Our work may be regarded as an   answer  for the fourth problem proposed by  East and  Azeef Muhammed in [Advances in Mathematics, 437 (2024) 109447].
\end{abstract}


\maketitle



\vspace{-0.8cm}

\section{Introduction}
Let $(S, \cdot)$ be a semigroup. As usual, we denote  the set of  all idempotents in $S$ by $E(S)$ and the set of all inverses of
$x\in S$ by $V(x)$. Recall that $V(x)=\{a\in S\mid xax=x, axa=a\}$ for all  $x\in S$.
A semigroup $S$ is called  {\em regular}  if $V(x)\not=\emptyset$ for any $x\in S$, and a regular semigroup $S$
is called   {\em inverse}  if $E(S)$ is a commutative subsemigroup  of $S$,  or equivalently, the  cardinal   of
$V(x)$  is equal to $1$ for  all $x\in S$  (see Howie \cite{Howie1} for example).
 The theory of inverse semigroups is perhaps the best developed within semigroup theory (see the monographs \cite{Lawson1,Petrich1}).

It is well known that a semigroup $S$ is an inverse semigroup if and only if there exists a unary operation $\circ: S\rightarrow S, x\mapsto x^\circ$ satisfying the following axioms:
\begin{equation}\label{inversekehua}
xx^\circ x=x,\, x^{\circ\circ}=x,\, (xy)^\circ =y^\circ x^\circ,\, xx^\circ yy^\circ =yy^\circ xx^\circ.
\end{equation}
Inspired by this fact, in 1978  Nordahl and Scheiblich \cite{Nordahl1} introduced  regular $\circ$-semigroups as a generalization of inverse semigroups in the range of regular semigroups. Recall that a unary semigroup $(S,\cdot,\circ)$ is called a {\em regular $\circ$-semigroup}  if the following axioms hold:
\begin{equation}\label{regularstar}
xx^{\circ}x=x,\,x^{\circ\circ}=x,\,(xy)^{\circ}=y^\circ x^\circ.
\end{equation}
Inverse semigroups are regular $\circ$-semigroups obviously. But the converse is not true.
Regular $\circ$-semigroups have been inverstigated extensively in literatures (for example, see \cite{Adair1,Auinger1,East1,East2,East3,Imaoka1,Imaoka2,Jones1,Nordahl1,Polak1,Szendrei1,Yamada1}, etc.)

On the other hand, in 1991 Lawson introduced Ehresmann semigroups as a generalization of inverse semigroups in the range of non-regular semigroups by using generalized Green relation developed by ``York School".  Let $S$ be a semigroup and let $E\subseteq E(S)$. The relation $\widetilde{\mathcal R}_E$ is defined on $S$ by the rule that for any
$x,y\in S$, we have $x \widetilde{\mathcal R}_E y$ if
$$ex = x \mbox{ if and only if } ey =y \mbox{ for all } e \in E.$$ Dually, we have the relation $\widetilde{\mathcal L}_E$ on $S$.
Observe that both $\widetilde{\mathcal R}_E$ and $\widetilde{\mathcal L}_E$ are equivalences on $S$ but $\widetilde{\mathcal R}_E$ (respectively, $\widetilde{\mathcal L}_E$) may not be a left congruence (respectively, a right congruence).  Recall that a {\em band} is a semigroup in which every element is idempotent and a {\em semilattice} is a commutative band. In view of Lawson \cite{Lawson2}, $(S,E)$ is called {\em $E$-Ehresmann} if
\begin{itemize}
\item$E$ is a subsemilattice of $S$,
\item every $\widetilde{\mathcal R}_E$-class contains an element of $E$  and $\widetilde{\mathcal R}_E$ is a left
congruence,
\item every $\widetilde{\mathcal L}_E$-class contains an element of $E$  and $\widetilde{\mathcal L}_E$ is a right
congruence.
\end{itemize}
If this is the case,  $E$ is called {\em the distinguished  semilattice} of $S$, and the $\widetilde{\mathcal R}_E$-class (respectively, $\widetilde{\mathcal L}_E$-class) containing $a\in S$ has a unique element in $E$ which will be denoted by $a^+$ (respectively, $a^\ast$).  An $E$-Ehresmann semigroup $(S,E)$ is called
{\em $E$-restriction}, if $ae=(ae)^+a$ and $ea=a(ea)^\ast$ for all $e\in E$ and $a\in S$.  We say that a semigroup $S$  is {\em restriction} (respectively, {\em Ehresmann}) in the sequel if $(S,E)$ is an $E$-restriction (respectively, $E$-Ehresmann) semigroup for some $E\subseteq E(S)$. Let $(S, \cdot)$ be an inverse semigroup. Define $x^+=xx^{-1}$ and $x^\ast=x^{-1}x$ for all $x\in S$. Then it is easy to check that $(S, \cdot, ^+, ^\ast)$ forms a restriction semigroup.
From Lemmas 2.2 and 2.4 and their dual in Gould \cite{Gould1}, we have the following characterizations of Ehresmann semigroups and restriction semigroups from a varietal perspective.
\begin{lemma}[\cite{Gould1}]\label{Ehresmann}A semigroup $(S ,\cdot)$ is Ehresmann if and only if there are two unary operations $`` +"$ and $`` \ast"$ on $S$ such that the following axioms hold:
$$x^+ x=x,\, x^+y^+=y^+x^+=(x^+y^+)^+,\, x^+(xy)^+= (xy)^+=(xy^+)^+,\, x^{+\ast}=x^+;$$
$$xx^\ast=x,\, x^\ast y^\ast=y^\ast x^\ast=(x^\ast y^\ast)^\ast,\, (xy)^\ast y^\ast=(xy)^\ast =(x^\ast y)^\ast,\, x^{\ast+}=x^\ast.$$
In this case, we denote this Ehresmann semigroup by $(S, \cdot, ^+, ^\ast)$. An Ehresmann $(S, \cdot, ^+, ^\ast)$ is restriction if and only if it additionally satisfies the following axioms  $$\mbox{\em \bf ample conditions}\,\,\,\,\,\,\,\,\,\,\,\,(xy)^+x=xy^+,\,\, x(yx)^\ast =y^\ast x.$$
\end{lemma}

We observe that adequate semigroups and ample semigroups defined firstly in Fountain \cite{Fountain1} are special Ehresmann semigroups and restriction semigroups, respectively. We also observe that restriction semigroups have appeared also as the type
$SL2$ $\gamma$-semigroups \cite{Batbedat1,Batbedat2} in the early 1980s.  Restriction semigroups
also have arisen in the work of Jackson and Stokes \cite{Jackson1} in the guise of
twisted $C$-semigroups and in that of Manes \cite{Manes1} as guarded semigroups, motivated by consideration of closure operators and categories,
respectively. The work of Manes has a forerunner in the restriction categories of Cockett and Lack \cite{Cockett1}, who were influenced by considerations
of theoretical computer science. In fact the term ``restriction semigroup" is exactly taken from the above work of Cockett and Lack.
More detailed information on Ehresmann  and restriction semigroups can be found in the survey articles of Gould \cite{Gould1} and Hollings \cite{Hollings1}.

In 2012, Jones \cite{Jones1} introduced $P$-Ehresmann  and $P$-restriction semigroups  from the view  of variety  and gave a common framework for restriction semigroups and regular $\ast$-semigroups.
A bi-unary semigroup $(S, \cdot, ^+, ^\ast)$  is called a {\em $P$-Ehresmann semigroup} if the following axioms hold:
$$x^+x=x,\, (xy)^+=(xy^+)^+,\, (x^+y^+)^+=x^+y^+ x^+,\, x^+ x^+=x^+,\, x^{+\ast}={x}^+,$$
$$xx^\ast=x,\, (xy)^\ast=(x^\ast y)^\ast,\, (x^\ast y^\ast)^\ast=y^\ast x^\ast y^\ast,\, x^\ast x^\ast=x^\ast,\, x^{\ast+}={x}^\ast.$$
A $P$-Ehresmann semigroup $(S, \cdot, ^+, ^\ast)$ is called a {\em  $P$-restriction semigroup} if it additionally satisfies the following axioms :
$$\mbox{\em\bf $P$-ample conditions}\,\,\,\,\,\,\,\,\,\,\,\,\,\,(xy)^+x=xy^+{x}^\ast,\,\,\, x(yx)^\ast=x^+y^\ast x.$$ Let $(S, \cdot, \circ)$ be a regular $\circ$-semigroup. Define $x^+=xx^{\circ}$ and $x^\ast=x^{\circ}x$ for all $x\in S$. Then it is easy to check that $(S, \cdot, ^+, ^\ast)$ forms a $P$-restriction semigroup.
Some works on $P$-Ehresmann  and $P$-restriction semigroups can be found in \cite{Jones1,Jones4,Stein1,Stein2,Wang1,Wang2,Wang3,Wang5} and the references therein.

In 2021, Jones \cite{Jones3} introduced DRC semigroups  and further generalized $P$-Ehremsmann semigroups. A bi-unary  semigroup $(S, \cdot,\, {}^+, {}^\ast)$ is called a {\em DRC semigroup} if the following axioms hold:
\begin{center}\begin{tabular}{lll|lll}
{\rm (\romannumeral1)} & $x^+x=x$&\mbox{}\hskip 6mm\mbox{}
&\mbox{}\hskip 6mm\mbox{}&{\rm (\romannumeral1)}$'$ & $xx^\ast=x$\\
{\rm (\romannumeral2)} & $(xy)^+=(xy^+)^+$&&&{\rm (\romannumeral2)}$'$ & $(xy)^\ast=(x^\ast y)^\ast$\\
{\rm (\romannumeral3)} & $(xy)^+=x^+ (xy)^+ x^+$&&&{\rm (\romannumeral3)}$'$ & $(xy)^\ast=y^\ast (xy)^\ast y^\ast$\\
{\rm (\romannumeral4)} & $x^{++}=x^+$&&&{\rm (\romannumeral4)}$'$& $x^{\ast\ast}=x^\ast$\\
{\rm (\romannumeral5)} & $x^{+\ast}=x^+$&&&{\rm (\romannumeral5)}$'$& $x^{\ast+}=x^\ast$\\
\end{tabular}\\[3mm]
{\bf DRC Conditions}
\end{center}
In this case, we denote $P(S)=\{x^{+}\mid x\in S\}$ and call it {\em the set of projections} of $S$. By (\romannumeral5) and (\romannumeral5)$'$, we have
\begin{equation}
P(S)=\{x^{+}\mid x\in S\}=\{x^{\ast}\mid x\in S\}.
\end{equation}
Observe that DRC semigroups were termed {\em York semigroups} by Jones \cite{Jones2}, {\em DR semigroups satisfying the congruence conditions}  by Stokes \cite{Stokes1} and {\em reduced $E$-Fountain semigroups} by Stein \cite{Stein1,Stein2}, respectively.
Recently,   Wang \cite{Wang6}  introduced DRC-ample conditions for DRC semigroups by which DRC-restriction  semigroups are discriminated from DRC semigroups.
A DRC semigroup $(S, \cdot,\, {}^+, {}^\ast)$ is called {\em a DRC-restriction semigroup} if it also satisfies the following axioms:
 \begin{center}
 \begin{tabular}{llllll}
{\rm (\romannumeral6)} & $x(yx)^\ast=(yx^+)^\ast x$&&&{\rm (\romannumeral6)}$'$ & $(xy)^+x=x(x^\ast y)^+$\\
\end{tabular}\\[3mm]
{\bf DRC-ample conditions}
\end{center}
 In view of the identities (\romannumeral2)$'$ and (\romannumeral2), (vi) (respectively, (vi$'$)) is exactly the axiom $x(y^\ast x)^\ast=(y^\ast x^+)^\ast x$ (respectively, $(xy^+)^+x=x(x^\ast y^+)^+$). This implies that  (vi) and (vi)$'$  are equivalent to the conditions that
\begin{equation}\label{dengjia}
x(e x)^\ast=(e x^+)^\ast x \mbox{ and } (xe)^+x=x(x^\ast e)^+   \mbox{ for all } x\in S \mbox{ and } e\in P(S),
\end{equation} respectively.

Moreover, a regular version of DRC-restriction semigroups also provided in Wang \cite{Wang6}. A unary semigroup $(S,\cdot,\circ)$ is called a {\em generalized regular $\circ$-semigroup} if the following axioms hold:
\begin{equation}\label{g-regularstar}
xx^{\circ}x=x,\,x^{\circ\circ}=x,\,(xy)^{\circ}=(xy)^{\circ}xx^{\circ},\,(xy)^{\circ}=y^{\circ}y(xy)^{\circ},\,(xy)^{\circ}x=(x^{\circ}xy)^{\circ}.
\end{equation}
Obviously, regular $\circ$-semigroups are generalized regular $\circ$-semigroups. But the converse does not hold (see \cite[Example 3.3]{Wang6}).
Let $(S,\cdot,\circ)$ be a  generalized regular $\circ$-semigroup. Define $x^+=xx^\circ$ and $x^\ast=x^\circ x$ for all $x\in S$. By \cite[Proposition 3.7]{Wang6},
$(S, \cdot, ^+, ^\ast)$ forms a DRC-restriction semigroup. DRC semigroups and their generalizations  have been studied by several authors, see \cite{Jones2,Jones3,Stein1,Stein2,Stein3,Stokes1,Stokes2,Stokes3, Wang4,Wang6,Wang7} for example.

On the relationship among Ehresman (respectively, restriction), $P$-Ehresmann (respectively, $P$-restriction) and DRC (respectively, DRC-restriction) semigroups, we have the following result.
\begin{prop}[\cite{Gould1,Jones1,Wang6}]\label{guanxi} The following statements hold:
\begin{itemize}
\item[(1)]
A $P$-Ehresmann (respectively, $P$-restriction)  semigroup is exactly a DRC (respectively, DRC-restriction) semigroup in which $(ef)^+=efe=(fe)^\ast$ for all projections $e$ and $f$.
\item [(2)]An Ehresman (respectively, restriction) semigroup is exactly a DRC (respectively, DRC-restriction) semigroup in which $(ef)^+=ef=fe=(fe)^\ast$ for all projections $e$ and $f$.
\end{itemize}
\end{prop}
In view of the previous statements,  we have the following  hierarchy of the above mentioned classes of semigroups, where ``A $\rightarrow$ B" denote A is a subclass of B.
{\small
$$\xymatrix{
  \mbox{inverse semigroups} \ar[d]_{} \ar[r]^{} & \mbox{restriction semigroups} \ar[d]_{} \ar[r]^{} & \mbox{Ehresmann semigroups} \ar[d]^{} \\
  \mbox{regular}\circ\mbox{-semigroups} \ar[d]_{} \ar[r]^{} & P\mbox{-restriction semigroups}  \ar[d]_{} \ar[r]^{} & P\mbox{-Ehresmann semigroups} \ar[d]^{} \\
  \mbox{generalized regular}\circ\mbox{-semigroups} \ar[r]^{} & \mbox{DRC-restriction semigroups} \ar[r]^{} &   \mbox{DRC semigroups}  }$$
  }
\begin{center}
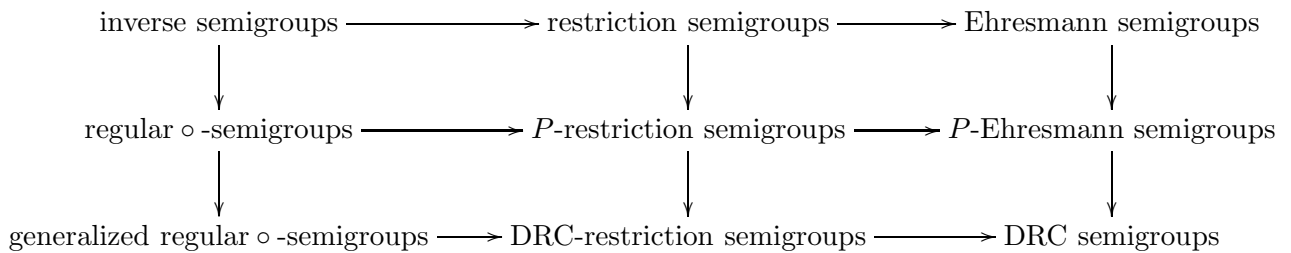
  {\bf Diagram 1: A hierarchy of several subclasses of DRC-semigroups}
\end{center}
According to Gould \cite{Gould1} and  Hollings \cite{Hollings3}, in the study of structure theory of inverse semigroups  three major approaches have emerged. These are: {\em categorical approach} developed in the texts \cite{Ehresmann1,Ehresmann2,Meakin1,Meakin2,Nambooripad1,Schein1,Schein2}),   {\em fundamental approach} developed in the texts \cite{Munn1,Munn2} and {\em covering approach} developed in the texts \cite{McAlister1,McAlister2}, respectively.  The main result on inverse semigroups obtained by using categorical approach was formulated by Lawson in \cite{Lawson1} as the following Ehresmann--Schein--Nambooripad Theorem, due to its varied authorship.
\begin{theorem}[\cite{Lawson1}]\label{esn}The category of inverse  semigroups together with  homomorphisms is isomorphic to the category  of inductive groupoids together with  inductive functors.
\end{theorem}
On the other hand, the main result on inverse semigroups obtained by using fundamental approach is the theorem below.
\begin{theorem}[\cite{Howie1}]\label{munn} Let $S$ be an inverse semigroup and $T_{E(S)}$ be the Munn semigroup of $E(S)$. Then there exists a homomorphism $\phi: S\rightarrow T_{E(S)}$ whose kernel is the maximum idempotent-separating congruence on $S$. Moreover, $S$ is fundamental if and   only if $S$ is isomorphic to a full inverse subsemigroup of $T_{E(S)}$.
\end{theorem}
Since our present paper does not involve covering approach, we  will no longer list the results on inverse semigroups obtained by using this approach.
Theorems \ref{esn} and \ref{munn} have been generalized in various directions, see for example \cite{Armstrong1,DeWolf1,East1,East2,FitzGerald1,Fountain1,Fountain2,Gomes1,Gould1,Gould2,Hollings2,Imaoka1,Imaoka2,Jones1,Jones3,Lawson2,Lawson3,Nambooripad1,Stokes4,Stokes2,Wang2,Wang3,Wang4,Wang6,Wang7,WangY}. More specifically, for regular case, Nambooripad \cite{Nambooripad1} generalized this two theorems to general regular semigroups; Imaoka and Fujiwara \cite{Imaoka2} and East and Azeef Muhammed \cite{East2} generalized Theorem \ref{esn} to locally inverse regular $\circ$-semigroups and general regular $\circ$-semigroups respectively; East and Azeef Muhammed \cite{East2}, Imaoka \cite{Imaoka1} and Jones \cite{Jones1} generalized Theorem \ref{munn} to regular $\circ$-semigroups. For non-regular cases, Fountain \cite{Fountain1}, Fountain  Gomes and  Gould \cite{Fountain2} and  Gomes and Gould \cite{Gomes1} generalized Theorem \ref{munn} to Ehresmann semigroups and restriction semigroups; Jones \cite{Jones1,Jones3} and Wang \cite{Wang6} generalized Theorem \ref{munn} to $P$-restriction, $P$-Ehresmann, DRC-restriction and DRC semigroups, respectively. On the other hand, Lawson \cite{Lawson2,Lawson3} generalized Theorem \ref{esn} to Ehresmann semigroups and restriction semigroups;
Wang \cite{Wang2,Wang3}  generalized Theorem \ref{esn} to a special class of $P$-restriction and $P$-Ehresmann semigroups, respectively. Recently, Wang \cite{Wang4}
generalized Theorem \ref{esn} to the whole class of DRC semigroups and so  generalized Theorem \ref{esn} to all subclasses of DRC semigroups mentioned in Diagram 1.

Although Wang \cite{Wang4} investigated  all subclasses of DRC semigroups mentioned in Diagram 1 by so-called categorical approach, we observe that in this process  he actually  used some kinds of generalized small categories or generalized  groupoids introduced there instead of the usual small categories or groupoids. More recently,  among other things East and Azeef Muhammed \cite{East1,East2}  generalized Theorem \ref{esn} to the whole class of regular $\circ$-semigroups by using the theory of chain projection ordered groupoids established there rather than the theory of generalized groupoids of any kind. In particular, they proved the following theorem.
\begin{theorem}[\cite{East1}]
The category of regular $\circ$-semigroups together with $\circ$-homomorphisms is isomorphic to the category of chain projection ordered groupoids together with chain projection ordered functors.
\end{theorem}

In the end of their paper \cite{East1}, East and Azeef Muhammed proposed the following natural problem.
\begin{problem}[The Fourth Problem in \cite{East1}] \label{disige}
Can the theory developed in \cite{East1} be applied to other categories of ``projection-based" semigroups such as $P$-Ehresmann semigroups and DRC semigroups?
\end{problem}

In this paper, we solve the above Problem \ref{disige} for DRC-restriction semigroups. We generalize the theory of chain projection ordered groupoids established
by East and Azeef Muhammed in \cite{East1,East2} to that of chain projection ordered categories by which DRC-restriction semigroups are characterized. More specifically,
we prove   the category  of DRC-restriction semigroups together with (2,1,1)-homomorphisms is isomorphic to the category of chain projection ordered categories together with chain projection ordered functors. Some special cases such as $P$-restriction semigroups and generalized regular $\circ$-semigroups are also considered.  Finally, some applications of the main theorem are given.  Our results generalize and enrich the corresponding  results obtained by  East and Azeef Muhammed in \cite{East1,East2}.

The paper is organized as follows. Section 2 gives some fundamental notions and results on ordered categories and groupoids, and Section 3 states some items on two-sided projection algebras. In Section 4, based on the results given in Sections 2 and 3, we establish the theory of chain projection ordered categories and construct the category of chain projection ordered categories together with chain projection ordered functors. Section 5 is devoted to proving that a DRC-restriction semigroup give rises to a chain projection ordered category, and Section 6 proves our category isomorphism theorem. In the final two sections, we consider two applications of our main result. Section 7 uses so-called chain semigroups to construct the free DRC-restriction semigroup over any arbitrary strong two-sided projection algebra  and Section 8 uses projection categories (groupoids)  to reobtain the fundamental representation theorem for DRC-restriction semigroups  appeared firstly in Wang \cite{Wang6}.

\section{Ordered categories}
In this setion, we give some basic results on ordered categories and ordered groupoids.  Assume that $C$ is a nonempty set, $``\circ"$ is a partial binary operation on $C$ and $$\dd: C\rightarrow C, x\mapsto \dd(x),\ \  \rr: C\rightarrow C, x\mapsto \rr(x)$$ are maps. According to \cite{Lawson2},   $(C, \circ, \dd, \rr)$ is called a {\em (small) category}, if the following conditions hold:
\begin{itemize}
\item[(C1)]For all  $x,y\in C$, $x\circ y$ is defined if and only if $\rr(x)=\dd(y),$ in this case, $\dd(x\circ y)=\dd(x)$ and $\rr(x\circ y)=\rr(y)$.
\item[(C2)]For all $x,y,z\in C$,  if $x\circ y$ and $y\circ z$ are defined,  then $(x\circ y)\circ z=x\circ(y\circ z)$.
\item[(C3)] For all $x\in C$, $\dd(x)\circ x$ and $x\circ\rr(x)$ are defined, and $\dd(x)\circ x=x=x\circ \rr(x)$.
\end{itemize}
\begin{remark}\label{fanchoudexingzhi}
Let $(C, \cdot, \dd, \rr)$ be a category and $x\in C$.  Then $\dd(x)\circ x$ and $x\circ\rr(x)$ are defined by {\rm (C3)}, and so $\rr(\dd(x))=\dd(x)$ and $\rr(x)=\dd(\rr(x))$ by {\rm (C1)}. This implies that $\{\dd(x)\mid x\in C\}=\{\rr(x)\mid x\in C\}$.  We call $$P_{C}=\{\dd(x)\mid x\in C\}=\{\rr(x)\mid x\in C\}$$ the {\em set of objects} of $C$. Clearly, $\dd(e)=e=\rr(e)$ for all $e\in P_{C}$.
\end{remark}
Let $(C, \circ, \dd, \rr)$ be a category and $``\leq"$ be a partial order on $C$.  Then $(C, \cdot, \dd, \rr, \leq)$ is called an {\em ordered category} if for all $a,b,c,d\in C$ and $p, q\in P_{C}$, the following statements hold:
\begin{itemize}
\item[(O1)]If $a\leq b$, then $\dd(a)\leq \dd(b)$ and $\rr(a)\leq \rr(b)$.
\item[(O2)]If $a\leq b$, $c\leq d$, $\rr(a)=\dd(c)$, $\rr(b)=\dd(d)$, then $a\circ c\leq b\circ d$.
\item[(O3)]If $p\leq \dd(a)$, then exists a unique element $u\in C$ such  that $u\leq a$ and $\dd(u)=p$.
\item[(O4)]If $q\leq \rr(a)$, then exists a unique element $v\in C$ such that $v\leq a$ and $\rr(v)=q$.
\end{itemize}
In this case, the elements $u,v$ in (O3) and (O4) are denoted by $u=_p\hspace{-1.5mm}{\downharpoonleft}a$ and $v=a{\downharpoonright}_{q}$,  and call them the
{\em left restriction} and {\em right restriction} on $p$ and $q$ of $a$, respectively.

\begin{lemma}\label{b} Let $(C, \cdot, \dd, \rr, \leq)$ be an ordered category, $a,b\in C$ and $p,q\in P_{C}$.
\begin{itemize}
\item[(1)] If $p\leq \dd(a)$, then ${_p}{\downharpoonleft}a\leq a$, $\dd({_p}{\downharpoonleft} a)=p$ and $\rr({_p}{\downharpoonleft}a)\leq \rr(a)$.
\item[(2)] If $q\leq \rr(a)$, then $a{\downharpoonright}{_q}\leq a$, $\rr(a{\downharpoonright}{_q})=q$ and $\dd(a{\downharpoonright}{_q})\leq \dd(a)$.
\item[(3)] If $a\leq b$, then $a={_{\dd(a)}}{\downharpoonleft} b$. In particular, ${_{\dd(a)}}{\downharpoonleft} a=a$.
\item[(4)] If $a\leq b$, then $a=b{\downharpoonright}{_{\rr(a)}}$. In particular, $a{\downharpoonright}{_{\rr(a)}}=a$.
\item[(5)]  If $p\leq \dd(a)$, then ${_p}{\downharpoonleft} a=a{\downharpoonright}{_{\rr({_p}{\downharpoonleft} a)}}$.
\item[(6)] If $q\leq \rr(a)$, then $a{\downharpoonright}{_q}={_{\dd(a{\downharpoonright}{_q})}}{\downharpoonleft} a$.
\item[(7)] If $p\leq \dd(a)$, $\rr(a)=\dd(b)$, then ${_p}{\downharpoonleft}(a\circ b)={_p}{\downharpoonleft}a\circ {_{\rr({_p}{\downharpoonleft}a)}}{\downharpoonleft} b$.
\item[(8)] If $q\leq \rr(b)$, $\rr(a)=\dd(b)$, then $(a\circ b){\downharpoonright}{_q}=a{\downharpoonright}{_{\dd(b{\downharpoonright}{_q})}}\circ b{\downharpoonright}{_q}$.
\item[(9)] If $p\leq q\leq \dd(a)$, then ${_p}{\downharpoonleft}({_q}{\downharpoonleft} a)={_p}{\downharpoonleft} a$.
 \item[(10)] If $p\leq q\leq \rr(a)$, then $(a{\downharpoonright}{_q}){\downharpoonright}{_p}=a{\downharpoonright}{_p}$.
\item[(11)]If $p,q\in P_{C}$ and $p\leq q$, then ${_p}{\downharpoonleft} q=p$.
\item[(12)]If $p,q\in P_{C}$ and $p\leq q$, then $q{\downharpoonright}{_p}=p.$
\end{itemize}
\end{lemma}
\begin{proof}
We only need to show items (1), (3), (5), (7), (9) and (11) by symmetry.

(1) By the definition of ${_p}{\downharpoonleft} a$ and (O3), we have ${_p}{\downharpoonleft}a\leq a$ and $\dd({_p}{\downharpoonleft} a)=p$, and so $\rr({_p}{\downharpoonleft}a)\leq \rr(a)$ by (O1).

(3) The fact $a\leq b$ and (O1) give $\dd(a)\leq \dd(b)$, and so $b$ has a unique left restriction on $\dd(a)$ by (O3). Since $a\leq b$ and $\dd(a)=\dd(a)$, we have   $a={_{\dd(a)}}{\downharpoonleft} b$.

(5) (O3) gives ${_p}{\downharpoonleft}a\leq a$ and so we have ${_p}{\downharpoonleft} a=a{\downharpoonright}{_{\rr({_p}{\downharpoonleft} a)}}$ by (4) in the present lemma.

(7) By (O3) and (O1), we have ${_p}{\downharpoonleft}a\leq a$ and $\rr({_p}{\downharpoonleft}a)\leq \rr(a)=\dd(b)$.  By the definitions of ${_p}{\downharpoonleft}(a\circ b) $ and ${_{\rr({_p}{\downharpoonleft}a)}}{\downharpoonleft}b$ and (O3), $${_p}{\downharpoonleft}(a\circ b)\leq  a\circ b,\, {_{\rr({_p}{\downharpoonleft}a)}}{\downharpoonleft}b\leq b,\, \rr(_{p}{\downharpoonleft}a)=\dd({_{\rr({_p}{\downharpoonleft}a)}}{\downharpoonleft}b),\dd({_p}{\downharpoonleft}(a\circ b))=p,$$ this together with the fact $\rr(a)=\dd(b)$ and (O2) implies that ${_p}{\downharpoonleft}a  \circ  {_{\rr({_p}{\downharpoonleft}a)}}{\downharpoonleft}b  \leq a\circ b $. On the other hand, we have $\dd({_p}{\downharpoonleft}a  \circ  {_{\rr({_p}{\downharpoonleft}a)}}{\downharpoonleft}b )=\dd({_p}{\downharpoonleft}a)=p$ by (C1) and (O3). By (3) in the present lemma, we obtain  ${_p}{\downharpoonleft}a\circ {_{\rr({_p}{\downharpoonleft}a)}}{\downharpoonleft} b={_p}{\downharpoonleft}(a\circ b)$.

(9) Item (O3) gives that ${_p}{\downharpoonleft} ({_q}{\downharpoonleft} a)\leq {_q}{\downharpoonleft} a \leq a$ and $\dd({_p}{\downharpoonleft} ({_q}{\downharpoonleft} a))=p$, and so  ${_p}{\downharpoonleft}({_q}{\downharpoonleft} a)={_p}{\downharpoonleft} a$ by (3) in the present lemma.

(11) Since $p\leq q$ and $\dd(p)=p$ by Remark \ref{fanchoudexingzhi}, we have $p= {_p}{\downharpoonleft}q$ by (3) in the present lemma.
\end{proof}

Let $(C, \cdot, \dd, \rr, \leq)$ be an ordered category and $p\in P_{C}$, Denote $$p^{\downarrow}=\{q\in P_{C}\mid q\leq p\}.$$ Let $a\in C$. By (1)  and (2) in Lemma \ref{b}, we have $\rr({_p}{\downharpoonleft}a)\leq \rr(a)$ and $\dd(a{\downharpoonright}_q)\leq \dd(a)$. Thus we can define the maps $\nu_{a}$ and $\mu_{a}$ as follows:
\begin{equation}\label{c}
\nu_{a}:\dd(a)^{\downarrow}\rightarrow \rr(a)^{\downarrow},\, p\mapsto p\nu_{a}=\rr({_p}{\downharpoonleft}a);\,\,\,\,\, \mu_{a}: \rr(a)^{\downarrow}\rightarrow \dd(a)^{\downarrow},\, q\mapsto q\mu_{a}=\dd(a{\downharpoonright}_q).
\end{equation}
By using the maps $\nu_a$ and $\mu_a$ ($a\in S$), we can restate (7) and (8) in Lemma \ref{b} as follows:  If $(C, \cdot, \dd, \rr, \leq)$ is an ordered category,  $a,b\in C$ and $ p,q\in P_{C}$ with $\rr(a)=\dd(b)$, $p\leq \dd(a)$ and $q\leq \rr(b)$.  Then
\begin{equation}\label{e}
{_p}{\downharpoonleft}(a\circ b)={_p}{\downharpoonleft}a \circ {_{p\nu_a}}\hspace{-0.5mm}{\downharpoonleft}b,\quad (a\circ b){\downharpoonright}_{q}= a{\downharpoonright}_{q\mu_{b}}\circ b{\downharpoonright}_{q}.
\end{equation}

\begin{lemma}\label{f} Let $(C, \cdot, \dd, \rr, \leq)$ be an ordered category and $a\in C$. Then $\nu_{a}$ and $\mu_{a}$ are mutually inverse bijections. In particular, if $e\in P_C$, then $\nu_e=\mu_e={\rm id}_{e^\downarrow}$.
\end{lemma}
\begin{proof}Let $p\leq \dd(a)$ and denote $q=p\nu_{a}=\rr(_{p}{\downharpoonleft}a)$. Then by (5) and (1) in Lemma \ref{b},
$$p\nu_{a}\mu_{a}=q \mu_{a}=\dd(a{\downharpoonright}{_q})=\dd({_p}{\downharpoonleft} a)=p.$$
Thus $\nu_{a}\mu_{a}={\rm id}_{\dd(a)^{\downarrow}}$. Dually,  we can prove that $\mu_{a}\nu_{a}={\rm id}_{\rr(a)^{\downarrow}}$. If $e\in P_C$, then by Remark \ref{fanchoudexingzhi}, we have $\dd(e)=\rr(e)=e$. Let $p\in \dd(e)^\downarrow=e^\downarrow$. Then $p\in P_C$ and $p\leq e$, and so $p\nu_{e}=\rr(_{p}{\downharpoonleft}e)=\rr(p)=p$ by Lemma \ref{b} (11) and Remark \ref{fanchoudexingzhi} again. This shows that  $\nu_e={\rm id}_{e^\downarrow}$. Moreover, $\mu_e=\nu_e^{-1}={\rm id}_{e^\downarrow}^{-1}={\rm id}_{e^\downarrow}$.
\end{proof}
\begin{lemma}\label{apdxn} Let $(C, \cdot, \dd, \rr, \leq)$ be an ordered category, $a\in C$ and $e,f\in P_C$ with $ e\leq \dd(a), f\leq \rr(a)$. Then $\nu_{_{e\downharpoonleft{a}}}=\nu_a|_{e^\downarrow}$ and $\mu_{a\downharpoonright_f}=\mu_a|_{f^\downarrow}$.
\end{lemma}
\begin{proof}
It is easy to see that  $\nu_{_{e}{\downharpoonleft}a}$ and $\nu_{a}{\mid}{_{e^{\downarrow}}}$ have the same domain $e^{\downarrow}$ by Lemma \ref{b} (1). Let $t\leq e$. By (\ref{c}) and Lemma \ref{b} (9), we have
$t\nu_{_{e}{\downharpoonleft}a}=\rr({_t}{\downharpoonleft}({_e}{\downharpoonleft} a))=\rr(_{t}{\downharpoonleft}a)=t\nu_{a}.$ Dually, $\mu_{a{\downharpoonleft}{_f}}=\mu_{a}{\mid}{_{f^{\downarrow}}}.$
\end{proof}

\begin{lemma}\label{ee}Let $(C, \cdot, \dd, \rr, \leq)$ be an ordered category, $a,b\in C$ and $\rr(a)=\dd(b)$. Then $$\nu_{a\circ b}=\nu_{a}\nu_{b},~\mu_{a\circ b}=\mu_{b}\mu_{a}.$$
\end{lemma}
\begin{proof}Let $p=\dd(a),q=\rr(a)=\dd(b),r=\rr(b)$. Then $\dd(a\circ b)=\dd(a)=p$ and $\rr(a\circ b)=\rr(b)=r,$ and so
$$\nu_{a}: p^{\downarrow}\rightarrow q^{\downarrow},\, \nu_{b}: q^{\downarrow}\rightarrow r^{\downarrow},\,\,\,  \nu_{a\circ b}: p^{\downarrow}\rightarrow r^{\downarrow},\nu_{a}\nu_{b}: p^{\downarrow}\rightarrow r^{\downarrow}$$
Let $e\in p^\downarrow$ (i.e. $s\leq p$) and denote $t=s\nu_{a}=\rr({_s}{\downharpoonleft} a)$. By Lemma \ref{b} (7), we have
$$s\nu_{a\circ b}=\rr({_s}{\downharpoonleft}(a\circ b))=\rr({_s}{\downharpoonleft}a \circ {_t}{\downharpoonleft} b)=\rr({_t}{\downharpoonleft} b)=t\nu_{b}=s\nu_{a}\nu_{b}.$$ Thus $\nu_{a\circ b}=\nu_{a}\nu_{b}$. Dually,  we can prove that $\mu_{a\circ b}=\mu_{b}\mu_{a}$.
\end{proof}
A category $(C, \circ, \dd, \rr)$ is called a {\em groupoid} if for all $a\in C$, there exists $a'\in C$ such that $a\circ a'=\dd(a)$ and $ a'\circ a=\rr(a)$. In this case, for $a\in C$, we have  $\rr(a')=\rr(a\circ a')=\rr(\dd(a))=\dd(a)$. Dual argument gives $\dd(a')=\rr(a)$.  The above $a'$ is unique. In fact, if $a{''}\in C$  and $a\circ a{''}=\dd(a)$ and $ a{''}\circ a=\rr(a)$, then $\dd(a)=\rr(a{''})$, and so $$a'=\dd(a')\circ a'=\rr(a)\circ a'=(a{''}\circ a) \circ a'=a{''}\circ \dd(a)=a{''}\circ \rr(a{''})=a{''}.$$
Denote the unique element $a'$ by $a^{-1}$.  An ordered category $(C, \circ, \dd, \rr, \leq)$ is called an {\em ordered groupoid} if $(C, \circ, \dd, \rr)$ is a groupoid.
\begin{lemma}\label{qunpei}Let $(C, \circ, \dd, \rr, \leq)$ be an ordered groupoid and $a, b\in C$.
\begin{itemize}
\item[(1)] $a\circ a^{-1}=\dd(a), a^{-1}\circ a=\rr(a), \dd(a)=\rr(a^{-1}), \rr(a)=\dd(a^{-1})$.
\item[(2)]If $a\leq b$, then $a^{-1}\leq b^{-1}$.
\item[(3)]If $p, q\in P_C$ and $p\leq \dd(a), q\leq \rr(a)$, then $a{\downharpoonright}_q=(_q{\downharpoonleft}a^{-1})^{-1}$, $_p{\downharpoonleft}a=(a^{-1}{\downharpoonright}_p)^{-1}$.
\item[(4)] $\mu_a=\nu_{a^{-1}}$.
\end{itemize}
\end{lemma}
\begin{proof}
(1) This follows from the statements before Lemma \ref{qunpei}.

(2) Let $a\leq b$. Then $a= {_{\dd(a)}}{\downharpoonleft}b$ by Lemma \ref{b} (3). Item (O1) gives $\dd(a)\leq \dd(b)$,  and so $\dd(a)={_{\dd(a)}}{\downharpoonleft}\dd(b)$ by Lemma \ref{b} (11).  Thus $$\rr(a^{-1})=\dd(a)={_{\dd(a)}}{\downharpoonleft}\dd(b)={_{\dd(a)}}{\downharpoonleft}(b\circ b^{-1})=
{_{\dd(a)}}{\downharpoonleft}b \circ _{\rr({_{\dd(a)}}{\downharpoonleft}b)}{\downharpoonleft} b^{-1}=a\circ{_{\rr(a)}}{\downharpoonleft}b^{-1}$$ by item (1) in the present lemma and Lemma \ref{b} (7). This implies that
$$a^{-1}=a^{-1} \circ \,\rr(a^{-1})=a^{-1}\,\circ\, a \circ\,{_{\rr(a)}}{\downharpoonleft}b^{-1}=\rr(a)\,\circ {_{\rr(a)}}{\downharpoonleft}b^{-1} $$$$=\dd({_{\rr(a)}}{\downharpoonleft}b^{-1})\,\circ {_{\rr(a)}}{\downharpoonleft}b^{-1}={_{\rr(a)}}{\downharpoonleft}b^{-1}\leq b^{-1}$$
by Lemma \ref{b} (1) and (C3).

(3) Since $q\leq \rr(a)=\dd(a^{-1})$, we have ${_q}{\downharpoonleft}a^{-1}\leq a^{-1}$, and so $({_q}{\downharpoonleft}a^{-1})^{-1}\leq a$ by item  (2) in the present lemma. By Lemma \ref{b} (1), we get $\rr(({_q}{\downharpoonleft}a^{-1})^{-1})=\dd({_q}{\downharpoonleft}a^{-1})=q$. By Lemma \ref{b} (4), we have $({_q}{\downharpoonleft}a^{-1})^{-1}=a{\downharpoonright}{_q}.$ The other identity can be proved dually.

(4) By (\ref{c}), we have $$\nu_{a^{-1}}:\dd(a^{-1})^{\downarrow}\rightarrow \rr(a^{-1})^{\downarrow},\, q\mapsto\rr({_q}{\downharpoonleft}a^{-1}).$$ By (1) and (3) in the present lemma,  $$\nu_{a^{-1}}:\rr(a)^{\downarrow}\rightarrow \dd(a)^{\downarrow},\, q\mapsto\dd(({_q}{\downharpoonleft}a^{-1})^{-1})=\dd(a{\downharpoonright}{_q}).$$ This implies that  $\nu_{a^{-1}}=\mu_{a}$ by (\ref{c}) again.
\end{proof}

\section{Two-sided projection algebras}
In this section, we shall give some basic notions and facts on two-sided projection algebras introduced by Jones in \cite{Jones1}.
Let $P$ be a non-empty set and``$\times$"  be a binary operation on $P$. Then $(P,\times)$ is called a {\em left projection algebra} if the following axioms hold:
\begin{itemize}
\item[(L1)]$e\times e=e$.
\item[(L2)]$(e\times f)\times e=e\times f$.
\item[(L3)]$e\times(f\times g)=(e\times f)\times (f\times g)$.
\item[(L4)]$(e\times f)\times g=(e\times f)\times(e\times g)$.
\end{itemize}
Dually, let $P$ be a non-empty set and ``$\star$" be a binary operation on $P$. Then $(P,\star)$ is called a {\em right projection algebra} if the following axioms hold:
\begin{itemize}
\item[(R1)]$e\star e=e.$
\item[(R2)]$(f\star e)=e\star(f\star e).$
\item[(R3)]$(g\star f)\star e=(g\star f)\star (f\star e).$
\item[(R4)]$g\star (f\star e)=(g\star e)\star(f\star e).$
\end{itemize}
Let $P$ be a non-empty set and ``$\times$" and ``$\star$" be binary operations on $P$. Then $(P,\times,\star)$ is called a {\em two-sided projection algebra}, or   {\em  projection algebra} for simplicity, if the following conditions hold:
\begin{itemize}
\item[(P1)]$(P,\times)$ forms a left projection algebra.
\item[(P2)]$(P,\star)$ forms a right projection algebra.
\item[(P3)]For all $e,f,g\in P$, we have $(e\star(f\times g))\star g=(e\star f)\star g$ and $g\times((g\star f)\times e)=g\times(f\times e)$.
\item[(P4)]For all $e,f\in P$, we have $(e\times f)\star e=e\times f$ and $f\times(e\star f)=e\star f$.
\end{itemize}
Let $(P,\times,\star)$ be a projection algebra. Define a relation ``$\leq_{P}$" on $P$ as follows: For all $e,f\in P$,
$$e\leq_{P} f\Longleftrightarrow e=f\times e.$$
Then $``\leq_{P}"$ is a partial order by \cite[Lemma 5.2]{Jones1}. Recall form \cite{Wang6} that a projection algebra $(P,\times,\star)$
is {\em strong} if the following axioms hold:
\begin{itemize}
\item[(SP1)]$e\times (((e\times f)\times g)\star f)=(e\times f)\times g.$
\item[(SP2)]$(f\times (g\star (f\star e)))\star e=g\star (f\star e).$
\end{itemize}
It is easy to see that (SP1) (respectively, (SP2)) is equivalent to the following condition: For all  $e,f, g\in P$,
\begin{equation}\label{wang2}
g\leq_P e\times f\Longrightarrow e\times (g\star f)=g\,\,\,
(\mbox{respectively, }g\leq_P f\star e\Longrightarrow (f\times g) \star e=g).
\end{equation}
A strong projection algebra $(P,\times,\star)$ is called
\begin{itemize}
\item {\em symmetric} if $e\times f=f\star e$ for all $e,f\in P$, and
\item {\em commutative} if $e\times f=f\times  e$ for all $e,f\in P$.
\end{itemize}
\begin{lemma}\label{jiben1}Let $(P,\times, \star)$ be a projection algebra and  $e,f,g\in P$.
\begin{itemize}
\item[(1)]  $e\leq_P f $ if and only if $e=e\times f=f\times e$. Moreover, $e\times f\leq_P e$.
\item[(2)] $e\leq_P f $if and only if $ e=e\star f=f\star e$. Moreover, $e\star f\leq_P f$.
\item[(3)]   $(e\times f)\star f=e\star f$ and $e\times (e\star f)=e\times f$.
\item[(4)] $(e\times f)\times (e\star f)=e\times f$ and $(e\times f)\star (e\star f)=e\star f$.
\item[(5)]If $e\leq_{P} f$, then $ g\times e\leq_{P}g\times f$ and $e\star g\leq_{P}f\star g$.
\end{itemize}
\end{lemma}
\begin{proof} Items (1)--(4) follow from \cite[Lemma 3.2]{Wang3}, and (5) follows from \cite[Lemma 5.2]{Wang6}, respectively.
\end{proof}
Let $(P,\times, \star)$ be a projection algebra. For each $e\in P$, define $$ e^{\downarrow}=\{x\in P\mid x\leq_P e\}.$$ By (1) and (2) in Lemma \ref{jiben1}, we have
\begin{equation}\label{xu3}
 e^{\downarrow}=\{e\times p\mid p\in P\}=\{p\star e\mid p\in P\}.
\end{equation}
Let $p\in P$. Define  $\theta_{p}$ and $\delta_{p}$ as follows:
\begin{equation}\label{j}
\theta_{p}: P\rightarrow P,\, q\mapsto q\theta_{p}=q\star p,\,\, \delta_{p}: P \rightarrow P,\, q\mapsto q\delta_{p}= p\times q.
 \end{equation}
Again by (1) and (2) in Lemma \ref{jiben1}, for all $p,q\in P$ we have
\begin{equation}\label{l}
p\leq_{P}q\Longleftrightarrow p\theta_{q}=q\theta_{p}=p\Longleftrightarrow p\delta_{q}=q\delta_{p}=p.
\end{equation}
Moreover, (\ref{xu3}) gives that
\begin{equation}\label{m}
\begin{array}{cc}
\im\theta_{p}=\{q\star p\mid q\in P\}=p^{\downarrow}=\{p\times q\mid q\in P\}=\im\delta_{p}.
\end{array}
\end{equation}
\begin{lemma}\label{o} Let $(P,\times,\star)$ be a projection algebra and $p\leq_{P}q$. Then $$\theta_{p}=\theta_{p}\theta_{q}=\theta_{q}\theta_{p}=\theta_{p}\delta_{q},\,\delta_{p}=\delta_{p}\delta_{q}=\delta_{q}\delta_{p}=\delta_{p}\theta_{q}.$$
\end{lemma}
\begin{proof}
Let $t\in P$. By Lemma \ref{jiben1} (2), we have $t\star p\leq_P p$ and $(t\star p)\star p=t\star p$. Since $p\leq_{P}q$, it follows that $p\star q=p$  and $q\times(t\star p)=t\star p$ by (2) and (1) in Lemma \ref{jiben1}.
Firstly, by (R3) we have
$$t\theta_{p}\theta_{q}=(t\star p)\star q\overset{({\rm R}3)}{=}(t\star p)\star(p\star q)=(t\star p)\star p=t\star p=t\theta_{p}.$$
Secondly, by (R4) it follows that
$$t\theta_{q}\theta_{p}=(t\star q)\star p=(t\star q)\star(p\star q)\overset{({\rm R}4)}{=}t\star (p\star q)=t\star p=t\theta_{p}.$$
Finally, $t\theta_{p}\delta_{q}=q\times(t\star p)=t\star p=t\theta_{p}$.
Thus $\theta_{p}=\theta_{q}\theta_{p}=\theta_{p}\theta_{q}=\theta_{p}\delta_{q}$. Dually, we have $\delta_{p}=\delta_{p}\delta_{q}=\delta_{q}\delta_{p}=\delta_{p}\theta_{q}$.
\end{proof}

Let $(P,\times,\star)$ be a projection algebra. Define a relation ${{\mathcal F}_{P}}$ on $P$ as follows:  For all $p,q\in P$,
\begin{equation}\label{desfxv}
p {{\mathcal F}_{P}}\, q\Longleftrightarrow p=q\delta_{p},q=p\theta_{q}.
\end{equation}
Let $p,q\in P$. Then $q\delta_{p}=p\times q$ and $p\theta_{q}=p\star q$ by the definitions of  $\delta_{p}$ and $\theta_{q}$. Thus
\begin{equation}\label{s}
p {{\mathcal F}_{P}}\, q\,\mbox{if and only if }\, p\times q=p \mbox { and } p\star q=q.
\end{equation}
\begin{lemma}\label{p}Let $(P,\times,\star)$ be a projection algebra and $p,q\in P$. Then $q\delta_{p}\leq_P p, p\theta_{q} \leq_P q$ and $q\delta_{p}$ $\mathcal{F}_{P}\, p\theta_{q}$.
\end{lemma}
\begin{proof}Firstly, we have $q\delta_{p}\leq_P p$ and $ p\theta_{q} \leq_P q$ by the facts that $q\delta_{p}=p\times q, p\theta_{q}=p\star q$ and (1) and (2) in Lemma \ref{jiben1}. On the other hand,
$$(q\delta_{p}) \times (p\theta_{q})=(p\times q)\times(p\star q)=p\times q=q\delta_{p},$$ $$(q\delta_{p}) \star (p\theta_{q})=(p\times q)\star(p\star q)=p\star q=p\theta_{q}$$ by Lemma \ref{jiben1} (4). Thus $q\delta_{p}\,\mathcal{F}_{P} \,p\theta_{q}$.
\end{proof}

\begin{lemma}\label{y}Let $(P,\times,\star)$ be a  projection algebra and $p, q, r\in P$.
\begin{itemize}
\item[(1)] If $p\leq_{P} q$ and $q=q\times r$, then $p=p\times r.$
\item[(2)]If $p=p\times q$ and $q\leq_{P}r$, then $p=p\times r.$
\end{itemize}
\end{lemma}
\begin{proof}
By the facts that $p\leq_{P} q, q=q\times r$  and Lemma \ref{jiben1} (1), we obtain $p=p\times q=q\times p$ and $q=q\times r$. This together with (L4) gives that
$$p=p\times q=(q\times p)\times (q\times r)\overset{({\rm L}4)}{=}(q\times p)\times r=p\times r.$$ This proves (1). On the other hand, by the facts $p=p\times q, q\leq_{P}r$ and Lemma \ref{jiben1} (1), we obtain $p=p\times q, q=q\times r=r\times q$. This together with (L3), (L4) and (L2) implies that
$$p=p\times q=p\times (r\times q)\overset{({\rm L}3)}{=}(p\times r)\times (r\times q)$$
$$=(p\times r)\times q\overset{({\rm L}4)}{=}(p\times r)\times (p\times q)=(p\times r)\times p\overset{({\rm L}2)}{=}p\times r.$$
This shows (2).
\end{proof}
\begin{lemma}\label{zjiy}
Let $(P,\times,\star)$ be a projection algebra, $p_{1},p_{2},\ldots,p_{k},q\in P$ and $p_{1}\mathcal{F}_{P}\,p_{2}\mathcal{F}_{P}\ldots \mathcal{F}_{P}\,p_{k}$. Then $$p_{k}\delta_{p_{k-1}}\delta_{p_{k-2}}\ldots \delta_{p_{1}}=p_{1},\,\, p_{1}\theta_{p_{2}}\theta_{p_{3}}\ldots \theta_{p_{k}}=p_{k}.$$
\end{lemma}
\begin{proof}
In view of the fact that $p_{1}\mathcal{F}_{P}p_{2}\mathcal{F}_{P} \ldots \mathcal{F}_{P}p_{k}$, (\ref{desfxv}) and (\ref{s}), we have $$p_{k}\delta_{p_{k-1}}\delta_{p_{k-2}}\ldots \delta_{p_{1}}=(p_{k-1}\times p_{k})\delta_{p_{k-2}}\ldots \delta_{p_{1}}$$$$=p_{k-1}\delta_{p_{k-2}}\ldots \delta_{p_{1}}=(p_{k-2}\times p_{k-1})\delta_{p_{k-3}}\ldots \delta_{p_{1}}$$$$=p_{k-2}\delta_{p_{k-3}}\ldots \delta_{p_{1}}=\cdots\cdots=p_{2}\delta_{p_{1}}=p_{1}\times p_{2}=p_{1}.$$ The other identity can be proved dually.
\end{proof}
Let $(P, \times, \star)$ and $(P', \times, \star)$ be two projection algebras. A map $\psi: P\rightarrow P'$ is called a {\em homomorphism} if
$$(y{\delta_x})\psi=(x\times y)\psi=(x\psi)\times (y\psi)=(y\psi){\delta_{x\psi}},\,\,\, (x\theta_y)\psi=(x\star y)\psi=(x\psi)\star (y\psi)=(x\psi)\theta_{y\psi}$$ for all $x, y\in P$.
\begin{lemma}\label{pin}Let $(P, \times, \star)$ and $(P', \times, \star)$ be two projection algebras, $\psi: P\rightarrow P{'}$ be a homomorphism  and  $p,q\in P$.
\begin{itemize}
\item[(1)] If $p\leq_P q$, then $p\psi\leq_{P'} q\psi$.
\item[(2)]If $p ~{\mathcal F}_P~ q$, then $p\psi~ \mathcal{F}_{P'}~  q\psi$.
\end{itemize}
\end{lemma}
\begin{proof}
Let $p\leq_P q$. Then $q\times p=p$. Since $\psi$ is homomorphism, we have $(q\psi)\times (p\psi)=p\psi$, and so $p\psi\leq_{P'} q\psi$. This gives (1), and (2) can be proved similarly.
\end{proof}
\begin{lemma}\label{jj}Let $(P,\times,\star)$ be a strong projection algebra and $p,s\in P$. Then $\delta_{p\theta_{s}}=\delta_{s}\delta_{p}\theta_{s}$ and $ \theta_{p\delta_{s}}=\theta_{s}\theta_{p}\delta_{s}$.
\end{lemma}
\begin{proof}Let $t\in P$. Then $(p\star s)\times t\leq_P p\star s$ by Lemma \ref{jiben1} (1),  and so $(p\star s)\times t=(p\times((p\star s)\times t))\star s$ by (\ref{wang2}).  Now,  (P3) gives that
$$t\delta_{p\theta_{s}}=(p\star s)\times t=(p\times((p\star s)\times t))\star s\overset{({\rm P}3)}{=}(p\times(s\times t))\star s=t\delta_{s}\delta_{p}\theta_{s}.$$
Dually, we have $\theta_{p\delta_{s}}=\theta_{s}\theta_{p}\delta_{s}$.
\end{proof}
Using  Lemma \ref{jj} and mathematical induction, we have the following corollary.
\begin{coro}\label{diav} Let $(P,\times,\star)$ be a strong projection algebra and  $p_{1},p_{2},\ldots,p_{k},q\in P$.
Then
$$\delta_{q\theta_{p_{1}}\theta_{p_{2}}\ldots\theta_{p_{k}}}=\delta_{p_{k}}\ldots\delta_{p_{1}}\delta_{q}\theta_{p_{1}}\ldots
\theta_{p_{k}},$$$$ \theta_{q\delta_{p_{k}}\delta_{p_{k-1}}\ldots\delta_{p_{1}}}=\theta_{p_{1}}\ldots \theta_{p_{k}}\theta_{q}\delta_{p_{k}}\ldots \delta_{p_{1}}.$$
\end{coro}

\section{Chain projection ordered categories}
In this section, we develop a theory of chain projection ordered categories, which generalizes the theory of chain projection ordered groupoids provided by  East and  Azeef Muhammed in \cite{East1} recently.
\begin{defn}\label{dzxc} Let $(P,\times,\star)$ be a strong projection algebra and  $k$ be a positive integer. Assume that $p_1, p_2, \ldots, p_k\in P$ and $p_{1}{\mathcal{F}_{P}}\, p_{2}\mathcal{F}_{P} \ldots \mathcal{F}_{P}\, p_{k}$. Then we call ${\mathfrak p}=(p_{1},p_{2},\ldots,p_{k})$ a {\em $P$-path} in $P$ with length $k$ from $p_1$ to $p_k$. Denote the set of all $P$-paths in $P$ by $\mathscr{P}(P)$. For all $p\in P$,  we identify $P$-path $(p)$ with length $1$ to $p$. Under this assumption, $P$ is a subset of $\mathscr{P}(P)$.
\end{defn}
Let $(P,\times,\star)$ be a strong projection algebra. For $${\mathfrak p}=(p_{1},p_{2},\ldots,p_{k}), \mathfrak{q}=(q_{1},q_{2},\ldots,q_{l})\in \mathscr{P}(P),$$ we define $\dd({\mathfrak p})=p_1, \rr({\mathfrak p})=p_k$,  and if $\rr({\mathfrak p})=p_{k}=q_{1}=\dd({\mathfrak q})$, we let
$${\mathfrak p}\circ {\mathfrak q}=(p_{1},p_{2},\ldots,p_{k}=q_{1},q_{2},\ldots,q_{l}).$$
It is easy to see that $(\mathscr{P}(P), \circ, \dd, \rr)$ is a small category, $P_{\mathscr{P}(P)}=P$ and $\dd({\mathfrak p})=p=\rr({\mathfrak p})$ for all $p\in P$. We call this category the {\em path category} of $(P,\times,\star)$. Let $\mathfrak{p}=(p_{1},p_{2},\ldots,p_{k})\in {\mathscr{P}(P)}$, $q\in P$ and $q\leq_P \dd(\mathfrak{p})=p_{1}$. Denote
\begin{equation}\label{u}
{_q}{\downharpoonleft}\mathfrak{p}=(q_{1},q_{2},\ldots,q_{k}), \mbox{ where } q_{1}=q, q_{i}=q_{i-1}\theta_{p_{i}}=q_{i-1}\star p_{i}, i=2,3,\ldots, k,
\end{equation}
and call it the {\em left restriction} of $\mathfrak{p}$ on $q$. In this case, observe that $q\leq_P p_{1}$, it follows that $q_1=q=q\star p_1=q\theta_{p_{1}}$ by Lemma
\ref{jiben1} (2). Combining (\ref{u}) and Lemma \ref{jiben1} (2), we have
\begin{equation}\label{v}
q_{i}=q\theta_{p_{1}}\theta_{p_{2}}\ldots\theta_{p_{i}},\,\, q_i\leq_P p_i,\, i=1,2,\ldots, k.
\end{equation}
Dually, let $\mathfrak{p}=(p_{1},p_{2},\ldots,p_{k})\in {\mathscr{P}(P)}$, $q\in P$ and $q\leq_P \rr(\mathfrak{p})=p_{k}$. Denote
\begin{equation}\label{mft}
\mathfrak{p}{\downharpoonright}_{q}=(q_{1},q_{2},\ldots,q_{k}), \mbox{ where } q_{k}=q, q_{i}=q_{i+1}\delta_{p_{i}}=p_{i}\times q_{i+1}, i=k-1, k-2, \ldots, 1
\end{equation}
and call it the {\em right restriction} of $\mathfrak{p}$ on $q$. In this case, observe that $q\leq_P p_k$, it follows that $q_k=q=p_k\times q=q\delta_{p_{k}}$ by
Lemma \ref{jiben1} (1).  Combining (\ref{mft}) and Lemma \ref{jiben1} (1),  we have
\begin{equation}\label{win}
q_{i}=q\delta_{p_{k}}\ldots\delta_{p_{i}},\,\, q_i\leq_P p_i, i=k,\, k-1, \ldots, 1.
\end{equation}
\begin{lemma}\label{ccts}Let $(P,\times,\star)$ be a strong projection algebra, $q, r\in P$ and  $\mathfrak{p}=(p_{1},p_{2},\ldots,p_{k})\in \mathscr{P}(P)$.
\begin{itemize}
\item[(1)] If $q\leq_{P} \dd(\mathfrak{p})$, then $_{q}{\downharpoonleft}\mathfrak{p}\in \mathscr{P}(P)$, $\dd(_{q}{\downharpoonleft} \mathfrak{p})=q$ and $\rr(_{q}{\downharpoonleft} \mathfrak{p})\leq_{P} \rr(\mathfrak{p})$.
\item[(2)]If $q\leq_{P} \rr(\mathfrak{p})$, then $\mathfrak{p}{\downharpoonright}{_q}\in \mathscr{P}(P)$, $\rr(\mathfrak{p}{\downharpoonright}{_q})=q$ and $\dd(\mathfrak{p}{\downharpoonright}{_q})\leq_{P} \dd(\mathfrak{p})$.
\item[(3)] If $q=\dd(\mathfrak{p})$, then $_{q}{\downharpoonleft} \mathfrak{p}=\mathfrak{p}$.
\item[(4)] If $q=\rr(\mathfrak{p})$, then $ \mathfrak{p}{\downharpoonright}_{q}=\mathfrak{p}$.
\item[(5)] If $r \leq_{P} q\leq_{P} \dd(\mathfrak{p})$, then $ {_r}{\downharpoonright} ({_q}{\downharpoonleft} \mathfrak{p}) ={_r}{\downharpoonleft} \mathfrak{p}$.
\item[(6)] If $r \leq_{P} q\leq_{P} \rr(\mathfrak{p})$, then $ (\mathfrak{p}{\downharpoonright}{_q}){\downharpoonright}{_r}=\mathfrak{p}{\downharpoonright}{_r}$.
\item[(7)] If $r \leq_{P}  \dd(\mathfrak{p})$, then ${_r}{\downharpoonleft} \mathfrak{p}=\mathfrak{p}{\downharpoonright}_{\rr({_r}{\downharpoonleft}\mathfrak{p})}$.
\item[(8)] If $r \leq_{P}  \rr(\mathfrak{p})$, then $\mathfrak{p}{\downharpoonright}{_r}={_{\dd(\mathfrak{p}{\downharpoonright}{_r})}}{\downharpoonleft}\mathfrak{p}$.
\end{itemize}
\end{lemma}
\begin{proof} We only need to show items (1), (3), (5) and (7) by symmetry.

(1) Let ${_q}{\downharpoonleft}\mathfrak{p}=(q_{1},q_{2},\ldots,q_{k}),$  where $$q_{1}=q, q_{i}=q_{i-1}\theta_{p_{i}}=q_{i-1}\star p_{i}, i=2,3,\ldots, k.$$
Since $\mathfrak{p}\in \mathscr{P}(P)$, we have $p_{i}\,\mathcal{F}_{P}\,p_{i+1},$ that is, $$p_i\times p_{i+1}=p_i, p_{i+1}=p_i\star p_{i+1},\, i=1, 2, \ldots, k-1.$$ Let $i\in\{1,2, \ldots, k-1\}$. By (R4) and (R1), we have
$$q_{i}\theta_{ q_{i+1}}=q_{i} \star q_{i+1}= q_{i}\star(q_{i}\star p_{i+1})=(q_{i}\star p_{i+1})\star(q_{i}\star p_{i+1})=q_{i}\star p_{i+1}=q_{i+1}.$$
On the other hand, by the fact that $q_{i+1}=q_{i}\star p_{i+1}$ and Lemma \ref{jiben1} (3), we have $$q_{i+1}\delta_{ q_{i}}=q_{i}\times q_{i+1}=q_{i}\times(q_{i}\star p_{i+1})=q_{i}\times p_{i+1}.$$ Observe  (\ref{v}) and the fact that $p_{i}\,\mathcal{F}_{P}\,p_{i+1}$, it follows that $q_{i}\leq_{P} p_{i}=p_{i}\times p_{i+1}$. This together with (\ref{wang2}) gives that $q_{i}=p_{i}\times(q_{i}\star p_{i+1})$. Combing (L4) and (L2), we obtain
$$q_{i+1}\delta _{q_{i}}=q_{i}\times p_{i+1}=(p_{i}\times(q_{i}\star p_{i+1}))\times p_{i+1}$$$$=(p_{i}\times(q_{i}\star p_{i+1}))\times(p_{i}\times p_{i+1})=(p_{i}\times(q_{i}\star p_{i+1}))\times p_{i}=p_{i}\times(q_{i}\star p_{i+1})= q_{i}.$$
This implies that $q_{i}\,\mathcal{F}_{P} \,q_{i+1} $, and so $_{q}{\downharpoonleft} \mathfrak{p}\in \mathscr{P}(P)$. By the definition of ${_q}{\downharpoonleft}\mathfrak{p}$, we have $\dd(_{q}{\downharpoonleft} \mathfrak{p})=q_{1}=q$. Moreover, (\ref{v}) gives that $\rr(_{q}{\downharpoonleft} \mathfrak{p})=q_{k}\leq_{P} p_{k}=\rr( \mathfrak{p})$.

(3) Let $\mathfrak{p}=(p_{1},p_{2},\ldots,p_{k})$. Then $q=\dd(\mathfrak{p})=p_{1}$, and so $_{q}{\downharpoonleft} \mathfrak{p}=(q_{1},q_{2},\ldots,q_{k})$, where $$q_{i}=q\theta_{p_{1}}\ldots\theta_{p_{i}},i=1,2, \ldots, k.$$ Obviously, $q_{1}=q=p_{1}$. Assume that $q_{i-1}=p_{i-1}$, $i\geq 2$. Since $\mathfrak{p}\in \mathscr{P}(P)$, it follows that $p_{i-1}\,\mathcal{F}_{P}\,p_{i}$, and so $p_{i-1}\star p_{i}=p_{i}$. By (\ref{u}), $q_{i}=q_{i-1}\star p_{i}=p_{i-1}\star p_{i}=p_{i}$. By mathematical induction, the result follows.

(5) Let $$\mathfrak{p}=(p_{1},p_{2},\ldots,p_{k}),\, _{q}{\downharpoonleft} \mathfrak{p}=(q_{1},q_{2},\ldots,q_{k}),$$$$  _{r}{\downharpoonleft}\mathfrak{p}=(r_{1},r_{2},\ldots,r_{k}),\,\,  _{r}{\downharpoonleft}(_{q}{\downharpoonleft} \mathfrak{p})=(s_{1},s_{2},\ldots,s_{k}).$$
By (\ref{u}) and (\ref{v}),
\begin{align}\label{min}
&q_{i}=q \theta_{p_{1}}\ldots\theta_{p_{i}}=q_{i-1}\star p_{i},r_{i}=r\theta_{p_{1}}\ldots\theta_{p_{i}}=r_{i-1}\star p_{i} , s_{i}
=r\theta_{q_{1}}\ldots\theta_{q_{i}}=s_{i-1}\star q_{i},
\notag
\\&q_{i}\leq_{P} p_{i},r_{i}\leq_{P} p_{i},s_{i}\leq_{P} q_{i},i=1,2, \ldots, k.
\end{align}
When $i=1$, $r_{1}=r=s_{1}$. Let $r_{i-1}=s_{i-1}$. Then the fact $s_{i-1}\leq_{P} q_{i-1}$ gives $s_{i-1}=s_{i-1}\star q_{i-1}$ by Lemma \ref{jiben1} (2). This together with the fact $q_{i}=q_{i-1}\star p_{i}$ and (R3) implies that
$$s_{i}=s_{i-1}\star q_{i}=(s_{i-1}\star q_{i-1})\star(q_{i-1}\star p_{i})=(s_{i-1}\star q_{i-1})\star p_{i}=s_{i-1}\star p_{i}=r_{i-1}\star p_{i}=r_{i}.$$
By mathematical induction, the result follows.

(7) Let $\mathfrak{p}=(p_{1},p_{2},\ldots,p_{k})\in \mathscr{P}(P)$ and $r\leq_{P}\dd(\mathfrak{p})$. By (\ref{v}), $${_r}{\downharpoonleft}\mathfrak{p}=(r,r\theta_{p_{2}},\ldots,r\theta_{p_{1}}\theta_{p_{2}}\ldots\theta_{p_{k}}).$$ Let $t=\rr({_r}{\downharpoonleft}\mathfrak{p})=r\theta_{p_{1}}\theta_{p_{2}}\ldots\theta_{p_{k}}$. Then $t\leq_{P} \rr(\mathfrak{p})=p_{k}$ by  item (1). In view of (\ref{win}),  $$\mathfrak{p}{\downharpoonright}{_{t}}=(t\delta_{p_{k}}\delta_{p_{k-1}}\ldots\delta_{p_{1}},\ldots,t\delta_{p_{k}}\delta_{p_{k-1}},t\delta_{p_{k}}).$$
In the following statements, we will show  ${_r}{\downharpoonleft}\mathfrak{p}=\mathfrak{p}{\downharpoonright}{_{t}}$. In fact, by the fact that $t\leq_{P} \rr(\mathfrak{p})=p_{k}$ and (\ref{l}), we have $t\delta_{p_{k}}=t=r\theta_{p_{1}}\theta_{p_{2}}\ldots\theta_{p_{k}}.$ On the other hand, by (\ref{v}) and the fact that $p_{k-1}\mathcal{F}_{P} p_{k}$, we have
$$r\theta_{p_{1}}\theta_{p_{2}}\ldots\theta_{p_{k-1}}\leq_{P} p_{k-1}= p_{k-1}\times p_{k}.$$ This together with (\ref{wang2}) gives that  $$r\theta_{p_{1}}\theta_{p_{2}}\ldots\theta_{p_{k}}\delta_{p_{k-1}}=p_{k-1}\times ((r\theta_{p_{1}}\theta_{p_{2}}\ldots\theta_{p_{k-1}})\star p_{k})=r\theta_{p_{1}}\theta_{p_{2}}\ldots\theta_{p_{k-1}}.$$ By mathematical induction, we have
$$r\theta_{p_{1}}\theta_{p_{2}}\ldots\theta_{p_{i}}=r\theta_{p_{1}}\theta_{p_{2}}\ldots\theta_{p_{k}}\delta_{p_{k}}\ldots \delta_{p_{i}},\,\,i=k,k-1,\ldots,1.$$ Thus ${_r}{\downharpoonleft}\mathfrak{p}=\mathfrak{p}{\downharpoonright}{_{t}}=\mathfrak{p}{\downharpoonright}_{\rr({_r}{\downharpoonleft}\mathfrak{p})}$.
\end{proof}
\begin{lemma}\label{kin}Let $(P,\times,\star)$ be a strong projection algebra and $\mathfrak{p},\mathfrak{q}\in \mathscr{P}(P) $, $\rr(\mathfrak{p})=\dd(\mathfrak{q})$.
\begin{itemize}
\item[(1)] If $r\in P$, $r \leq_{P} \dd(\mathfrak{p})$, $s=\rr({_{r}}{\downharpoonleft}\mathfrak{p})$, then ${_{r}}{\downharpoonleft}(\mathfrak{p}\circ \mathfrak{q} )={_{r}}{\downharpoonleft}\mathfrak{p}  \circ  { _{s}}{\downharpoonleft }\mathfrak{q}$.
\item[(2)] If $l\in P,\,\,l \leq_{P} \rr(\mathfrak{q})$, $t=\dd(\mathfrak{q}{\downharpoonright}{_l})$, then $(\mathfrak{p}\circ \mathfrak{q} ){\downharpoonright}{_l}=\mathfrak{p}{\downharpoonright}{_t}\circ   \mathfrak{q}{\downharpoonright}{_l}$.
\end{itemize}
\end{lemma}
\begin{proof}
Assume that $\mathfrak{p}=(p_{1},p_{2},\ldots,p_{k})$ and $\mathfrak{q}=(q_{1},q_{2},\ldots,q_{l})$. Since $\rr(\mathfrak{p})=\dd(\mathfrak{q})$, we have $p_{k}=q_{1}$. By (\ref{u}),
$$_{r}{\downharpoonleft}\mathfrak{p}=(r,r\theta_{p_{2}},r\theta_{p_{2}}\theta_{p_{3}},\ldots,r\theta_{p_{2}}\ldots\theta_{p_{k}}),$$ and so
$s=\rr(_{r}{\downharpoonleft}\mathfrak{p})=r\theta_{p_{2}}\ldots\theta_{p_{k}}$. Again by (\ref{u}), we have
$$_{r}{\downharpoonleft} (\mathfrak{p}\circ \mathfrak{q} )={_r}{\downharpoonleft}(p_{1},p_{2},\ldots,p_{k},q_{2},\ldots,q_{l})
$$$$=(r,r\theta_{p_{2}},r\theta_{p_{2}}\theta_{p_{3}},\ldots,r\theta_{p_{2}}\ldots\theta_{p_{k}}
=s,\,\,\,\, s\theta_{q_{2}},s\theta_{q_{2}}\theta_{q_{3}},\ldots,s\theta_{q_{2}}\ldots\theta_{q_{l}})
$$$$ =(r,r\theta_{p_{2}},r\theta_{p_{2}}\theta_{p_{3}},\ldots,r\theta_{p_{2}}\ldots\theta_{p_{k}})\circ(s,s\theta_{q_{2}},s\theta_{q_{2}}\theta_{q_{3}},\ldots,s\theta_{q_{2}}\ldots\theta_{q_{l}})
={_r}{\downharpoonleft}\mathfrak{p}\circ { _s}{\downharpoonleft}\mathfrak{q}.$$ This proves item (1). Dually, we can prove (2).
\end{proof}

\begin{defn}Let $(C,\circ,\dd,\rr, \leq)$ be an ordered category such that $(P_{C},\times,\star)$ is a strong projection algebra. Then $(C,\circ,\dd,\rr, \leq, P_C, \times, \star)$ is called a {\em weak projection ordered category} if $p\leq q$ if and only if $p\leq_{P_{C}} q$ for all $p,q\in P_{C}$.
\end{defn}

\begin{theorem}\label{dddnu}Let $(P,\times,\star)$ be a strong projection algebra. Define a relation  ``$\leq$" on $\mathscr{P}(P)$ as follows: For all $\mathfrak{p}, \mathfrak{q}\in \mathscr{P}(P)$,
$$\mathfrak{p}\leq_{\mathscr{P}} \mathfrak{q}\Longleftrightarrow \mbox{there exists }r\in P \mbox{ such that } r\leq_P \dd(\mathfrak{q}) \mbox{ and } \mathfrak{p}={_r}{\downharpoonleft}\mathfrak{q}.$$
Then ``$\leq_{\mathscr{P}}$" is a partial order on $\mathscr{P}(P)$, and $(\mathscr{P}(P),\circ,\dd,\rr, \leq_{\mathscr{P}}, P, \times, \star)$(Note that $P_{\mathscr{P}(P)}=P$) forms a weak projection ordered category.
\end{theorem}
\begin{proof}
Let $\mathfrak{p},\mathfrak{q},\mathfrak{m}\in \mathscr{P}(P)$. By Lemma \ref{ccts} (3), we have $_{\dd(\mathfrak{p})}{\downharpoonleft}\mathfrak{p}=\mathfrak{p}$, and so $\mathfrak{p}\leq_{\mathscr{P}} \mathfrak{p}$. Let $\mathfrak{p}\leq_{\mathscr{P}} \mathfrak{q}$ and $\mathfrak{q}\leq_{\mathscr{P}} \mathfrak{p}$. Then there exist $r,l\in P$ such that $$r\leq_{P} \dd(\mathfrak{q}) ,l\leq_{P} \dd(\mathfrak{p}), \mathfrak{p}={_r}{\downharpoonleft}\mathfrak{q},\mathfrak{q}={_l}{\downharpoonleft}\mathfrak{p}$$ By Lemma \ref{ccts} (1),  $r\leq_{P} \dd(\mathfrak{q})=\dd(_{l}{\downharpoonleft}\mathfrak{p})=l$. Dually, $l\leq_{P}r=\dd(\mathfrak{p})$. Thus  $r=l=\dd(\mathfrak{p})=\dd(\mathfrak{q})$, and hence $\mathfrak{p}={_r}{\downharpoonleft}\mathfrak{q}={_{\dd(\mathfrak{q})}}{\downharpoonleft} \mathfrak{q}=\mathfrak{q}$ by Lemma \ref{ccts} (3). Assume that $\mathfrak{p}\leq_{\mathscr{P}} \mathfrak{q}$ and $\mathfrak{q}\leq_{\mathscr{P}} \mathfrak{m}$. Then there exist $r,l\in P$ such that $$r\leq_{P} \dd(\mathfrak{q}),l\leq_{P} \dd(\mathfrak{m}),\mathfrak{p}={_r}{\downharpoonleft}\mathfrak{q},\mathfrak{q}={_l}{\downharpoonleft}\mathfrak{m}.$$ This together with Lemma \ref{ccts} (1) gives that $r\leq_{P}\dd(\mathfrak{q})=\dd(_{l}{\downharpoonleft}\mathfrak{m})=l\leq_{P}\dd(m)$.  By Lemma \ref{ccts} (5), we have $ _{r}{\downharpoonleft}\mathfrak{m}= {_r}{\downharpoonleft}({_l}{\downharpoonleft}\mathfrak{m})={_r}{\downharpoonleft} \mathfrak{q}=\mathfrak{p}$, and so $\mathfrak{p}\leq_{\mathscr{P}}  \mathfrak{m}$.  Thus ``$\leq_{\mathscr{P}} $"  is a partial order.

Let $p,q\in P$. If $p\leq_{\mathscr{P}}  q$, then there exists $r\in P$ such that $r\leq_{P}\dd(q)=q, p={_r}{\downharpoonleft}q$. By $(\ref{u})$, we have $p=r$,  and so $p\leq_{P}q$.
Conversely,  if $p \leq_{P} q$, then $p\leq_{P}q=\dd(q)$. By $(\ref{u})$, we obtain $_{p}{\downharpoonleft} q=p$. This gives that $p\leq_{\mathscr{P}}  q$. Thus for all $p,q\in P$, $p\leq_{\mathscr{P}}  q$ if and only if $p\leq_{P}q$.

(O1) Let $\mathfrak{p}\leq_{\mathscr{P}}  \mathfrak{q}$. Then there exists $r\in P$ such that $r\leq_{P} \dd(\mathfrak{q})$ and $ \mathfrak{p}={_r}{\downharpoonleft}\mathfrak{q}$. By Lemma
\ref{ccts} (1),
$$\dd(\mathfrak{p})=\dd({_r}{\downharpoonleft}\mathfrak{q})=r\leq_{P} \dd(\mathfrak{q}) ,\rr(\mathfrak{p})=\rr({_r}{\downharpoonleft}\mathfrak{q})\leq_{P} \rr(\mathfrak{q}).$$ This implies that $\dd(\mathfrak{p})\leq_{\mathscr{P}}  \dd(\mathfrak{q})$ and $\rr(\mathfrak{p})\leq_{\mathscr{P}}  \rr(\mathfrak{q})$.

(O2) Let $\mathfrak{p}\leq_{\mathscr{P}}  \mathfrak{q},\mathfrak{n}\leq_{\mathscr{P}}  \mathfrak{m}$ and $\rr(\mathfrak{p})=\dd(\mathfrak{n}),\rr(\mathfrak{q})=\dd(\mathfrak{m})$. Then there exist $r,l\in P$ such that $$\mathfrak{p}={_r}{\downharpoonleft}\mathfrak{q},\mathfrak{n}={_l}{\downharpoonleft}\mathfrak{m},r\leq _{P} \dd(\mathfrak{q}),l\leq_{P} \dd(\mathfrak{m}).$$ By Lemma \ref{ccts} (1), $l=\dd({_l}{\downharpoonleft}{\mathfrak m})=\dd(\mathfrak{n})=\rr(\mathfrak{p})=\rr({_r}{\downharpoonleft}\mathfrak{q})$, and so $\mathfrak{p}\circ \mathfrak{n}=({_r}{\downharpoonleft }\mathfrak{q}) \circ ({_l}{\downharpoonleft}\mathfrak{m})={_r}{\downharpoonleft}(\mathfrak{q}\circ \mathfrak{m})$  by Lemma \ref{kin} (1). Observe that $r\leq_{P}\dd(\mathfrak{q})=\dd(\mathfrak{q}\circ \mathfrak{m})$, it follows that $\mathfrak{p}\circ \mathfrak{n}\leq_{\mathscr{P}}  \mathfrak{q}\circ \mathfrak{m}$.

(O3) Let  $r\leq_{P}  \dd(\mathfrak{q})$. By Lemma \ref{ccts} (1), we have ${_r}{\downharpoonleft} \mathfrak{q}\in \mathscr{P}(P)$ and $\dd({_r}{\downharpoonleft} \mathfrak{q})=r$. Obviously, ${_r}{\downharpoonleft} \mathfrak{q}\leq_{\mathscr{P}}  \mathfrak{q}$.  Let $\mathfrak{p}\in \mathscr{P}(P)$, $\dd(\mathfrak{p})=r$ and $\mathfrak{p}\leq_{\mathscr{P}}  \mathfrak{q}$. Then there exists $\l\in P$ such that $l\leq_{P}  \dd(\mathfrak{q})$ and $\mathfrak{p}={_l}{\downharpoonleft } \mathfrak{q}$. This implies that $l=\dd({_l}{\downharpoonleft} \mathfrak{q})=\dd(\mathfrak{p})=r$ by Lemma \ref{ccts} (1) again. Thus $\mathfrak{p}={_l}{\downharpoonleft} \mathfrak{q}={_r}{\downharpoonleft} \mathfrak{q}$.

(O4) Let $l\leq_{P}  \rr(\mathfrak{q})$. By Lemma \ref{ccts} (2), we have $\mathfrak{q}{\downharpoonright}{_l}\in \mathscr{P}(P)$ and $\rr(\mathfrak{q}{\downharpoonright}{_l})=l$.  By Lemma \ref{ccts} (8), we have $\mathfrak{q}{\downharpoonright}{_l}={_{\dd(\mathfrak{q}{\downharpoonright}{_l})}}{\downharpoonleft } \mathfrak{q}$, and so $\mathfrak{q}{\downharpoonright}{_l}\leq_{\mathscr{P}}  \mathfrak{q}$.  Now let $\mathfrak{p}\in \mathscr{P}(P), \rr(\mathfrak{p})=l$ and $ \mathfrak{p}\leq_{\mathscr{P}}  \mathfrak{q}$. Then there exists $m\leq_{P}  \dd(\mathfrak{q})$ such that $\mathfrak{p}={_m}{\downharpoonleft}\mathfrak{q}=\mathfrak{q}{\downharpoonright}{_{\rr({_m}{\downharpoonleft}\mathfrak{q})}}$ by Lemma \ref{ccts} (7).
So  $\rr({_m}{\downharpoonleft}\mathfrak{q})=\rr(\mathfrak{q}{\downharpoonright}{_{\rr({_m}{\downharpoonleft}\mathfrak{q})}})=\rr(\mathfrak{p})=l$ by Lemma \ref{ccts} (2). Thus $\mathfrak{p}=\mathfrak{q}{\downharpoonright}{_{\rr({_m}{\downharpoonleft}\mathfrak{q})}}=\mathfrak{q}{\downharpoonright}{_l}$.
\end{proof}

In the following statements, we shall give some properties of general weak projection ordered categories.
Let $(C,\circ,\dd,\rr,\leq, P_C, \times,\star)$ be a weak projection ordered category and $a\in C$. Then $\dd(a)\in P_{C}$. By (\ref{c}) and (\ref{j}), we can define
\begin{equation}\label{aa}
\Theta_{a}=\theta_{\dd(a)}\nu_{a}:P_{C}\longmapsto \rr(a)^{\downarrow},x\longmapsto \rr(_{x\star \dd(a)}{\downharpoonleft}a).
\end{equation}
Dually, we have
\begin{equation}\label{cpia}
\Delta_{a}=\delta_{\rr(a)}\mu_{a}:P_{C}\longmapsto \dd(a)^{\downarrow},y\longmapsto \dd(a{\downharpoonright}_{\rr(a)\times y}).
\end{equation}
It is easy to see that
\begin{equation}\label{ccwio}
\Theta_{p}=\theta_{p},\,\,\, \Delta_{p}=\delta_{p}.
\end{equation}
 for all $p\in P_{C}$ by (1) and (2) in Lemma \ref{jiben1}, Remark \ref{fanchoudexingzhi} and Lemma \ref{b} (11).
 Moreover, by (1) and (2) in Lemma \ref{jiben1}, we can obtain that
\begin{equation}\label{xpha}
\Theta_{a}{\mid}{_{\dd(a)^{\downarrow}}}=\nu_{a},\,\,\,\Delta_{a}{\mid}{_{\rr(a)^{\downarrow}}}=\mu_{a}.
\end{equation}
\begin{lemma}\label{dd}Let $(C,\circ,\dd,\rr,\leq, P_C, \times,\star)$ be a weak projection ordered category and $a, b\in C$.
\begin{itemize}
\item[(1)] $\nu_{a}\theta_{\rr(a)}=\nu_{a}=\nu_{a}\delta_{\rr(a)}$ and $\mu_{a}\delta_{\dd(a)}=\mu_{a}=\mu_{a}\theta_{\dd(a)}$.
\item[(2)]$\Theta_{a}\theta_{\rr(a)}=\Theta_{a}=\theta_{\dd(a)}\Theta_{a}$ and $\Delta_{a}\delta_{\dd(a)}=\Delta_{a}=\delta_{\rr(a)}\Delta_{a}$.
\item[(3)]$\Theta_{a}\mu_{a}=\theta_{\dd(a)}$.
\item[(4)] If $a\leq b$, then $\Theta_{a}=\theta_{\dd(a)}\nu_{b}$.
\item[(5)]For all $p,q\in P_{C}$, if $p\leq \dd(a)$ and $q\leq \rr(a)$, then $$\Theta_{_{p}\downharpoonleft a}=\theta_{p}\Theta_{a},\,\,\Delta_{a\downharpoonright_{q}}=\delta_{q}\Delta_{a},\,\,\Theta_{a\downharpoonright_{q}}=\theta_{q\mu_{a}}\Theta_{a},\,\,
   \Delta_{_{p}\downharpoonleft a}=\delta_{p\nu_{a}}\Delta_{a}.$$
\item[(6)] If $\rr(a)=\dd(b)$, then $\Theta_{a\circ b}=\Theta_{a}\Theta_{b}$ and $\Delta_{a\circ b}=\Delta_{b}\Delta_{a}$.
\item[(7)] $\Theta_{a}\Delta_{a}=\theta_{\dd(a)}$ and $\Delta_{a}\Theta_{a}=\delta_{\rr(a)}$.
\end{itemize}
\end{lemma}
\begin{proof}
We first observe that $p\leq q$ if and only if $p\leq_{P_{C}} q$ for all $p,q\in P_{C}$ as $C$ is a weak projection ordered category.

(1) It is easy to see that  $\nu_{a}\theta_{\rr(a)}$,\,$\nu_{a}$ and $\nu_{a}\delta_{\rr(a)}$ have the same domain $\dd(a)^{\downarrow}$. Let $x\in \dd(a)^{\downarrow}$. Then $x\nu_{a}\in \rr(a)^{\downarrow}$, i.e. $x\nu_{a}\leq \rr(a)$. By (1) and (2) in Lemma \ref{jiben1}, we have $$x(\nu_{a}\theta_{\rr(a)})=(x\nu_{a})\theta_{\rr(a)}=(x\nu_{a})\star \rr(a)=x\nu_{a}=\rr(a)\times (x\nu_{a})=x\nu_{a}\delta_{\rr(a)}.$$ This implies that $\nu_{a}\theta_{\rr(a)}=\nu_{a}=\nu_{a}\delta_{\rr(a)}$. Dually, $\mu_{a}\delta_{\dd(a)}=\mu_{a}=\mu_{a}\theta_{\dd(a)}$.

(2) The fact $\Theta_{a}=\theta_{\dd(a)}\nu_{a}$ and (1) imply that $\Theta_{a}\theta_{\rr(a)}=\theta_{\dd(a)}\nu_{a}\theta_{\rr(a)}=\Theta_{a}$. On the other hand, let $t\in P_{C}$. Then $t\star \dd(a)\leq \dd(a)$ and so $(t\star \dd(a))\star \dd(a)=t\star \dd(a)$ by Lemma \ref{jiben1} (2). This implies that $$t\theta_{\dd(a)}\Theta_{a}=\rr(_{(t\star \dd(a))\star \dd(a)}{\downharpoonleft} a)=\rr(_{t\star \dd(a)}{\downharpoonleft} a)=t\Theta_{a}.$$  Since $\theta_{\dd(a)}\Theta_{a}$ and $\Theta_{a}$ have the same domain $P_{C}$,  we have $\theta_{\dd(a)}\Theta_{a}=\Theta_{a}$. Dually, $\Delta_{a}\delta_{\dd(a)}=\Delta_{a}=\delta_{\rr(a)}\Delta_{a}$.

(3) Obviously, $\Theta_{a}\mu_{a}$ and $\theta_{\dd(a)}$ have the same domain $P_{C}$. Let $t\in P_{C}$. By Lemma \ref{jiben1} (2), $t\star \dd(a)\leq \dd(a)$. By using  (5)  and (1) in Lemma \ref{b},  we have $$t\Theta_{a}\mu_{a}=\dd(a{\downharpoonright}_{\rr(_{t\star \dd(a)}{\downharpoonleft} a)})=\dd(_{t\star \dd(a)}{\downharpoonleft} a)=t\star \dd(a)=t\theta_{\dd(a)}.$$ Thus $\Theta_{a}\mu_{a}=\theta_{\dd(a)}$.

(4) Let $a\leq b$.  Then $a= {_{\dd(a)}}{\downharpoonleft}b$ by Lemma \ref{b} (3). Let $t\in P_{C}$. By Lemma \ref{jiben1} (2) and (O1),  we have $t\star \dd(a)\leq \dd(a)\leq \dd(b)$. By Lemma \ref{b} (9),
$$t\Theta_{a}=\rr(_{(t\star \dd(a))}{\downharpoonleft} a)=\rr(_{(t\star \dd(a))}{\downharpoonleft}(_{\dd(a)}{\downharpoonleft} b))=\rr(_{(t\star \dd(a))}{\downharpoonleft } b)=t\theta_{\dd(a)}\nu_{b}.$$

(5) Let $p\leq \dd(a)$ and $t\in P_{C}$. By Lemma \ref{jiben1} (2), we have $t\star p\leq p\leq \dd(a)$ and $(t\star p)\star \dd(a)=t\star p$. By the definition of $\Theta_{a}$, (1)  and (9) in Lemma \ref{b},
$$t\Theta_{{_p}{\downharpoonleft} a}=\rr(_{(t\star \dd({_p}{\downharpoonleft} a))}{\downharpoonleft}(_{p}{\downharpoonleft} a))=\rr(_{t\star p}{\downharpoonleft}({_p}{\downharpoonleft } a))=\rr(_{(t\star p)}{\downharpoonleft} a)=\rr(_{(t\star p)\star \dd(a)}{\downharpoonleft} a)=t\theta_{p}\Theta_{a}.$$
This implies that $\Theta_{{_p}{\downharpoonleft} a}=\theta_{p}\Theta_{a}$.
Dually, $\Delta_{a\downharpoonright_{q}}=\delta_{q}\Delta_{a}$ for all $q\in P_{C}$ with $q\leq \rr(a)$ . On the other hand, let $q\in P_{C}$ and $q\leq \rr(a)$. By Lemma \ref{b} (6), we obtain $a{\downharpoonright}{_q}={_{\dd(a{\downharpoonright}_{q})}}{\downharpoonleft}a={_q\mu_{a}}{\downharpoonleft}a$, and so $\Theta_{a{\downharpoonright}{_q}}=\Theta_{_{q\mu_{a}}{\downharpoonleft}a}=\theta_{q\mu_{a}}\Theta_{a}$. Dually, $\Delta_{_{p}\downharpoonleft a}=\delta_{p\nu_{a}}\Delta_{a}$ for all $p\in P_C$ with $p\leq \dd(a)$.

(6) Let $p=\dd(a)$ and $q=\rr(a)=\dd(b)$. Using (1) of the present lemma, Lemma \ref{ee} and the fact that $\dd(a\circ b)=\dd(a)$ in order,
$$\Theta_{a}\Theta_{b}=\theta_{\dd(a)}\nu_{a}\theta_{\dd(b)}\nu_{b}=\theta_{\dd(a)}\nu_{a}\theta_{\rr(a)}\nu_{b}
=\theta_{\dd(a)}\nu_{a}\nu_{b}=\theta_{\dd(a)}\nu_{a\circ b}=\theta_{\dd(a\circ b)}\nu_{a\circ b}=\Theta_{a\circ b}.$$
Dually, $\Delta_{a\circ b}=\Delta_{b}\Delta_{a}$.

(7) Let $t\in P$. Then $t\star \dd(a)\leq \dd(a)$ By Lemma \ref{jiben1} (2) . By Lemma \ref{b} (1), we have $\rr(_{(t\star \dd(a))}{\downharpoonleft} a)\leq \rr(a)$, which  together with Lemma \ref{jiben1} (1) gives that $\rr(a)\times \rr(_{(t\star \dd(a))}{\downharpoonleft} a)=\rr(_{(t\star \dd(a))}{\downharpoonleft} a).$ By (5) and (1) in Lemma \ref{b},
$$t\Theta_{a}\Delta_{a}=\dd(a{\downharpoonright}_{\rr(a)\times \rr(_{(t\star \dd(a))}{\downharpoonleft} a)} )=\dd(a{\downharpoonright}_{\rr(_{(t\star \dd(a))}{\downharpoonleft} a)} )=\dd(_{(t\star \dd(a))}{\downharpoonleft} a)=t\star \dd(a)=t\theta_{\dd(a)}.$$ Since $\Theta_{a}\Delta_{a}$ and $\theta_{\dd(a)}$ have the same domain $P_{C}$, we have $\Theta_{a}\Delta_{a}=\theta_{\dd(a)}$. Dually, we can prove that $\Delta_{a}\Theta_{a}=\delta_{\rr(a)}$.
\end{proof}
Let $(C,\circ,\dd,\rr,\leq, P_C, \times,\star)$ be a weak projection ordered category. Then $p^{\downarrow}=\{x\in P_{C}\mid x\leq p\}$ is a subalgebra of $(P_{C},\times,\star)$ for all $p\in P_{C}$. In fact, if $x,y\in p^{\downarrow}$, then $x\leq p$ and $y\leq p$. By (1) and (2) in Lemma \ref{jiben1}, we have $x\times y\leq x\leq p$ and $x\star y\leq y\leq p$. This yields that $x\times y,x\star y\in p^{\downarrow}$.
\begin{lemma}\label{ii}
Let $(C,\circ,\dd,\rr,\leq, P_C, \times,\star)$ be a weak projection ordered category. Then the following statements are equivalent:
\begin{itemize}
\item[\rm (G1a)] For all $a\in C$ and $p,q\in P_{C}$, if $p\leq \dd(a)$ and $q\leq \rr(a)$, then $\delta_{p\nu_{a}}=\Delta_{a}\delta_{p}\Theta_{a}$ and $ \theta_{q\mu_{a}}=\Theta_{a}\theta_{q}\Delta_{a}$.
\item[\rm (G1b)]For all $a\in C$ and $p\in P_{C}$, we have $\delta_{p\Theta_{a}}=\Delta_{a}\delta_{p}\Theta_{a}$ and $\theta_{p\Delta_{a}}=\Theta_{a}\theta_{p}\Delta_{a}$.
\item[\rm (G1c)]For all $a\in C$ and $p, q\in P_{C}$, if $p\leq \dd(a)$ and $q\leq \rr(a)$, then $\Delta_{_{p}{\downharpoonleft} a}=\Delta_{a}\delta_{p}$ and $\Theta_{a{\downharpoonright}_{q}}=\Theta_{a}\theta_{q}$.
\item[\rm (G1d)]For all $a\in C$,  the map $\nu_{a}: \dd(a)^{\downarrow}\rightarrow\rr(a)^{\downarrow},\, p\mapsto \rr({_p}{\downharpoonleft} a)$ is a homomorphism between projection algebras.
\end{itemize}
\end{lemma}
\begin{proof}Let $a\in C$ and denote $s=\dd(a)$ and $t=\rr(a)$. Then $s, t\in P_{C}$.

(G1a) $\Longrightarrow$ (G1b). Let $p\in P_{C}$. By using the definition of $\Theta_{a}$, (G1a), the fact that $p\theta_{s}=p\star s \leq s=\dd(a)$, Lemma \ref{jj} and Lemma
\ref{dd} (2) in order, we have
$$\delta_{p\Theta_{a}}=\delta_{(p\theta_{s})\nu_{a}} =\Delta_{a}\delta_{(p\theta_{s})}\Theta_{a} =\Delta_{a}\delta_{s}\delta_{p}\theta_{s}\Theta_{a}
 =\Delta_{a}\delta_{p}\Theta_{a}.$$ Dually, $\theta_{p\Delta_{a}}=\Theta_{a}\theta_{p}\Delta_{a}$.

(G1a) $\Longrightarrow$ (G1c).
Let $p\in P_{C}$ with $ p\leq \dd(a)$. By (5) and (1) in Lemma \ref{b}, we have ${_p}{\downharpoonleft } a=a{\downharpoonright}_{r}$, where $r=\rr({_p}{\downharpoonleft} a)\leq \rr(a)$. By using Lemma \ref{dd} (5), (G1a), Lemma \ref{dd} (7) and Lemma \ref{o}, we have
$$\Delta_{_{p}{\downharpoonleft} a}=\delta_{p\nu_{a}}\Delta_{a} =\Delta_{a}\delta_{p}\Theta_{a}\Delta_{a} =\Delta_{a}\delta_{p}\theta_{\dd(a)}
 =\Delta_{a}\delta_{p}.$$ Dually, $\Theta_{a\downharpoonright_{q}}=\Theta_{a}\theta_{q}$ for all $q\in P_{C}$ with $ q\leq \rr(a)$.

(G1b) $\Longrightarrow$ (G1a).  Let $p\in P_{C}$ with $ p\leq \dd(a)$. By Lemma \ref{jiben1} (2), we have $p=p\star \dd(a)=p\theta_{\dd(a)}$, and so $p\nu_{a}=(p\theta_{\dd(a)})\nu_{a}=p(\theta_{\dd(a)}\nu_{a})=p\Theta_{a}$.  In view of (G1b), we obtain that  $\delta_{p\nu_{a}}=\delta_{p\Theta_{a}}=\Delta_{a}\delta_{p}\Theta_{a}$. Dually, $ \theta_{q\mu_{a}}=\Theta_{a}\theta_{q}\Delta_{a}$ for all $q\in P_{C}$ with $ q\leq \rr(a)$.

(G1c) $\Longrightarrow$ (G1a).
Let $p \in P_{C}$ with $p\leq \dd(a)$. By Lemma \ref{b} (1), we have $p\nu_{a}=\rr({_p}{\downharpoonleft} a)\leq \rr(a)$ and so $\delta_{p\nu_{a}}\delta_{\rr(a)}=\delta_{p\nu_{a}}$ by Lemma \ref{o}. Using (G1c) and (5) and (7) in Lemma \ref{dd}, we have
$$\Delta_{a}\delta_{p}\Theta_{a}=\Delta_{_{p}{\downharpoonleft} a}\Theta_{a} =\delta_{p\nu_{a}}\Delta_{a}\Theta_{a}
 =\delta_{p\nu_{a}}\delta_{\rr(a)} =\delta_{p\nu_{a}}.$$
Dually, $ \theta_{q\mu_{a}}=\Theta_{a}\theta_{q}\Delta_{a}$ for all $q\in P_{C}$ with $ q\leq \rr(a)$.

(G1c) $\Longrightarrow$ (G1d).
Assume that (G1c) holds. Then (G1a) also holds by the previous paragraph. Let $p,q \in P_{C}$ with $p,q \in \dd(a)^{\downarrow}$. Then $p,q\leq \dd(a)$.
On one hand,
\begin{equation*}
\begin{aligned}
(q\nu_{a})\times( p\nu_{a})&=p\nu_{a}\delta_{q\nu_{a}}\,\,\,\,\,\,(\mbox{by (\ref{j})})\\
&=p\nu_{a}\Delta_{a}\delta_{q}\Theta_{a}\,\,\,\,\,\,(\mbox{by (G1a)})\\
&=p\nu_{a}\delta_{\rr(a)}\mu_{a}\delta_{q}\theta_{\dd(a)}\nu_{a}\,\,\,\,\,\,(\mbox{by (\ref{aa}) and (\ref{cpia})})\\
&=p\nu_{a}\mu_{a}\delta_{q}\nu_{a}\,\,\,\,\,\,(\mbox{Lemma \ref{dd} (1) and \,Lemma \ref{o}})\\
&=p\delta_{q}\nu_{a}\,\,\,\,\,\,(\mbox{by Lemma \ref{f}})\\
&=(q\times p)\nu_{a}.\,\,\,\,\,\,(\mbox{by (\ref{j})})
\end{aligned}
\end{equation*}
This shows that $\nu_a$ preserves $\times$. On the other hand,
 \begin{equation*}
\begin{aligned}
p\Theta_{a}\theta_{\rr({_q}{\downharpoonleft}a)}&=p\Theta_{a{\downharpoonright}{_{\rr({_q}{\downharpoonleft}a)}}}\,\,\,\,\,\,( \mbox{by(G1c)})\\
&=p\Theta_{{_q}{\downharpoonleft}a}\,\,\,\,\,\,(\mbox{by Lemma \ref{b} (5)})\\
&=\rr({_{p\star \dd({_q}{\downharpoonleft}a)}}{\downharpoonleft}({_q}{\downharpoonleft}a))\,\,\,\,\,\,(\mbox{by (\ref{aa})})\\
&=\rr({_{p\star q}}{\downharpoonleft}({_q}{\downharpoonleft}a))\,\,\,\,\,\,(by \mbox{Lemma \ref{b} (1)})\\
&=\rr({_{p\star q}}{\downharpoonleft}a)\,\,\,\,\,\,(\mbox{by Lemma \ref{jiben1} (2) and Lemma \ref{b} (9)})\\
&=(p\star q)\nu_{a}.\,\,\,\,\,\,(\mbox{by (\ref{c})})
\end{aligned}
\end{equation*}
Moreover,
\begin{equation*}
\begin{aligned}
p\Theta_{a}\theta_{\rr({_q}{\downharpoonleft}a)}&=\rr({_{p\star \dd(a)}}{\downharpoonleft}a)\theta_{\rr({_q}{\downharpoonleft}a)}\,\,\,\,\,\,(\mbox{by (\ref{aa})})\\
&=\rr({_p}{\downharpoonleft}a)\theta_{\rr({_q}{\downharpoonleft}a)}\,\,\,\,\,\,(\mbox{by Lemma \ref{jiben1} (2)})\\
&=\rr({_p}{\downharpoonleft}a)\star \rr({_q}{\downharpoonleft}a)\,\,\,\,\,\,(\mbox{by (\ref{j})})\\
&=(p\nu_{a})\star(q\nu_{a}).\,\,\,\,\,\,(\mbox{by (\ref{c})})
\end{aligned}
\end{equation*}
 Thus $(p\star q)\nu_{a}=(p\nu_{a})\star(q\nu_{a}).$  Therefore $\nu_{a}$ is a homomorphism.

(G1d) $\Longrightarrow$  (G1a). By Lemma \ref{f} and (G1d),
$\nu_{a}: \dd(a)^{\downarrow}\rightarrow\rr(a)^{\downarrow},\,\,\,p\mapsto\rr({_p}{\downharpoonleft}a)$ is an isomorphism.
Let $p\leq \dd(a)=s$ and $e\in P_{C}$. We shall prove that $e\delta_{p\nu_{a}}=e\Delta_{a}\delta_{p}\Theta_{a}$ in the sequel.
Denote $f=t\times e=e\delta_{t}$, where $t=\rr(a)$. Then by  Lemma \ref{jiben1} (1), we have $f\leq t$, i.e. $f\in \rr(a)^{\downarrow}$, and so we can let $g=f\mu_{a}$. By Lemma \ref{f}, we have $f=f\mu_{a}\nu_{a}=g\nu_{a}$. Obviously, we obtain $g\in \dd(a)^{\downarrow}$, i.e. $g\leq \dd(a)=s$. By the facts that $p\nu_{a}\leq \rr(a)=t$, $p\leq \dd(a)=s$ and Lemma \ref{o}, we have $\delta_{p\nu_{a}}=\delta_{t}\delta_{p\nu_{a}}$ and $\delta_{p}\theta_{s}=\delta_{p}$. Using Lemma \ref{o}, the fact that $\nu_{a}$ is an isomorphism and the definitions of $\Delta_{a}$ and $\Theta_{a}$ in order, we have
$$e\delta_{p\nu_{a}}=e\delta_{t}\delta_{p\nu_{a}}=f\delta_{p\nu_{a}}=(g\nu_{a})\delta_{p\nu_{a}}=(g\delta_{p})\nu_{a}=(f\mu_{a})\delta_{p}\nu_{a}
=e\delta_{t}\mu_{a}\delta_{p}\nu_{a}$$$$=e\delta_{t}\mu_{a}\delta_{p}\theta_{s}\nu_{a}=e\delta_{\rr(a)}\mu_{a}\delta_{p}\theta_{\dd(a)}\nu_{a}=e\Delta_{a}\delta_{p}\Theta_{a}.$$
This shows $\delta_{p\nu_{a}}=\Delta_{a}\delta_{p}\Theta_{a}$. On the other hand,
by Lemma \ref{f} and (G1d),    $$\mu_{a}: \rr(a)^{\downarrow}\rightarrow \dd(a)^{\downarrow},\, q\mapsto q\mu_{a}=\dd(a{\downharpoonright}_q)$$  is also an isomorphism.  Thus $ \theta_{q\mu_{a}}=\Theta_{a}\theta_{q}\Delta_{a}$ by dual arguments.
\end{proof}

\begin{defn}\label{aaa}A weak projection ordered category $(C,\circ,\dd,\rr,\leq,P_{C},\times,\star)$ is called  a {\em projection ordered category} if  the conditions (G1a)--(G1d) in Lemma \ref{ii} hold.
\end{defn}

\begin{lemma}\label{nn}
Let $(C,\circ,\dd,\rr,\leq,P_{C},\times,\star)$ be a projection ordered category and $a\in C, p,q\in P_{C}$. Then
$$p\Delta_{a}\delta_{q}\Theta_{a}\theta_{p}=q\Theta_{a}\theta_{p},\,\,  p\Theta_{a}\theta_{q}\Delta_{a}\delta_{p}=q\Delta_{a}\delta_{p}.$$
\end{lemma}
\begin{proof}
Denote $t=q\Theta_{a}$. By (G1b) and Lemma \ref{jiben1} (3),
$$p\Delta_{a}\delta_{q}\Theta_{a}\theta_{p}=p\delta_{q\Theta_{a}}\theta_{p}=p\delta_{t}\theta_{p}=(t\times p)\star p=t\star p=t\theta_{p}=q\Theta_{a}\theta_{p}.$$
The other identity can be proved dually.
\end{proof}

\begin{defn}\label{ll}
Let $(C,\circ,\dd,\rr,\leq, P_C, \times,\star)$ be a projection ordered category, $\mathscr{P}(P_{C})$ be the path category of $P_{C}$ and $\varepsilon: \mathscr{P}(P_{C})\rightarrow C$ be a map. Then $\varepsilon$ is called an {\em evaluation map} if the following conditions hold: For all $p\in P_{C}$ and $ \mathfrak{p}, \mathfrak{q}\in \mathscr{P}(P_{C})$, 
\begin{itemize}
\item[\rm (E1)] $\varepsilon(p)=\varepsilon(p, p)=p.$
\item[\rm (E2)] $\dd(\varepsilon(\mathfrak{p}))=\dd(\mathfrak{p})$ and $\rr(\varepsilon(\mathfrak{p}))=\rr(\mathfrak{p})$.
\item[\rm (E3)]If $\mathfrak{p}\circ \mathfrak{q}$ is defined, then $\varepsilon(\mathfrak{p}\circ \mathfrak{q})=\varepsilon(\mathfrak{p})\circ \varepsilon(\mathfrak{q})$.
\item[\rm (E4)] If $\mathfrak{p}\leq_{\mathscr{P}} \mathfrak{q}$, then $\varepsilon(\mathfrak{p})\leq \varepsilon(\mathfrak{q}) $.
\end{itemize}
\end{defn}
\begin{prop}\label{fder}
Let $(C,\circ,\dd,\rr,\leq, P_C, \times,\star)$ be a projection ordered category, $\mathscr{P}(P_{C})$ be the path category of $P_{C}$ and $\varepsilon: \mathscr{P}(P_{C})\rightarrow C$ be a map satisfying (E1)--(E3). Then the following statements are equivalent:
\item[\rm (E4)] If $\mathfrak{p}\leq_{\mathscr{P}} \mathfrak{q}$, then $\varepsilon(\mathfrak{p})\leq \varepsilon(\mathfrak{q}) $.
\item[\rm (E5)] For all $\mathfrak{p}\in \mathscr{P}(P_{C})$ and $t\in P_{C},$ the fact that $t\leq \dd(\mathfrak{p})$ implies that $\varepsilon(_{t}{\downharpoonleft} \mathfrak{p} )={_t}{\downharpoonleft} \varepsilon(\mathfrak{p})$.
 \item[\rm (E6)] For all $\mathfrak{p}\in \mathscr{P}(P_{C})$ and $t\in P_{C},$ the fact that $t\leq \rr(\mathfrak{p})$ implies that $\varepsilon(\mathfrak{p}{\downharpoonleft}{_t} )=\varepsilon(\mathfrak{p}){\downharpoonleft}{_t}$.
\end{prop}
\begin{proof}
Assume that (E5) holds and $\mathfrak{p}\leq_{\mathscr{P}} \mathfrak{q}$. Then $\mathfrak{p}={_r}{\downharpoonleft}\mathfrak{q}$ for some $r\in P_{C}$ with $r\leq \dd(\mathfrak{q})$.  This implies that $\varepsilon(\mathfrak{p})=\varepsilon({_r}{\downharpoonleft} \mathfrak{q})={_r}{\downharpoonleft} \varepsilon(\mathfrak{q})\leq \varepsilon(\mathfrak{q})$. Conversely, assume that (E4) holds and $q\in P_{C}$ with $q\leq \dd(\mathfrak{p})$. By Theorem \ref{dddnu}, we have ${_q}{\downharpoonleft}\mathfrak{p}\leq_{\mathscr{P}} \mathfrak{p}$ and $q \leq_{\mathscr{P}} \dd(\mathfrak{p})$. Using (E1), (E4) and (E2), we have $q=\varepsilon(q)\leq \varepsilon(\dd(\mathfrak{p}))=\dd(\varepsilon(\mathfrak{p}))$. On the other hand,  (E4) gives $\varepsilon({_q}{\downharpoonleft}\mathfrak{p})\leq \varepsilon(\mathfrak{p})$, and (E2), (E1) and Lemma \ref{ccts} (1) together imply that $\dd(\varepsilon({_q}{\downharpoonleft}\mathfrak{p}))=\varepsilon(\dd({_q}{\downharpoonleft}\mathfrak{p}))=\dd({_q}{\downharpoonleft}\mathfrak{p})
=q.$  By Lemma \ref{b} (3), $\varepsilon(_{q}{\downharpoonleft} \mathfrak{p} )=_{q}{\downharpoonleft}\varepsilon(\mathfrak{p}).$ Thus  (E4)   is equivalent to (E5). Dually,  we can prove that (E4) is equivalent to (E6).
\end{proof}

Let $(C,\circ,\dd,\rr,\leq, P_C, \times,\star)$ be a projection ordered category, $\mathscr{P}(P_{C})$ be the path category of $P_{C}$ and $\varepsilon: \mathscr{P}(P_{C})\rightarrow C$ be an  evaluation map, $\mathfrak{p}=(p_{1},p_{2},\ldots,p_{k})\in \mathscr{P}(P_{C})$. By (E3) and the fact that $\mathfrak{p}=(p_1, p_2)\circ (p_2, p_3)\circ \cdots \circ (p_{k-1}, p_k)$ in $\mathscr{P}(P)$,  we have
\begin{equation}\label{pp}
\varepsilon(\mathfrak{p})=\varepsilon(p_{1},p_{2})\circ \varepsilon(p_{2},p_{3})\circ \cdots \circ \varepsilon(p_{k-1},p_{k}).
\end{equation}
\begin{lemma}\label{oo}
Let $(C,\circ,\dd,\rr,\leq, P_C, \times,\star)$ be a projection ordered category with an  evaluation map $\varepsilon:\mathscr{P}(P_{C})\rightarrow C $ and $$a\in C, p, q, r, s\in P_{C},~ p \mathcal{F}_{P_{C}} q,~ r\leq p, s\leq q, a\leq \varepsilon(p,q).$$
\begin{itemize}
\item[(1)]
$_{r}{\downharpoonleft} \varepsilon(p,q)=\varepsilon(r,r\theta_{q}) $ and $  \varepsilon(p,q){\downharpoonright}{_s}=\varepsilon(s\delta_{p},s).$
\item[(2)]$$\dd(\varepsilon(p,q))=p,\quad   \nu_{\varepsilon(p,q)}=\theta_{q}{\mid}_{p^{\downarrow}}, \quad\Theta_{\varepsilon(p,q)}=\theta_{p}\theta_{q},$$
    $$\rr(\varepsilon(p,q))=q,\quad \mu_{\varepsilon(p,q)}=\delta_{p}{\mid}_{q^{\downarrow}},\quad \Delta_{\varepsilon(p,q)}=\delta_{q}\delta_{p},$$
\item[(3)]$a=\varepsilon(\dd(a),\rr(a))$.
\end{itemize}
\end{lemma}
\begin{proof}
(1) By the fact that $r\leq p=\dd((p,q))$ and (\ref{u}), we have ${_r}{\downharpoonleft}(p,q)=(r,r\theta_{q})$ and so  $\varepsilon(r,r\theta_{q})=\varepsilon({_r}{\downharpoonleft}(p,q))={_r}{\downharpoonleft} \varepsilon(p,q)$ by (E5). Dually, $\varepsilon(p,q){\downharpoonright}{_s}=\varepsilon(s\delta_{p},s)$.

(2) By (E2), we obtain $\dd(\varepsilon(p,q))=\dd((p,q))=p$. This implies that $\nu_{\varepsilon(p,q)}$ and $\theta_{q}{\mid}_{p^{\downarrow}}$ have the same domain $p^{\downarrow}$. Let $r\in p^{\downarrow}$. Then $r\leq p$ and so
$$r\nu_{\varepsilon(p,q)}=\rr(_{r}{\downharpoonleft} \varepsilon(p,q))\overset{({\rm E}5)}{=}\rr(\varepsilon(_{r}{\downharpoonleft} (p,q)))\overset{({\rm 4}.1)}{=}\rr(\varepsilon(r,r\theta_{q}))\overset{({\rm E}2)}{=}\rr(r,r\theta_{q})=r\theta_{q}.$$ Observe that the range of $\theta_{p}$ is $p^{\downarrow}$, it follows that $$\Theta_{\varepsilon(p,q)}=\theta_{\dd(\varepsilon(p,q))}\nu_{\varepsilon(p,q)}=\theta_{p}\nu_{\varepsilon(p,q)}=\theta_{p} (\theta_{q}{\mid}_{p^{\downarrow}})=\theta_{p}\theta_{q}.$$ The other three  equalities can be proved dually.

(3) By the fact $a\leq \varepsilon(p,q)$, (O1) and (E2), we have $\dd(a)\leq \dd(\varepsilon(p,q))=\dd(p,q)=p$. Item (1) in the present lemma gives that $_{\dd(a)}{\downharpoonleft}\varepsilon(p,q)=\varepsilon(\dd(a),\dd(a)\theta_{q})$. On the other hand, by the fact that $a\leq \varepsilon(p,q)$ and Lemma \ref{b} (3), we obtain $a={_{\dd(a)}}{\downharpoonleft}\varepsilon(p,q)$. So $a=\varepsilon(\dd(a),\dd(a)\theta_{q})$. This implies that $$\rr(a)=\rr(\varepsilon(\dd(a),\dd(a)\theta_{q}))=\rr(\dd(a),\dd(a)\theta_{q})=\dd(a)\theta_{q}$$  by(E2).  Thus $a=\varepsilon(\dd(a),\rr(a))$.
\end{proof}
\begin{defn}\label{lihfg}Let $(C,\circ,\dd,\rr,\leq, P_C, \times,\star)$ be a projection ordered category, $b\in C$ and $(e,f)\in P_{C}\times P_{C}$. Then $(e,f)$ is called a {\em $b$-linked pair} if
\begin{equation}\label{xx}f=e\Theta_{b}\theta_{f},\,\,\,\,e=f\Delta_{b}\delta_{e}.
\end{equation}

\end{defn}
\begin{lemma}\label{ccc}Let $(C,\circ,\dd,\rr,\leq, P_C, \times,\star)$ be a projection ordered category, $b\in C$ and  $(e,f)$ be  a  $b$-linked pair. Denote
\begin{equation}\label{fff}
e_{1}=e\theta_{\dd(b)},e_{2}=f\Delta_{b},f_{1}=e\Theta_{b},f_{2}=f\delta_{\rr(b)}
\end{equation}
\begin{itemize}
\item[(1)]$e_{i}\leq \dd(b),f_{i}\leq \rr(b),\,i=1,2$.
\item[(2)]$e=e\times \dd(b),\, \rr(b)\star f=f,\, e\mathcal{F}_{P_{C}} e_{i},\, f_{i} \mathcal{F}_{P_{C}} f,\,i=1,2$.
\item[(3)]$_{e_{i}}{\downharpoonleft} b=b{\downharpoonright}_{f_{i}},\, i=1,2.$
\end{itemize}
\end{lemma}
\begin{proof}
(1) Since ${\rm im}\Delta_{b}={\rm im}\theta_{\dd(b)}\subseteq \dd(b)^{\downarrow}$ and  ${\rm im}\Theta_{b}={\rm im}\delta_{\rr(b)}=\rr(b)^{\downarrow}$, we have $e_{i}\leq \dd(b)$ and $f_{i}\leq \rr(b)$, $i=1, 2$.

(2) By (\ref{xx}) and (\ref{fff}), we have $e=f\Delta_{b}\delta_{e}=e_{2}\delta_{e}=e\times e_2.$
Using (\ref{xx}), (\ref{fff}) and (G1b), we can obtain $$e\theta_{e_{2}}=e\theta_{f\Delta_{b}}=e\Theta_{b}\theta_{f}\Delta_{b}=f\Delta_{b}=e_{2}.$$
This implies that $e\mathcal{F}_{P_{C}}e_{2}$. By the facts $e=e\times e_2, e_2\leq \dd(b)$ and Lemma \ref{y} (2), we get $e=e\times \dd(b)$. This together with (3) and (4) in Lemma \ref{jiben1} gives that $$e\times e_{1}=e\times e\theta_{\dd(a)}=e\times (e\star \dd(b))=e\times \dd(b)=e ,$$  $$e\star e_{1}=e\star e\theta_{\dd(b)}=e\star(e\star \dd(b))=(e\times \dd(b))\star(e\star \dd(b))=e\star \dd(b)=e\theta_{\dd(b)}=e_{1}.$$This shows that $e\mathcal{F}_{P_{C}}e_{1}$.
By (\ref{xx}) and (\ref{fff}), $f_{1}\star f=f_{1}\theta_{f}=e\Theta_{b}\theta_{f}=f$. Again by (\ref{xx}), (\ref{fff}) and (G1b), $$f_{1}\times f=f\delta_{f_{1}}=f\delta_{e\Theta_{b}}=f\Delta_{b}\delta_{e}\Theta_{b}=e\Theta_{b}=f_{1}.$$ Thus $f_{1}\mathcal{F}_{P_{C}}f$.
On the other hand, (\ref{fff}), (L4) and (L1) imply $$f_{2}\times f=f\delta_{\rr(b)}\times f=(\rr(b)\times f)\times f=(\rr(b)\times f)\times(\rr(b)\times f)=\rr(b)\times f=f\delta_{\rr(b)}=f_{2}.$$ By Lemma \ref{jiben1} (3), $f_{2}\star f=(\rr(b)\times f)\star f=\rr(b)\star f.$ Since $f_{1}\leq \rr(b)$, we have  $f=f_{1}\star f\leq \rr(b)\star f\leq f$ by (5) and (2) in Lemma \ref{jiben1}, and so $\rr(b)\star f=f=f_{2}\star f$. Thus $f_{2}\mathcal{F}_{P_{C}}f$.

(3) By  (\ref{xx}),\,(\ref{fff}) and Lemma \ref{f}, we have $e_{1}\nu_{b}=e\theta_{\dd(b)}\nu_{b}=e\Theta_{b}=f_{1}$ and $$e_{2}\nu_{b}=f\Delta_{b}\nu_{b}=f\delta_{\rr(b)}\mu_{b}\nu_{b}=f\delta_{\rr(b)}=f_{2}.$$ By Lemma \ref{b} (5), $_{e_{i}}{\downharpoonleft}b=b{\downharpoonright}_{\rr(_{e_{i}}{\downharpoonleft}b)}=b{\downharpoonright}_{e_{i}\nu_{b}}=b{\downharpoonright}_{f_{i}},\,i=1,2$.
\end{proof}
Let $(C,\circ,\dd,\rr,\leq, P_C, \times,\star)$ be a projection ordered category with an  evaluation map $\varepsilon:\mathscr{P}(P_{C})\rightarrow C $, $b\in C$, $(e,f)$ be a $b$-linked pair and $e_{1},e_{2},f_{1},f_{2}\in P_{C}$ be defined as in (\ref{fff}). By (1) and (2) in  Lemma \ref{b}, (E2), Lemmas \ref{oo} and \ref{ccc}, the following two elements are well defined:
\begin{equation}\label{nnn}
\lambda(e,b,f)=\varepsilon(e,e_{1})\circ{_{e_{1}}}{\downharpoonleft}b \circ \varepsilon(f_{1},f)=\varepsilon(e,e_{1})\circ b{\downharpoonright}_{f_{1}}\circ \varepsilon(f_{1},f),
\end{equation}
\begin{equation}\label{mmm}
\rho(e,b,f)=\varepsilon(e,e_{2})\circ {_{e_{2}}}{\downharpoonleft}b\circ \varepsilon(f_{2},f)=\varepsilon(e,e_{2})\circ b{\downharpoonright}_{f_{2}}\circ \varepsilon(f_{2},f).
\end{equation}

\begin{defn}\label{lll}
Let $(C,\circ,\dd,\rr,\leq, P_C, \times,\star)$ be a projection ordered category with an  evaluation map $\varepsilon:\mathscr{P}(P_{C})\rightarrow C $. Then $(C,\circ,\dd,\rr,\leq, P_C, \times,\star,\varepsilon)$ is called a {\em chain projection ordered category} if the following condition holds:
\begin{itemize}
\item[\rm (G2)] For all $b\in C$ and $b$-linked pair $(e,f)$, we have $\lambda(e,b,f)=\rho(e,b,f)$, where $\lambda(e,b,f)$ and $\rho(e,b,f)$ are defined by  (\ref{nnn}) and (\ref{mmm}) respectively.
\end{itemize}
Moreover, a chain projection ordered category $(C,\circ,\dd,\rr,\leq,P_C,\times,\star,\varepsilon)$ is called a
\begin{itemize}
\item  {\em chain projection ordered groupoid} if $(C,\circ,\dd,\rr)$ is a groupoid,
\item  {\em symmetric chain projection ordered category} if $(P_C, \times,\star)$ is a symmetric strong projection algebra, and for all $p,q\in P_C$, $p\, {\mathcal F}_{P_{C}}\, q$ implies that $\varepsilon(p,q,p)=p$,
\item  {\em symmetric chain projection ordered groupoid} if it is symmetric and $(C,\circ,\dd,\rr)$ is also a groupoid,
\item  {\em commutative chain projection ordered category} if  $(P_C, \times,\star)$ is a commutative strong projection algebra,
\item  {\em commutative chain projection ordered groupoid} if is commutative and $(C,\circ,\dd,\rr)$ is also a groupoid.
\end{itemize}
\end{defn}
\begin{defn}\label{pwdj}
Assume that $(C_{1},\circ,\dd,\rr,\leq, P_C\times,\star,\varepsilon_{1})$ and $(C_{2},\circ,\dd,\rr,\leq,P_C,\times,\star,\varepsilon_{2})$ are chain projection ordered categories and $\phi:C_1\rightarrow C_2,\,\,a\rightarrow a\phi$ be a map. Then $\phi$ is called {\em a chain projection ordered functor} if for all $a,b\in C_{1}$ and $e,f\in P_{C_{1}}$,  the following conditions hold:
\begin{itemize}
\item[\rm (F1)]$\dd(a\phi)=(\dd(a))\phi,\,\,\rr(a\phi)=(\rr(a))\phi$.
\item[\rm(F2)]If $a\circ b$ is defined, then $(a\circ b)\phi=(a\phi) \circ (b\phi)$.
\item[\rm(F3)]If $a\leq b$, then $a\phi \leq b\phi$.
\item[\rm(F4)]$(e\times f)\phi=e\phi \times f\phi,\,\,(e\star f)\phi=e\phi \star f\phi$.
\item[\rm(F5)]For all $(p_{1},p_{2},\ldots,,p_{k})\in \mathscr{P}(P_{C_{1}})$,  $$(\varepsilon_{1}((p_{1},p_{2},\ldots,,p_{k})))\phi=\varepsilon_{2}((p_{1}\phi,p_{2}\phi,\ldots,,p_{k}\phi)).$$
\end{itemize}
\end{defn}
\begin{lemma}\label{wangw}Let $\phi$ be a chain projection ordered functor between chain projection ordered groupoids $(C_{1},\circ,\dd,\rr,\leq,P_{C_1}, \times,\star,\varepsilon_{1})$ and $(C_{2},\circ,\dd,\rr,\leq,P_{C_2}\times,\star,\varepsilon_{2})$. Then $(x^{-1})\psi=(x\psi)^{-1}$ for all $x\in C_1$.
\end{lemma}
\begin{proof}
By the facts $x\circ x^{-1}=\dd(x)$ and $ x^{-1}\circ x=\rr(x)$ and (F1), (F2),
$$(x\psi)\circ (x^{-1}\psi)=(x\circ x^{-1})\psi=(\dd(x))\psi=\dd(x\psi),\,\,(x^{-1}\psi)\circ(x\psi)=\rr(x\psi).$$
This implies that $(x\psi)^{-1}=x^{-1}\psi$ by the statements before Lemma \ref{qunpei}.
\end{proof}
The following proposition is obvious.
\begin{prop}
The class of (respectively, symmetric, commutative) chain projection ordered categories (respectively, groupoids) together with chain projection ordered functors forms a category.
\end{prop}
\begin{lemma}\label{g}
Let $\phi$ be a chain projection ordered functor between chain projection ordered categories $(C_{1},\circ,\dd,\rr,\leq,P_{C_1}, \times,\star,\varepsilon_{1})$ and $(C_{2},\circ,\dd,\rr,\leq,P_{C_2}\times,\star,\varepsilon_{2})$, and $x\in C_{1}$, $e,f\in P_{C_{1}}$,$e\leq \dd(x),f\leq\rr(x)$. Then $({_e}{\downharpoonleft} x)\varphi={_{e\varphi}}{\downharpoonleft}(x\varphi),\, (x{\downharpoonright}{_f})\varphi=(x\varphi){\downharpoonright}{_{f\varphi}}.$
\end{lemma}
\begin{proof}(F1) gives $e\varphi=(\dd(e))\varphi=\dd(e\varphi)\in P_{C_{2}}$, and  (F3) gives $e\varphi\leq(\dd(x))\varphi=\dd(x\varphi)$. This implies that  ${_{e\varphi}}{\downharpoonleft} (x\varphi)$ is defined. By the fact that ${_e}{\downharpoonleft} x\leq x$ and (F3), we have $({_e}{\downharpoonleft} x)\varphi\leq x\varphi$, and by  the fact $\dd({_e}{\downharpoonleft} x)=e$ and (F1), we obtain   $\dd(({_e}{\downharpoonleft} x)\varphi)=(\dd({_e}{\downharpoonleft} x))\varphi=e\varphi.$
By Lemma \ref{b} (3) , we get $({_e}{\downharpoonleft} x)\varphi={_{e\varphi}}{\downharpoonleft}(x\varphi)$. Dually, $(x{\downharpoonright}{_f})\varphi=(x\varphi){\downharpoonright}{_{f\varphi}}$.
\end{proof}

\section{From DRC-restriction semigroups to chain projection ordered categories}
In this section, we present some  results on DRC-restriction semigroups and prove that each DRC-restriction semigroup give rises to a chain projection ordered category. We begin by giving the following lemma.

\begin{lemma}\label{aisdm} Let $(S,\cdot,\circ)$ be a generalized regular $\circ$-semigroup (respectively, a regular $\circ$-semigroup, an inverse semigroup) and denote $x^{+}=xx^{\circ}$ and $x^{\ast}=x^{\circ}x$ for all $x\in S$. Then $(S,\cdot, ^{+}, ^{\ast})$ is a DRC-restriction semigroup (respectively, $P$-restriction semigroup, restriction semigroup) and for all $x\in S$,
\begin{equation}\label{dlrio}
x^{+}=xx^{\circ},\,x^{\ast}=x^{\circ}x,\,x^{\ast}=x^{\circ +},\,x^{+}=x^{\circ \ast}.
\end{equation}
Conversely, let $(S,\cdot, ^{+}, ^{\ast})$ be a  DRC-restriction semigroup (respectively, $P$-restriction semigroup, restriction semigroup) and for each  $x\in S$,  there exists $x^{\circ}\in S$ such that
\begin{equation}\label{gtio}
x^{+}=xx^{\circ},\,x^{\ast}=x^{\circ}x,\,x^{\ast}=x^{\circ +},\,x^{+}=x^{\circ \ast}.
\end{equation}
Then the above $x^{\circ}$ is unique and $(S, \cdot, \circ)$ forms a generalized regular $\circ$-semigroup (respectively, a regular $\circ$-semigroup, an inverse semigroup).
\end{lemma}
\begin{proof} By \cite[Proposition 3.7]{Wang6}, under the given assumption, $(S,\cdot,^{+},^{\ast})$ is a DRC-restriction semigroup, and (\ref{dlrio}) follows from (\ref{g-regularstar}).
This gives the direct part. Conversely, let $(S,\cdot,^{+},^{\ast})$ be DRC-restriction semigroup satisfying  (\ref{gtio}). We first assert that the element $x^{\circ}$ in (\ref{gtio}) is unique. In fact, let $y\in S$ and  $x^{+}=xy,\,x^{\ast}=yx,\,x^{\ast}=y^{+},\,x^{+}=y^{\ast}.$ Then by (i) and (i)$^{'}$, $$y=y^{+}y=x^{\ast}y=x^{\circ}xy=x^{\circ}x^{+}=x^{\circ}x^{\circ\ast}=x^{\circ}.$$
Moreover, $xx^{\circ}x=x,\,x^{\circ}xx^{\circ}=x^{\ast}x^{\circ}=x^{\circ +}x^{\circ}=x^{\circ}.$
This implies that $x^{\ast}=x^{\circ}x \,\,{\mathcal L}\,\, x\,\, {\mathcal R}\,\, xx^{\circ}=x^{+},$ and so $x^{\circ}$ is the unique inverse element of $x$ satisfying  $xx^{\circ}=x^{+}$ and $x^{\ast}=x^{\circ}x$.  By \cite[Proposition 3.6]{Wang6}, $(S,\cdot,\circ)$ is a generalized regular $\circ$-semigroup.

 The other two special cases follow from \cite[Proposition 6.1]{Wang2}, (\ref{inversekehua}) and Proposition \ref{guanxi}, we omit the details.
\end{proof}
\begin{remark}\label{tongyi}By Lemma \ref{aisdm}, a generalized regular $\circ$-semigroup (respectively, a regular $\circ$-semigroup, an inverse semigroup) is exactly an algebraic system  $(S,\cdot, ^{+}, ^{\ast}, {\circ})$ where $(S,\cdot,^{+},^{\ast})$ is a DRC-restriction semigroup (respectively, $P$-restriction semigroup, restriction semigroup) and ``${\circ}$" is a unary operation on $S$ satisfying the axioms:
$$x^{+}=xx^{\circ},\,x^{\circ}x=x^{\ast},\,x^{\ast}=x^{\circ +},\,x^{\circ\ast}=x^{+}.$$
\end{remark}
The following lemmas provide some properties of projections in DRC semigroups.
\begin{lemma}[\cite{Wang4}]\label{kangwei}Let $(S, \cdot,\, {}^+, {}^\ast)$ be a DRC semigroup and $e,f\in P(S)$.
\begin{itemize}
\item[(1)] $e^+=e=e^\ast \in E(S)$.
\item[(2)] $(ef)^+=e(ef)^+e$, $(ef)^\ast=f(ef)^\ast f$.
\item[(3)]   $(ef)^+f=ef=e(ef)^\ast $.
\item[(4)]  $(ef)^+(ef)^\ast=ef$.
\end{itemize}
\end{lemma}
\begin{coro}[\cite{Wang4}]\label{zhangyu}In DRC Conditions, the axiom (\romannumeral3) (respectively, (\romannumeral3)$'$) can be replaced by the axiom $(x^+y^+)^+=x^+(x^+y^+)^+x^+\,\, ({respectively},\, (x^\ast y^\ast)^\ast=y^\ast(x^\ast y^\ast)^\ast y^\ast)$.
\end{coro}

\begin{lemma}[\cite{Jones1,Wang6}]\label{touyingdaishu}
Let $(S, \cdot,\, {}^+, {}^\ast)$ be a DRC semigroup. Define two binary operations  $``\times_{S}"$ and $``\star_{S}"$ on $P(S)$ as follows: For all $e,f\in P(S)$,
\begin{equation}\label{touyingdaishuyunsuan} e\times_S f=(ef)^+, e\star_S f=(ef)^\ast.
\end{equation}
Then $(P(S), \times_S, \star_S )$ forms a projection algebra and is called the  projection algebra of $S$. In particular,  if $S$ is DRC-restriction, then $P(S)$ is strong. Moreover, if $S$ is P-restriction, then $P(S)$ is symmetric; If $S$ is restriction, then $P(S)$ is commutative.
\end{lemma}
Let $(S, \cdot,\, {}^+, {}^\ast)$ be a DRC semigroup. Define a partial order $``\omega_S"$ on $P(S)$ as follows: For all $e,f\in P(S)$,  $$e \omega_S f \Longleftrightarrow fef=e.$$  By DRC Conditions and Lemma \ref{kangwei}, for all $x,y\in S$ and $e,f\in P (S)$,  we have
\begin{equation}\label{xiaoyu}
(xy)^+\,\omega_S x^+,\,\, (xy)^\ast\,\omega_S y^\ast,\,\, (ef)^+\,\omega_S e, \,\,  (ef)^\ast\omega_S f.
\end{equation}
Define binary relations $``\leq^r_S"$ and $``\leq^l_S"$ on $S$ as follows: For all $a,b\in S$,
\begin{equation}\label{xiaoyus} a\leq^r_S b \Longleftrightarrow a=a^+b,\,   a^+\omega_S b^+,
\end{equation}
\begin{equation}\label{xiaoyusl} a\leq^l_S b \Longleftrightarrow a=ba^\ast,\, a^\ast\omega_S b^\ast.
\end{equation}
By (\romannumeral1), (\romannumeral1)$'$ and Lemma \ref{kangwei}, we can obtain the following result easily.
\begin{lemma}\label{xiaoyuslw}Let $(S, \cdot,\, {}^+, {}^\ast)$ be a DRC semigroup. Then $``\leq^r_S"$ and $``\leq^l_S"$ are partial orders on $S$,  and for all $e,f\in P(S)$,
$$e\leq^r_S f \Longleftrightarrow  e \omega_S f\Longleftrightarrow e\leq^l_S f\Longleftrightarrow e\leq_{P(S)} f,
$$
where $\leq_{P(S)}$  is the natural partial order of the projection algebra $(P(S), \times_S, \star_S)$.
\end{lemma}
\begin{lemma}[\cite{Wang4}]\label{zhuw2}Let $(S, \cdot,\, {}^+, {}^\ast)$ be a DRC semigroup and $a,b\in S$.
\begin{itemize}
\item[(1)] If $a\leq^r_S b$ or $a\leq^l_S b$, then $a^+\omega b^+$, $a^\ast\omega b^\ast$.
\item[(2)] $a\leq^r_S b$ if and only if there exists $e\in P(S)$ such that $a=eb$ and $e\omega_S b^+$.
\item[(3)]  $a\leq^l_S b$ if and only if there exists $e\in P(S)$ such that $a=be$ and $e\omega_S b^\ast$.
\item[(4)] For all $x\in S$ and $e\in P(S)$, if $e\omega_S x^+$, then there exists a unique element $z\in S$ such that $z^+=e$ and $z\leq^l_S x$, this element is exactly $ex$.
\item[(5)]  For all $x\in S$ and $e\in P(S)$, if $e\omega_S x^\ast$, then there exists a unique element $z\in S$ such that $z^\ast=e, z\leq^r_S x$, this element is exactly $xe$.
\end{itemize}
\end{lemma}
\begin{lemma}\label{yizhi}Let $(S, \cdot,\, {}^+, {}^\ast)$ be a DRC semigroup. Then $\leq^r_S=\leq^l_S$ if and only if $S$ is DRC-restriction.
\end{lemma}
\begin{proof}
Let $\leq^{l}_{S}\,\subseteq\, \leq^{r}_S$ and $x,y\in S$. By (\ref{xiaoyu}), $(x^{\ast}y)^{+} \omega_{S}\, x^{\ast +}=x^{\ast}$, and so $x(x^{\ast}y)^{+}\leq^{l}_{S} x$ by Lemma \ref{zhuw2} (3). By hypothesis,
$x(x^{\ast}y)^{+}\leq^{r}_{S} x$, which implies that  $$x(x^{\ast}y)^{+}=(x(x^{\ast}y)^{+})^{+}x=(xx^{\ast }y)^{+}x=(xy)^{+}x $$ by (ii) and (i)$'$. This gives (vi)$'$.
Conversely, assume that (vi)$'$ is satisfied, $a,b\in S$ and $a\leq^l_S b$. Then $a=ba^{\ast}$ and $a^{\ast}\omega_{S} b^{\ast}$. Thus we have $b^{\ast}a^{\ast}=a^{\ast}$ and
$$a^{+}b=(ba^{\ast})^{+}b=b(b^{\ast}a^{\ast})^{+}=b(a^{\ast})^{+}=ba^{\ast}=a,\,\,\,  b^{+}a^{+}b^{+}=b^{+}(ba^{\ast})^{+}b^{+}=(ba^{\ast})^{+}=a^{+}$$ by DRC-conditions.
This implies that $a\leq^r_S b$. Therefore $\leq^l_S\, \subseteq \,\leq^r_S.$
The above statement together with its dual now gives the desired result.
\end{proof}
Let $(S, \cdot,\, {}^+, {}^\ast)$ be a DRC-restriction semigroup. By lemma \ref{yizhi}, $\leq^r_S=\leq^l_S$. In this case, denote $\leq_S=\leq^r_S=\leq^l_S$ and call it the {\em natural partial order} of $S$.
\begin{coro}\label{zhu2}Let $(S, \cdot,\, {}^+, {}^\ast)$ be a DRC-restriction semigroup and $a, b, c, d\in S$.
\begin{itemize}
\item[(1)] If $a\leq_S b$, then $a^+\leq_S b^+$ and $a^\ast\leq_S b^\ast$.
\item[(2)] $a\leq_S b$ if and only if there exists $e\in P(S)$ such that $a=eb$ and $e\omega_S b^+$.
\item[(3)]$a\leq_S b$ if and only if there exists $e\in P(S)$ such that $a=be$ and $e\omega_S b^\ast$.
\item[(4)]$a\leq_S b$ if and only if $a=a^{+}b=ba^{\ast}$.
\item[(5)]If $a\leq_S b, c\leq_S d$ and $ a^{\ast}=c^{+}, b^{\ast}=d^{+}$, then $ac\leq_S bd$.
\item[(6)] For all $x\in S$ and $e\in P(S)$, if $e\leq_S x^+$, then  there exists a unique element $z\in S$ such that $z^+=e$ and $z\leq_S x,$ this element is exactly $ex$. In this case, we denote this unique element by ${_e}{\downharpoonleft}{x}$, that is, ${_e}{\downharpoonleft}{x}=ex$.
\item[(7)]  For all $x\in S$ and $e\in P(S)$, if $e\leq_S x^\ast$, then  there exists a unique element $z\in S$ such that $z^\ast=e$ and $z\leq_S x$, this element is exactly $x e$. In this case, we denote this unique element by ${x}{\downharpoonright}_e$, that is, ${x}{\downharpoonright}_e=xe$.
\end{itemize}
\end{coro}
\begin{proof} By Lemma \ref{zhuw2}, we only need to prove (4) and (5).

(4) If $a\leq_S b$, then we have $a=a^{+}b=ba^{\ast}$ by (\ref{xiaoyus}), (\ref{xiaoyusl}) and Lemma \ref{yizhi}. Conversely, assume that $a=a^{+}b=ba^{\ast}$. Then $a^{+}=(ba^{\ast})^{+}=b^{+}(ba^{\ast})^{+}b^{+}=b^{+}a^{+}b^{+}$ by (iii), and so $a^{+}\omega_S b^{+}$. By  (\ref{xiaoyus}) and Lemma \ref{yizhi}, we have $a\leq_S b$ .

(5) Let $a\leq_{S} b$ and $c\leq_{S} d$. Then $a=a^{+}b,a^{+}\,\omega_{S}\,b^{+},c=c^{+}d,c^{+}\,\omega_{S}\,d^{+}$. By DRC Conditions and the hypothesis, $$(ac)^{+}=(ac^{+})^{+}=(aa^{\ast})^{+}=a^{+},~(bd)^{+}=(bd^{+})^{+}=(bb^{\ast})^{+}=b^{+},$$ and so $(ac)^{+}\,\omega_{S}\,(bd)^{+}.$ Moreover, $(ac)^{+}bd=a^{+}bd=ad=aa^{\ast}d=ac^{+}d=ac.$ Thus $ac\leq_{S} bd$.
\end{proof}

Let $(S, \cdot,\, {}^+, {}^\ast)$ be a DRC  semigroup. In $(P(S), \times_S, \star_S)$, denote
$$\langle e\rangle=\{x\in P(S)\mid x\leq_{P(S)} e\}=\{x\in P(S)\mid x\, \omega_{S}\, e\}$$ for any $e\in P(S)$ (see Lemma \ref{xiaoyuslw}). By (\ref{xiaoyu}), we can define
\begin{equation}\label{rho}\rho_a: \langle a^+\rangle\rightarrow \langle a^\ast\rangle,\,\, x\mapsto (xa)^\ast;\,\,\,\,
\sigma_a: \langle a^\ast\rangle\rightarrow \langle a^+\rangle,\,\, y\mapsto (ay)^+
\end{equation} for all $a\in S$.
In particular, by Lemma \ref{kangwei} we have
\begin{equation}\label{hengdeng}
\rho_e =\sigma_e=\mbox{the identity map on }\langle e\rangle
\end{equation}
for all $e\in P(S)$.
\begin{lemma}[\cite{Wang6}]\label{huni}Let $(S, \cdot,\, {}^+, {}^\ast)$ be a DRC  semigroup. Then $S$ is DRC-restriction if and only if for all $a\in S$,  $\rho_a$ and  $\sigma_a$ are mutually inverse bijections. In this case, for all $a\in S$,  $\rho_a$ and  $\sigma_a$ are mutually inverse isomorphisms between projection algebras.
\end{lemma}

Let $(S, \cdot,\, {}^+, {}^\ast)$ be a DRC-restriction  semigroup. Define the {\em restricted product $``\circ_S"$} of $S$ as follows:
$$a\circ_S b=\left\{\begin{array}{cc}
ab
& \mbox{ if}\ a^\ast=b^+ ,\\
\mbox{undefined} & \mbox{otherwise},
\end{array} \right.
$$
where $ab$ is the multiplication of $a$ and $b$ in $S$. Moreover, define the following maps:
\begin{equation}\label{qqyyaa}
\dd_S: S \rightarrow S,\, x\rightarrow x^{+},\,\,  \rr_S: S\rightarrow S,\, x\rightarrow x^{\ast}.
\end{equation}
\begin{equation}\label{hui}
\varepsilon_S: \mathscr{P}({P(S)})\rightarrow S, \,\, (p_{1},p_{2},\ldots,p_{k})\mapsto p_{1}p_2\cdots p_{k}.
\end{equation}
where $\mathscr{P}({P(S)})$ is the path category of $(P(S), \times_S, \star_S)$.
Now we can state our main result in this section.
\begin{theorem}\label{fii}
Let $(S, \cdot,\, {}^+, {}^\ast)$ be a DRC-restriction  semigroup. Then $$\textbf{\rm \bf C}(S)=(S, \circ_S, \dd_S, \rr_S, \leq_S,P_{S}, \times_S, \star_S, \varepsilon_S)$$ is a chain projection ordered category, where $P_{S}=P(S)$. Moreover, we have the following results:
\begin{itemize}
\item  If $S$ is a generalized regular $\circ$-semigroup, then $\textbf{\rm \bf C}(S)$ is a chain projection ordered  groupoid.
\item  If $S$ is P-restriction, then $\textbf{\rm \bf C}(S)$ is symmetric  chain projection ordered category.
\item   If $S$ is a  regular $\circ$-semigroup,  then $\textbf{\rm \bf C}(S)$ is a symmetric chain projection ordered  groupoid.
\item   If $S$ is restriction,  Then $\textbf{\rm \bf C}(S)$ is a commutative chain projection ordered category.
\item   If $S$ is an inverse  semigroup,  then $\textbf{\rm \bf C}(S)$ is a commutative chain projection ordered groupoid.
\end{itemize}
\end{theorem}

\begin{proof}
Let $x,y,z\in S$. By the definitions of $\circ_{S},\dd_{S}$ and $\rr_{S}$,  $x\circ_{S} y$ is defined if and only if  $x^{\ast}=y^{+}$, if and only if $\rr_{S}(x)=\dd_{S}(y).$ In this case, $$\dd_{S}(x\circ_{S} y)=(xy)^{+}=(xy^{+})^{+}=(xx^{\ast})^{+}=x^{+}=\dd_{S}(x).$$ Dually, $\rr_{S}(x\circ_{S}\, y)=\rr_{S}(y)$. If $x\circ_{S} y$ and $y\circ_{S}\, z$ are defined, by the above discussions $(x\circ_{S} y)\circ_{S} z$ and $x\circ_{S} (y\circ_{S} z)$ are defined and $$(x\circ_{S} y)\circ_{S} z=(xy)z=x(yz)=x\circ_{S} (y\circ_{S} z).$$ Since $\rr_{S}(\dd_{S}(x))=\rr_{S}(x^{+})=(x^{+})^{\ast}=x^{+}=\dd_{S}(x)$, it follows that $\dd(x)\circ_{S} x$ is defined and $\dd_{S}(x)\circ_{S} x=\dd_{S}(x)x=x^{+}x=x$. Dually, $x\circ_{S} \rr_{S}(x)$ is defined and $x\circ_{S} \rr_{S}(x)=x$.
Thus $C=(S,\circ_{S},\dd_{S},\rr_{S})$ is a category whose set of objects is $$P_{S}=\{\dd_{S}(x)\mid x\in S\}=\{x^{+}\mid x\in S\}=P(S).$$
By Lemma \ref{zhu2}, $(S,\circ_{S},\dd_{S},\rr_{S},\leq_{S})$ is an ordered category, and for all $x\in S$ and $e,f\in P_{S}=P(S)$, if $e\leq_{S} \dd_{S}(x)=x^{+}$ and $f\leq_{S} \rr_{S}(x)=x^{\ast}$, then ${_e}{\downharpoonleft}x=ex$ and $x{\downharpoonright}{_f}=xf$. By Lemma \ref{touyingdaishu}, $(P_{S},\times_{S},\star_{S})$ is a strong projection algebra. Moreover, $e \leq_{S} f$ if and only if $e \leq_{P_{S}} f$ for all $e, f\in P_{S}$ by Lemma \ref{xiaoyuslw}.  This shows that $(S,\circ_{S},\dd_{S},\rr_{S},\leq_{S}, P_S, \times_S, \star_S)$ forms a weak projection ordered category.
Let $x\in S$.  Then in $(S,\circ_{S},\dd_{S},\rr_{S},\leq_{S}, P_S, \times_S, \star_S)$,
$$\dd_{S}(x)^{\downarrow}=\{e\in P_{S}\mid e\leq_{S} \dd_{S}(x)\},\,\,\rr_{S}(x)^{\downarrow}=\{f\in P_{S}\mid f\leq_{S} \rr_{S}(x)\},$$
\begin{equation}\label{qqyyee}
\nu_{x}:(x^{+})^{\downarrow}=\dd_{S}(x)^{\downarrow}\rightarrow \rr_{S}(x)^{\downarrow}=(x^{\ast})^{\downarrow},\,\,\,p\rightarrow \rr_{S}({_p}{\downharpoonleft} x)=(px)^{\ast},
\end{equation}
\begin{equation}\label{qqyyvv}
\mu_{x}:(x^{\ast})^{\downarrow}=\rr_{S}(x)^{\downarrow}\rightarrow \dd_{S}(x)^{\downarrow}=(x^{+})^{\downarrow},\,\,\,q\rightarrow \dd_{S}(x{\downharpoonright} {_q})=(xq)^{+}.
\end{equation}
By (\ref{rho}), we have
\begin{equation}\label{bope}
\nu_{x}=\rho_{x}\,\, \mbox{and }\,\, \mu_{x}=\sigma_{x}
\end{equation}
Lemma \ref{huni} gives that $\mu_{x}$ and $\nu_{x}$ are mutually inverse isomorphisms between projection algebras.  Thus, $(S,\circ_{S},\dd_{S},\rr_{S},\leq_{S}, P_S, \times_{S},\star_{S})$ is a   projection ordered category. Define
$$\varepsilon_S:\mathscr{P}(P_{S})\rightarrow S,\,(p_{1},p_{2},\ldots,p_{k})\rightarrow p_{1}p_{2}\ldots p_{k}.$$
In the followings, we shall prove that $\varepsilon_{S}$ is an evaluation map. Let $\mathfrak{p}=(p_{1},p_{2},\ldots,p_{k}),\mathfrak{q}=(q_{1},q_{2},\ldots,q_{l})\in \mathscr{P}(P_{S})$. Then $$p_{1},p_{2},\ldots,p_{k},q_{1}\\[2mm],q_{2},\ldots,q_{l}\in P_{S},\, p_{1}\mathcal{F}_{P_{S}}\,p_{2}\,\mathcal{F}_{P_{S}}\cdots \mathcal{F}_{P_{S}}\,p_{k},\,\,q_{1}\,\mathcal{F}_{P_{S}}\,q_{2}\mathcal{F}_{P_{S}}\,\cdots \mathcal{F}_{P_{S}}\,q_{l}.$$ Firstly, let $p\in P_{S}$. Then $\varepsilon_{S}(p)=p$ and $\varepsilon_{S}(p,p)=pp=p$, and so (E1) holds. Secondly, $$\dd_{S}(\varepsilon_{S}(\mathfrak{p}))=\dd_{S}(p_{1}p_{2}\ldots p_{k})=(p_{1}p_{2}\ldots p_{k})^{+}=(p_{1}(p_{2}\ldots p_{k})^{+})^{+}$$$$=p_{1}\times_{S} (p_{2}\ldots p_{k})^{+}= (p_{2}\ldots p_{k})^{+}\delta_{p_{1}}=\ldots =p_{k}\delta_{p_{k-1}}\ldots\delta_{p_{1}}=p_{1}=\dd((p_{1},p_{2},\ldots,p_{k}))$$ by (ii) and Lemma \ref{zjiy}. Dually,  $\rr_{S}(\varepsilon_{S}(\mathfrak{p}))=p_{k}=\rr((p_{1},p_{2},\ldots,p_{k}))$. This gives (E2).
If $\rr(\mathfrak{p})=\dd(\mathfrak{q})$, i.e. $p_{k}=q_{1}$, then $\mathfrak{p}\circ \mathfrak{q}=(p_{1},p_{2},\ldots p_{k-1},p_{k}=q_{1},q_{2},\ldots ,q_{l})$. Since $p_{k}q_{1}=p_{k}p_{k}=p_{k}$, we have
$$\varepsilon_{S}(\mathfrak{p}\circ \mathfrak{q})=p_{1}p_{2}\ldots p_{k-1}p_{k}q_{2}\ldots q_{l}=(p_{1}p_{2}\ldots p_{k-1}p_{k})(q_{1}q_{2}\ldots q_{l})$$$$=\varepsilon_{S}((p_{1},p_{2},\ldots p_{k-1},p_{k}))\circ_{S} \varepsilon_{S}((q_{1},q_{2},\ldots,q_{l})).$$
This implies  (E3) is true. Let $q\in P_{S}$, $q\leq_{S}\dd(\mathfrak{p})=p_{1},$ and ${_q}{\downharpoonleft} \mathfrak{p}=(s_{1},s_{2},\ldots ,s_{k})$, where $$s_{i}=q\theta_{p_{1}}\theta_{p_{2}}\ldots\theta_{p_{i}}=(qp_{1})^{\ast}\theta_{p_{2}}\ldots\theta_{p_{i}} =((qp_{1})^{\ast}p_{2})^{\ast}\theta_{p_{3}}\ldots\theta_{p_{i}}=(qp_{1}p_{2})^{\ast}\theta_{p_{3}}\ldots\theta_{p_{i}} $$$$=\ldots=(qp_{1}p_{2}\ldots p_{i})^{\ast}=p_{i}^{\ast}(qp_{1}p_{2}\ldots p_{i})^{\ast}=p_{i}(qp_{1}p_{2}\ldots p_{i})^{\ast},i=1,2,\ldots,k.$$
If $k=1$, then ${_q}{\downharpoonleft} \mathfrak{p}=q$. Since $q\leq_{S}\dd(\mathfrak{p})=p_{1}$, we have $q=qp_{1}$, and so $$\varepsilon_{S}({_q}{\downharpoonleft} \mathfrak{p})=\varepsilon_{S}(q)=q=qp_{1}={_q}{\downharpoonleft} p_{1}= {_q}{\downharpoonleft}\varepsilon_{S}(p_{1})={_q}{\downharpoonleft}\varepsilon_{S}(\mathfrak{p}).$$
Denote $\mathfrak{t}=(p_{1},p_{2},\ldots ,p_{k-1})$.  Then ${_q}{\downharpoonleft} \mathfrak{t}=(s_{1},s_{2},\ldots ,s_{k-1})$. Assume that $\varepsilon_{S}({_q}{\downharpoonleft} \mathfrak{t})={_q}{\downharpoonleft}\varepsilon_{S}(\mathfrak{t})$. Then $s_{1}s_{2}\ldots s_{k-1}=qp_{1}p_{2}\ldots p_{k-1}$. Thus $$\varepsilon_{S}({_q}{\downharpoonleft} \mathfrak{p})=s_{1}s_{2}\ldots s_{k-1} s_{k}$$$$=qp_{1}p_{2}\ldots p_{k-1}p_{k}(qp_{1}p_{2}\ldots p_{k-1}p_{k})^{\ast}=qp_{1}p_{2}\ldots p_{k-1}p_{k}={_q}{\downharpoonleft}\varepsilon_{S}(\mathfrak{p}).$$
By mathematical induction, for all $q\in P_{S}$ and $\mathfrak{p}\in \mathscr{P}(P_{S})$, $q\leq_{S} \dd_{S}(\mathfrak{p})$ implies that $\varepsilon_{S}({_q}{\downharpoonleft} \mathfrak{p})={_q}{\downharpoonleft}\varepsilon_{S}(\mathfrak{p})$. This shows that (E5) is true, and so (E4) is also true by Proposition \ref{fder}. Thus $\varepsilon_{S}$ is an evaluation map.

Now let $a\in S$. Then by (\ref{aa}), (\ref{bope}) and (\ref{rho}),
\begin{equation}\label{yyqqaa}
p\Theta_{a}=p\theta_{\dd_{S}(a)}\nu_a=p(\theta_{a^{+}}\nu_a)=(p\star_{S} a^{+})\nu_a
=((pa^+)^\ast a)^{\ast}=(pa^+a)^{\ast}=(pa)^{\ast}
\end{equation}
for all $p\in P_{S}$. Dually,
\begin{equation}\label{yyqqee}
p\Delta_{a}=(ap)^{+} \mbox{ for all } p\in P_{S}.
\end{equation}
Let $b\in S,e,f\in P_{S}$ and $(e,f)$ be a $b$-linked pair. Then
$\dd_{S}(b)=b^+$ and $\rr_{S}(b)=b^\ast $. By (\ref{fff}), $$e_{1}=e\theta_{\dd_{S}(b)}=e\star_{S}b^+ =(eb^+)^\ast,\,\,e_{2}=f\Delta_{b}=(bf)^+,$$$$f_{1}=e\Theta_{b}=(eb)^\ast,\,\ f_{2}=f\delta_{\rr_{S}(b)}=b^\ast \times_{S} f=(b^\ast f)^{+}.$$ According to (\ref{nnn}), Lemma \ref{kangwei} and DRC Conditions,
$$\lambda(e,b,f)=\varepsilon_{S}(e,e_{1})\circ_{S} {_{e_{1}}}{\downharpoonleft} b \circ_{S} \varepsilon_{S}(f_{1},f)=ee_{1}e_{1}bf_{1}f$$$$=ee_{1}bf_{1}f=e(eb^+)^{\ast} b(eb)^{\ast}f=eb^+(eb^+)^{\ast}b^+b(eb)^{\ast}f=eb^+b(eb)^{\ast}f=(eb)(eb)^{\ast}f=ebf.$$
Dually, we can prove that $\rho(e,b,f)=ebf$. Thus $\lambda(e,b,f)=\rho(e,b,f)$. We have shown that $\textbf{\rm \bf C}(S)$ is a chain projection ordered category.

Finally, if $(S, \cdot,\, {}^+, {}^\ast, \circ)$ is a generalized regular $\circ$-semigroup, then $\dd_S(x)=x^{+}=x\circ_S x^\circ$ and $\rr(x)=x^{\ast}=x^\circ \circ_S x$ by (\ref{dlrio}).
This shows that $(S, \circ_S, \dd_S, \rr_S, ^\circ)$ is a groupoid, and so $\textbf{\rm \bf C}(S)$ is a chain projection ordered groupoid.
The other cases can be proved by using Lemmas \ref{aisdm} and \ref{touyingdaishu}, we omit the details.
\end{proof}
\begin{lemma}\label{zhang4}Let $(S_1, \cdot,\, ^+,  ^\ast)$ and $(S_2, \cdot,\, ^+,  ^\ast)$ be two DRC-restriction semigroups  and $\theta$ be a (2,1,1)-homomorphism from $S_1$ to $S_2$. Then the rule $\textbf{\rm \bf C}(\theta): S_1 \rightarrow  S_2,\,\,x\rightarrow x\theta$ provides a chain projection ordered functor from $\textbf{\rm \bf C}( S_1)$ to $\textbf{\rm \bf C}(S_2)$.
\end{lemma}
\begin{proof}
Firstly, let $x,y\in S_1$. Then $$\dd(x\theta)=(x\theta)^+=x^+\theta=(\dd(x))\theta,\,\, \rr(x\theta)=(x\theta)^\ast=x^\ast \theta=(\rr(x))\theta.$$
This gives (F1).  If $x\circ y$ is defined, then $x^{\ast}=\rr(x)=\dd(y)=y^+$, and so $(x\theta)^{\ast}=x^{\ast}\theta=y^+\theta=(y\theta)^+$. This implies  $(x\theta)\circ(y\theta)$ is defined and $$(x\circ y)\theta=(xy)\theta=(x\theta)(y\theta)=(x\theta)\circ(y\theta).$$ Thus (F2) holds. Let $x\leq_{S_1} y$. Then $x=x^+y=yx^{\ast}$ by Corollary \ref{zhu2} (4), and so $x\theta=(x\theta)^+(y\theta)=(y\theta)(x\theta)^{\ast}$. This shows that $x\theta\leq_{S_2} y\theta$. So (F3) holds. Let $e,f\in P_{S_1}$. Then $$(e\times_{S_1} f)\theta=(ef)^+\theta=((e\theta)(f\theta))^+=(e\theta)\times_{ S_2}(f\theta).$$ Dually, $(e\star_{S_1} f)\theta=(e\theta)\star_{ S_2}(f\theta)$. Thus (F4) is true. Finally, let $\mathfrak{p}=(p_{1},p_{2},\ldots,p_{k})\in \mathscr{P}(P_{S_{1}})$. Then $$(\varepsilon_1(\mathfrak{p}))\theta=(p_{1}p_{2}\ldots p_{k})\theta=(p_{1}\theta)(p_{2}\theta)\ldots(p_{k}\theta)=\varepsilon_2((p_{1}\theta,p_{2}\theta,\ldots,p_{k}\theta)).$$
Thus (F5) holds. The above discussion tells us that $\textbf{\rm\bf C}(\theta)$ is a chain projection functor.
\end{proof}


\section{The category isomorphism}
In this section, we first show how to construct  a  DRC-restriction semigroup from a chain projection ordered category, and  then  prove that the category  of DRC-restriction semigroups together with (2,1,1)-homomorphisms is isomorphic to the category of chain projection ordered categories together with chain projection ordered functors.

Let $(C,\circ,\dd,\rr,\leq,P_{C},\times,\star,\varepsilon)$ be a chain projection ordered category and $a,b\in C$.  By Lemma \ref{p}, $$(\rr(a)\times \dd(b),\rr(a)\star \dd(b))\in \mathcal{F}_{P_{C}},\,\, \rr(a)\times \dd(b)\leq \rr(a),\rr(a)\star \dd(b)\leq \dd(b).$$ Thus $a{\downharpoonright}_{\rr(a)\times \dd(b)},$ $\varepsilon(\rr(a)\times \dd(b),\rr(a)\star \dd(b) )$ and $_{\rr(a)\star \dd(b)}{\downharpoonleft} b$ are all defined. By (1)and (2) in  Lemma \ref{b} and Lemma \ref{oo}, we can define
\begin{equation}\label{kkk}
a\bullet b=a{\downharpoonright}_{\rr(a)\times \dd(b)} \circ~ \varepsilon(\rr(a)\times \dd(b),\rr(a)\star \dd(b) )  \circ~ _{\rr(a)\star \dd(b)}{\downharpoonleft} b.
\end{equation}

\begin{lemma}\label{vvvv}Let $(C,\circ,\dd,\rr,\leq, P_C, \times,\star,\varepsilon)$ be a chain projection ordered category and $a,b,c\in C$.
\begin{itemize}
\item[(1)]If $\rr(a)=\dd(b)$, then $a\bullet b=a\circ b$.
\item[(2)]$\rr(a\bullet b)=\rr(a)\Theta_{b}$ and $\dd(a\bullet b)=\dd(b)\Delta_{a}$.
\item[(3)]$\Theta_{a\bullet b}=\Theta_{a}\Theta_{b}$ and $\Delta_{a\bullet b}=\Delta_{b}\Delta_{a}$.
\item[(4)]$\rr((a\bullet b)\bullet c)=\rr(a)\Theta_{b}\Theta_{c}=\rr(a\bullet(b\bullet c))$ and $ \dd(a\bullet(b\bullet c))=\dd(c)\Delta_{b}\Delta_{a}=\dd((a\bullet b)\bullet c)$.
\end{itemize}
\end{lemma}
\begin{proof}Denote $p=\rr(a), q=\dd(b), p{'}=q\delta_{p}=p\times q,  q{'}=p\theta_{q}=p\star q.$

(1) If $p=q$, then by (L1) and (R1), we have $p{'}=q{'}=p$,  and so $$a\bullet b=a{\downharpoonright}{_p} \circ \varepsilon(p,p) \circ  {_p}{\downharpoonleft} b=a\circ p \circ b=(a\circ \rr(a))\circ b=a\circ b.$$
by (3) and (4) in  Lemma \ref{b}, (E1) and (C3).

(2) By (\ref{kkk}), we have $$\rr(a\bullet b)=\rr(_{q'}{\downharpoonleft} b)=q'\nu_{b}=p\theta_{q}\nu_{b}=\rr(a)\theta_{\dd(b)}\nu_{b}=\rr(a)\Theta_{b}.$$
Dually, $\dd(a\bullet b)=\dd(b)\Delta_{a}$.

(3) Let $t\in P_{C}$. By (R4), (P3), (R4) and (R3),
$$t\theta_{p'}\theta_{q'}=(t\star p')\star q'=t\theta_{q\delta_{p}}\theta_{p\theta_{q}}=(t\star(p\times q))\star(p\star q)=((t\star(p\times q ))\star q)\star(p\star q)$$$$=((t\star p)\star q))\star(p\star q)=(t\star p)\star(p\star q)=(t\star p)\star q=t\theta_p\theta_q.$$ This implies that $\theta_{p'}\theta_{q'}=\theta_p\theta_q$.
Using (\ref{kkk}), Lemma \ref{dd} (6), (G1c), Lemma \ref{dd} (5), Lemma \ref{oo} , Lemma \ref{o}, Lemma \ref{dd} (5) and the fact  $\theta_{p'}\theta_{q'}=\theta_p\theta_q$ in order, we have
$$\Theta_{a\bullet b}=\Theta_{a{\downharpoonright}_{p'}} \Theta_{\varepsilon(p',q')} \Theta_{_{q'}{\downharpoonleft}b} =\Theta_{a}\theta_{p'} \theta_{p'}\theta_{q'} \theta_{q'}\Theta_{b}$$$$ =\Theta_{a}\theta_{p'}\theta_{q'} \Theta_{b} =\Theta_{a}\theta_{p}\theta_{q}\Theta_{b} =\Theta_{a}\theta_{\rr(a)} \theta_{\dd(b)}\Theta_{b} =\Theta_{a}\Theta_{b}.$$
Dually, $\Delta_{a\bullet b}=\Delta_{b}\Delta_{a}$.

(4) By (2) and (3) of the present lemma, we have $$\rr((a\bullet b)\bullet c)=\rr(a\bullet b)\Theta_{c}=\rr(a)\Theta_{b}\Theta_{c}=\rr(a)\Theta_{b\bullet c}=\rr(a\bullet(b\bullet c)).$$
Dually, $\dd(a\bullet(b\bullet c))=\dd(c)\Delta_{b}\Delta_{a}=\dd((a\bullet b)\bullet c)$.
\end{proof}
\begin{lemma}\label{ccccc}Let $(C,\circ,\dd,\rr,\leq, P_C, \times,\star,\varepsilon)$ be a chain projection ordered category and $a,b,c\in C$, and denote$$p=\rr(a),q=\dd(b),r=\rr(b),s=\dd(c).$$
\begin{itemize}
\item[(1)] $(a\bullet b)\bullet c=a{\downharpoonright}{_e}\circ \varepsilon(e,e_{1})\circ { _{e_1}}{\downharpoonleft}b\circ \varepsilon(f_{1},f)\circ {_f}{\downharpoonleft} c$, where
$$e=s\Delta_{b}\delta_{p},\, e_{1}=e\theta_{q},\, f_{1}=e\Theta_{b},\, f=p\Theta_{b}\theta_{s}.$$
\item[(2)] $a\bullet(b\bullet c)=a{\downharpoonright}{_e} \circ \varepsilon(e,e_{2}) \circ b{\downharpoonright}_{f_{2}}\circ \varepsilon(f_{2},f)\circ {_f}{\downharpoonleft} c$, where $$e=s\Delta_{b}\delta_{p},\,  e_{2}=f\Delta_{b},f_{2}=f\delta_{r},\,f=p\Theta_{b}\theta_{s}.$$
\item[(3)] $(a\bullet b)\bullet c=a\bullet(b\bullet c)$. This implies that $(C, \bullet)$ forms a semigroup.
\end{itemize}
\end{lemma}
\begin{proof} (1) Denote $p'=q\delta_{p}=p\times q, q'=p\theta_{q}=p\star q$. Then
$$a\bullet b=a'\circ \varepsilon'\circ b', a'=a{\downharpoonright}_{p'},\varepsilon'=\varepsilon(p',q'),b'={_{q'}}{\downharpoonleft} b.$$
By Lemma \ref{vvvv} (2), we can let $t=\rr(a\bullet b)=\rr(b')=\rr({_{q'}}{\downharpoonleft} b)=p\Theta_{b}$. Then
$$(a\bullet b)\bullet c= (a\bullet b){\downharpoonright}{_{t'}} \circ \varepsilon(t',s')\circ {_{s'}}{\downharpoonleft} c,\, t'=s\delta_{t}=t\times s,s'=t\theta_{s}=t\star s.$$
Denote $\nu=\dd(b'{\downharpoonright}{_{t'}})$ and $\mu=\dd(\varepsilon'{\downharpoonright}{_\nu})$. By Lemma \ref{b} (8),
$$(a\bullet b){\downharpoonright}{_{t'}}=(a'\circ \varepsilon'\circ b'){\downharpoonright}{_{t'}}=a{'}{\downharpoonright}{_\mu}
 ~\circ~ ~\varepsilon'{\downharpoonright}{_\nu}~ \circ~ b'{\downharpoonright}{_{t'}}.$$
 By Lemma \ref{b} (2), (E2) and Lemma \ref{jiben1} (1),$$\mu=\dd(\varepsilon'{\downharpoonright}{_\nu})\leq \dd(\varepsilon')=\dd(\varepsilon(p',q'))=\dd(p',q')=p'=p\times q\leq p,$$ and so we have $a'{\downharpoonright}{_\mu}=(a{\downharpoonright}{_{p'}}){\downharpoonright}{_\mu}=a{\downharpoonright}{_\mu}$ by Lemma \ref{b} (10). In view of  Lemma \ref{b} (2), we obtain $$\varepsilon'{\downharpoonright}{_\nu}\leq \varepsilon'=\varepsilon(p',q'),\,\,\dd(\varepsilon'{\downharpoonright}{_\nu})=\rr(a'{\downharpoonright}{_\mu})=\mu,\,\, \rr(\varepsilon'{\downharpoonright}{_\nu})=\nu.$$  By Lemma \ref{oo} (3), $\varepsilon'{\downharpoonright}{_\nu}=\varepsilon(\mu,\nu)$. Moreover, since
$b'{\downharpoonright}{_{t'}}\leq b'={_{q'}}{\downharpoonleft} b\leq b$ and $\rr(b'{\downharpoonright}{_{t'}})=t'$, we have $b'{\downharpoonright}{_{t'}}=b{\downharpoonright}{_{t'}}$ by Lemma \ref{b} (3). This gives that $$b'{\downharpoonright}{_{t'}}=b{\downharpoonright}{_{t'}}={_{\dd(b{\downharpoonright}{_{t'}})}}{\downharpoonleft} b={_{\dd(b'{\downharpoonright}{_{t'}})}}{\downharpoonleft} b={_\nu}{\downharpoonleft} b$$ by Lemma \ref{b} (6).
Thus $$(a\bullet b)\bullet c= (a\bullet b){\downharpoonright}{_{t'}}\circ \varepsilon(t',s')\circ {_{s'}}{\downharpoonleft} c$$$$=a'{\downharpoonright}{_\mu}\circ \varepsilon'{\downharpoonright}{_\nu}\circ b'{\downharpoonright}{_{t'}}\circ \varepsilon(t',s')\circ{ _{s'}}{\downharpoonleft} c =a{\downharpoonright}{_\mu}\circ \varepsilon(\mu,\nu)\circ {_\nu}{\downharpoonleft} b\circ \varepsilon(t',s')\circ {_{s'}}{\downharpoonleft} c.$$
In the sequel, we shall prove the followings:
\begin{center} (a) $\mu=e$, \hspace{1cm} (b) $\nu=e_{1}$, \hspace{1cm}  (c) $t'=f_{1}$,  \hspace{1cm}     (d) $s'=f$.
\end{center}

\noindent (a) By Lemma \ref{vvvv} (4), we have $\dd(a{\downharpoonright}{_\mu})=\dd((a\bullet b)\bullet c)=s\Delta_{b}\Delta_{a}$. Since $a{\downharpoonright}{_\mu}\leq a$, it follows that $a{\downharpoonright}{_\mu}={_{s\Delta_{b}\Delta_{a}}}{\downharpoonleft} a$ by Lemma \ref{b} (3), and so $$\mu=\rr(a{\downharpoonright}{_\mu})=\rr({_{s\Delta_{b}\Delta_{a}}}{\downharpoonleft} a)=s\Delta_{b}\Delta_{a}\nu_{a}=s\Delta_{b}\delta_{p}\mu_{a}\nu_{a}=s\Delta_{b}\delta_{p}=e$$ by Lemma \ref{b} (2) and Lemma \ref{f}.

\noindent(c)  By (G1b), $t'=s\delta_{t}=s\delta_{p\Theta_{b}}=s\Delta_{b}\delta_{p}\Theta_{b}=e\Theta_{b}=f_{1}$.

\noindent(b) Since ${_\nu}{\downharpoonleft} b=b{\downharpoonright}{_{t'}}$, we have
$$\nu=\dd({_\nu}{\downharpoonleft} b)=\dd(b{\downharpoonright}{_{t'}})=t'\mu_{b}=f_{1}\mu_{b}=(e\Theta_{b})\mu_{b}=e\Theta_{\dd(b)}=e\theta_{q}=e_{1} $$ by Lemma \ref{b} (1), (c) and
Lemma \ref{dd} (3).

\noindent(d) $s'=t\theta_{s}=p\Theta_{b}\theta_{s}=f.$

(2) This is the dual of (1).

(3) By (1) and (2), we only need to prove
\begin{equation}\label{xxx}
\varepsilon(e,e_{1})\circ {_{e_1}}{\downharpoonleft} b\circ \varepsilon(f_{1},f)=\varepsilon(e,e_{2}) \circ b{\downharpoonright}{_{f_{2}}}\circ \varepsilon(f_{2},f).
\end{equation}
By the fact $f\leq s$ and Lemma \ref{o}, we have $\theta_{f}=\theta_{s}\theta_{f}$. By Lemma \ref{nn},
$$e\Theta_{b}\theta_{f}=s\Delta_{b}\delta_{p}\Theta_{b}\theta_{f} =s\Delta_{b}\delta_{p}\Theta_{b}\theta_{s}\theta_{f}
 =p\Theta_{b}\theta_{s}\theta_{f} =p\Theta_{b}\theta_{f} =f.$$
Dually, $e=f\Delta_{b}\delta_{e}$. This implies that $(e,f)$ is a $b$-lined pair. By (G2),  (\ref{xxx}) follows.
\end{proof}
\begin{lemma}\label{edd}
Let $(C,\circ,\dd,\rr,\leq, P_C, \times,\star,\varepsilon)$ be a chain projection ordered category, $a\in C$ and $s,t\in P_{C}$.
\begin{itemize}
\item[(1)] $a\,\bullet\,t=a{\downharpoonright}{_{\rr(a)\times t}}\,\circ\,\varepsilon(\rr(a)\times t, \rr(a)\star t)$. In particular, if $\rr(a)\, \mathcal{F}_{P_C}\, t$, then $a\, \bullet\, t=a~\circ~\varepsilon(\rr(a), t)$,\,\, if $t\leq \rr(a)$, then $a\bullet t=a{\downharpoonright}{_t}$.
\item[(2)]$t\, \bullet\, a= \varepsilon(t\times \dd(a), t\star \dd(a))\,\circ\, { _{t\star \dd(a)}}{\downharpoonleft} a$.
In particular, if $t \,\mathcal{F}_{P_C}\, \dd(a)$, then $t\,\bullet\, a= \varepsilon(t,\dd(a))\,\circ\, a$,\,\, if $t\,\leq\, \dd(a)$, then $t\,\bullet\, a={_t}{\downharpoonleft} a$
\item[(3)]$s \bullet t=\varepsilon(s\times t, s\star t), \dd(s\bullet t)=s\times t, \rr(s\bullet t)=s\star t$. In particular, if $s \,\mathcal{F}_{P_C}\, t$, then $s\bullet t=\varepsilon(s,t), \dd(s\bullet t)=s, \rr(s\bullet t)=t$.
\item[(4)]$s \bullet (s\times t)\bullet s=s\times t$,\,\, $t\bullet (s\star t) \bullet t=s\star t$.
\item[(5)] $s\leq \,t$ (i.e. $s\leq_{P_C} \,t$ ) if and only if  $s=t\bullet s\bullet t$.
\end{itemize}
\end{lemma}
\begin{proof}
(1) Since $t\in P_{C}$, we have $\rr(t)=t=\dd(t)$ and $\rr(s)=s=\dd(s)$ by Remark \ref{fanchoudexingzhi}. Denote $\dd(a)=p$ and $\rr(a)=q.$ Then $q\star t\leq t$ by Lemma \ref{jiben1} (2), and so ${_{q\star t}}{\downharpoonleft}t=q\star t$ by Lemma \ref{b} (12). By (E2), $\rr(\varepsilon(q\times t,q\star t))=\rr(q\times t,q\star t)=q\star t$. This together with (C3) gives that $\varepsilon(q\times t,q\star t)\circ q\star t=\varepsilon(q\times t,q\star t)$. Thus $$a\bullet t=a{\downharpoonright}{_{q\times t}}\circ \varepsilon(q\times t,q\star t)\circ {_{q\star t}}{\downharpoonleft}t =a{\downharpoonright}{_{q\times t}}\circ \varepsilon(q\times t,q\star t)\circ (q\star t)$$$$=a{\downharpoonright}{_{q\times t}}\circ \varepsilon(q\times t,q\star t) =a{\downharpoonright}{_{\rr(a)\times t}}\circ \varepsilon(\rr(a)\times t,\rr(a)\star t).$$
If $\rr(a)=q \,\mathcal{F}_{P_{C}}\, t$, then $q\times t=q$ and $ q\star t=t$, and so $$a\bullet t=a{\downharpoonright}{_q}\circ \varepsilon(q,t)=a{\downharpoonright}{_{\rr(a)}}\circ \varepsilon(\rr(a),t)=a\circ \varepsilon(\rr(a),t)$$ by Lemma \ref{b} (3). If $t\leq \rr(a)=q$, then  $q\times t=q\star t=t$ by (1) and (2) in Lemma \ref{jiben1}.
In view of (E1), Lemma \ref{b} (2) and (C3), we obtain that
$$a\bullet t=a{\downharpoonright}{_{q\times t}}\circ \varepsilon(q\times t,q\star t)=a{\downharpoonright}{_ t}\circ\varepsilon(t,t)=a{\downharpoonright}{_ t}\circ t=a{\downharpoonright}{_ t}\circ \rr(a{\downharpoonright}{_ t})=a{\downharpoonright}{_ t}.$$

(2) This is the dual of item (1).

(3) By Lemma \ref{jiben1} (1), we get $s\times t\leq s$, and so $s{\downharpoonright}{_{s\times t}}=s\times t$  by Lemma \ref{b} (11).  Now (E2) gives that  $\dd(\varepsilon(s\times t,s\star t))=\dd(s\times t,s\star t)=s\times t$. Dually, we have $\rr(\varepsilon(s\times t,s\star t))=s\star t$. This together with (C3) implies that $$(s\times t)\circ \varepsilon(s\times t,s\star t)=\varepsilon(s\times t,s\star t).$$ In view of item (1) in this lemma,  we obtain $$s\bullet t=s{\downharpoonright}{_{s\times t}}\circ \varepsilon(s\times t,s\star t)=(s\times t)\circ \varepsilon(s\times t,s\star t)=\varepsilon(s\times t,s\star t).$$ If $s \mathcal{F}_{P_{C}} t$, then $s\times t=s,s\star t=t$, and so $s\bullet t=\varepsilon(s,t)$.

(4) By item (3) of the present lemma, (1) and (2) in Lemma \ref{jiben1} and (E1),
$$s\bullet (s\times t)=  \varepsilon(s\times (s\times t), s\star (s\times t))= \varepsilon(s\times t, s\times t)=s\times t,$$
$$(s\times t)\bullet s  = \varepsilon(((s\times t)\times s, (s\times t)\star s) = \varepsilon(s\times t, s\times t)=s\times t.$$
This gives that $s\bullet (s\times t)\bullet s=s\times t$. Dually, $t\bullet (s\star t) \bullet t=s\star t$.

(5) If $s\leq t $, then  $s\times t=s\star t=s=t\times s=t\star s$ by (1) and (2) in Lemma \ref{jiben1}. By item (3) in the present lemma, we have $s\bullet t=\varepsilon(s\times t, s\star t)=\varepsilon(s, s)=s$ by (E1), and so  $$t\bullet s\bullet t=t\bullet s=\varepsilon(t\times s, t\star s)=\varepsilon(s,s)=s.$$
Conversely, assume that $s=t\bullet s\bullet t$.
Since $\dd(\varepsilon(s\times t, s\star t))=s\times t$ by Lemma \ref{oo} (2) and $\rr(t)=t$, it follows that $$s=t\bullet s\bullet t=t\bullet \varepsilon(s\times t,s\star t)=t{\downharpoonright}{_{t\times (s\times t)}}\circ U \circ V=(t\times (s\times t))\circ U \circ V$$ for some $U,V\in C$ by item (3) in the present lemma, (\ref{kkk}), Lemma \ref{jiben1} (1) and
Lemma \ref{b} (12).
This implies that $s=\dd(s)=\dd({t\times (s\times t)})=t\times (s\times t)\leq t$ by the fact that $s, t\times(s\times t)\in P_{C}$, Remark \ref{fanchoudexingzhi} and Lemma \ref{jiben1} (1).
\end{proof}
\begin{theorem}\label{xii}Let $(C,\circ,\dd,\rr,\leq, P_C, \times,\star,\varepsilon)$ be a chain projection ordered category. Define
$$
a\bullet b=a{\downharpoonright}_{\rr(a)\times \dd(b)} \circ~ \varepsilon(\rr(a)\times \dd(b),\rr(a)\star \dd(b) )  \circ~ _{\rr(a)\star \dd(b)}{\downharpoonleft} b,\,\,\label{zi} a^\clubsuit=\dd(a),\,\, a^\spadesuit=\rr(a)
$$
for all $a, b\in C$. Then  $\textbf{S}(C)=(C, \bullet,\,  {}^\clubsuit, {}^\spadesuit)$ forms a DRC-restriction semigroup with set of projections $P(\textbf{S}(C))=P_C$. Moreover, we have the following results:
\begin{itemize}
\item  If $C$ is a chain projection ordered groupoid, then $\textbf{S}(C)$ is a generalized regular $\circ$-semigroup.
\item  If $C$ is symmetric chain projection ordered category, then $\textbf{S}(C)$ is $P$-restriction.
\item  If $C$ is a symmetric chain projection ordered groupoid, then $\textbf{S}(C)$ is a regular $\circ$-semigroup.
\item  If $C$ is commutative chain projection ordered category, then $\textbf{S}(C)$ is restriction.
\item  If $C$ is commutative chain projection ordered groupoid, then $\textbf{S}(C)$ is an inverse semigroup.

\end{itemize}
\end{theorem}
\begin{proof} By Lemma \ref{ccccc} (3), $(C, \bullet)$ forms a semigroup. To show $\textbf{S}(C)$ is a DRC-restriction semigroup, we need to prove (i)--(vi) by symmetry. Let $x,y\in C$.

(i)  By (C3), $\dd(x)\circ x$ is defined and $\dd(x)\circ x=x$. In view of Lemma \ref{vvvv}, we have $x^\clubsuit\bullet x =\dd(x)\bullet x=\dd(x)\circ x=x$. Thus (i) holds.

(ii)  By the definition of the operation $\bullet$, we have $\dd(x\bullet y)=\dd(x{\downharpoonright}{_{\rr(x)\times \dd(y)}})$  and $\dd(x\bullet y^\clubsuit)=\dd(x{\downharpoonright}{_{ \rr(x)\times \dd(y^\clubsuit)}})$.
Since $\dd(y^\clubsuit)=\dd({\dd(y)})=\dd(y)$ by Remark \ref{fanchoudexingzhi}, it follows that $$(x\bullet y)^\clubsuit={\dd(x\bullet y)}={\dd(x\bullet y^\clubsuit)}=(x\bullet y^\clubsuit)^\clubsuit.$$

(iii) By Lemma \ref{edd} (3),
\begin{eqnarray}\label{kang1}
&&(x^\clubsuit \bullet y^\clubsuit)^\clubsuit
=({\dd(x)\bullet \dd(y)})^\clubsuit={\dd({\dd(x)\bullet \dd(y)})} =\dd(x)\times \dd(y).
\end{eqnarray}
This together with Lemma \ref{edd} (4) gives that
$$x^\clubsuit \bullet(x^\clubsuit \bullet y^\clubsuit)^\clubsuit \bullet x^\clubsuit=\dd(x)\bullet (\dd(x)\times \dd(y)) \bullet \dd(x)=\dd(x)\times \dd(y)=(x^\clubsuit \bullet y^\clubsuit)^\clubsuit.$$ By Corollary \ref{zhangyu}, (iii) is true.

(iv)--(v) In view of Remark \ref{fanchoudexingzhi}, we have $(x^\clubsuit)^\clubsuit=\dd(\dd(x))=\dd(x)=x^\clubsuit$ and  $$(x^\clubsuit)^\spadesuit=(\dd(x))^\spadesuit=\rr(\dd(x))=\dd(x)=x^\clubsuit.$$
By the above discussion and its dual,   $\textbf{S}(C)=(C, \bullet,\,  {}^\clubsuit,  {}^\spadesuit)$ forms a DRC semigroup with set of projections
\begin{eqnarray}\label{yd}
P({\textbf{S}(C)})=\{x^{\clubsuit}\mid x\in C\}=\{\dd(x)\mid x\in C\}=\{\rr(x)\mid x\in C\}=\{x^{\spadesuit}\mid x\in C\}=P_{C}.
\end{eqnarray}
Let $e\in P({\textbf{S}({C})})=P_{C}$ and $ x\in C$. Then $\rr(e)=e$ by Remark \ref{fanchoudexingzhi}. Denote $t=e\star \dd(x)$. By (1) and (2) in Lemma \ref{jiben1}, we have $t\leq \dd(x)$ and $t\times \dd(x)=t\star \dd(x)=t$. Moreover, Lemma \ref{b} (11) gives that ${_t}{\downharpoonleft}\dd(x)=t$. Since $x^\clubsuit=\dd(x)$, we have $(e\bullet x^{\clubsuit})^{\spadesuit}=\rr(e\bullet \dd(x))=e\star \dd(x)=t$ by Lemma \ref{edd} (3).  In view of Lemma \ref{edd} (2),  (E1), Lemma \ref{b} (1) and (C3), we have
$$(e\bullet x^{\clubsuit})^{\spadesuit}\bullet x=t\bullet x= \varepsilon(t\times \dd(x),t\star \dd(x)) \,\circ \, {_{t\star \dd(x)}}{\downharpoonleft}x=$$$$\varepsilon(t,t)\,\circ \,{_t}{\downharpoonleft}x= t\circ \,{_t}{\downharpoonleft}x=\dd({_t}{\downharpoonleft}x)\,\circ \,{_t}{\downharpoonleft}x={_t}{\downharpoonleft}x.$$
Denote $u=\rr({_t}{\downharpoonleft}x)$. Then $u\in P_{C}$, and so $\rr(u)=u$ by Remark \ref{fanchoudexingzhi}.  By Lemma \ref{b} (1), we have $u\leq \rr(x)$, and so $$\rr(x)\star u=\rr(x)\star \rr(u)=\rr(x)\times u=\rr(x)\times \rr(u)=u $$ by (1) and (2) in Lemma \ref{jiben1}. In riew of Lemma \ref{edd} (2),
$$(e\bullet x)^{\spadesuit}=\rr(e\bullet x)= \rr({_{e\star \dd(x)}}{\downharpoonleft} x)=\rr({_t}{\downharpoonleft}x)=u.$$
This together with Lemma \ref{edd} (1),  (E1),  Lemma \ref{b} (2) and (C3) yields that
 $$x \bullet(e\bullet x)^{\spadesuit}=x\bullet u =a{\downharpoonright}{_{\rr(x)\times \rr(u)}}\circ
\varepsilon(\rr(x)\times \rr(u),\rr(x)\star \rr(u)) $$$$=x{\downharpoonright}{_u}\circ \varepsilon(u,u)=x{\downharpoonright}{_u}\circ u=x{\downharpoonright}{_u}\circ \rr(x{\downharpoonright}{_u})=x{\downharpoonright}{_u}.$$ By the fact that $u=\rr({_t}{\downharpoonleft}x)$ and Lemma \ref{b} (5), we have ${_t}{\downharpoonleft}x=x{\downharpoonright}{_u}$. Thus $x\bullet (e\bullet x)^{\spadesuit}=(e\bullet x^{\clubsuit})^{\spadesuit}\bullet x$. This proves (vi). Thus $\textbf{S}(C)$ is a DRC-restriction semigroup by (\ref{dengjia}).

Finally, let $(C,\circ,\dd,\rr, ^{-1})$ be a groupoid and $x\in C$. By Lemmas \ref{vvvv} and \ref{qunpei}, $$x\bullet x^{-1}=x\circ x^{-1}=\dd(x)=x^\clubsuit=\dd(x)=\rr(x^{-1})= (x^{-1})^\spadesuit.$$ Dually, $x^{-1}\bullet x =x^\spadesuit=(x^{-1})^\clubsuit$. This implies that $\textbf{S}(C)=(C, \bullet,\,  {}^\clubsuit, {}^\spadesuit,  {-1})$ is a generalized regular $\circ$-semigroup by Remark \ref{tongyi}. The other cases can be proved by Proposition \ref{guanxi}, Lemma \ref{aisdm} and Remark \ref{tongyi}, we omit the details.
\end{proof}

Let $(S,\cdot, ^+, ^\ast)$ be a DRC-restriction semigroup. We use $\textbf{E}(S)$ to denote the subsemigroup of $S$ generated by $P({S})$. We say that $S$
is  {\em projection-generated} if $\textbf{\rm \bf E}(S)=S$.
\begin{prop}\label{fgpu}
Let $(C,\circ,\dd,\rr,\leq, P_C, \times, \star,\varepsilon)$ be a chain projection ordered category and $S={\rm \bf{S}}$$(C)$. Then $$P(S)=P_{C},\,\, \textbf{\rm \bf E}(S)={\rm im}(\varepsilon)=\{\varepsilon(\mathfrak{c})\mid \mathfrak{c}\in \mathscr{P}(P)\}.$$ Thus $S$ is projection-generated if and only if $\varepsilon$ is surjective.
\end{prop}
\begin{proof}
By (\ref{yd}), $P(S)=P_{C}$. Let $\mathfrak{p}=(p_{1},p_{2},\ldots,p_{k})\in \mathscr{P}(P_C)$. Then by (\ref{pp}), Lemma \ref{vvvv} (1) and Lemma \ref{edd} (3),
$$\varepsilon(\mathfrak{p})=\varepsilon(p_{1},p_{2})\circ \varepsilon(p_{2},p_{3})\circ \ldots \circ \varepsilon(p_{k-1},p_{k}) =\varepsilon(p_{1},p_{2})\bullet \varepsilon(p_{2},p_{3})\bullet \ldots \bullet \varepsilon(p_{k-1},p_{k})$$$$=(p_{1}\bullet p_{2})\bullet(p_{2}\bullet p_{3})\bullet \ldots \bullet(p_{k-1}\bullet p_{k}) =p_{1}\bullet p_{2}\bullet \ldots \bullet p_{k}\in \textbf{E}(S).$$
Conversely, let $a\in \textbf{E}(S)$. Then there exists an positive integer $k$ and $p_{1},p_{2},\ldots ,p_{k}\in P$ such that $a=p_{1}\bullet p_{2}\bullet \cdots \bullet p_{k}$. In the sequel, we shall prove that $a\in {\rm im}{\varepsilon}$ by mathematical induction. When $k=1$, we have $\varepsilon(p_{1})=p_{1}$ by (E1), and so $p_{1}\in {\rm im}\varepsilon$. Let $p_{1}\bullet p_{2}\bullet \ldots \bullet p_{k-1}=\varepsilon(\mathfrak{p})$, where $\mathfrak{p}\in \mathscr{P}(P)$. Denote $s=\rr(\varepsilon(\mathfrak{c}))\times p_{k}$ and $t=\rr(\varepsilon(\mathfrak{c}))\star p_{k}$. By Lemma  \ref{edd} (1), (E6) and (E3),
$$a=p_{1}\bullet p_{2}\bullet \ldots \bullet p_{k-1}\bullet p_{k}=\varepsilon(\mathfrak{c})\bullet p_{k}=\varepsilon(\mathfrak{c}){\downharpoonright}{_s}\circ \varepsilon(s,t)=\varepsilon(\mathfrak{c}{\downharpoonright}{_s})\circ \varepsilon(s,t)=\varepsilon(\mathfrak{c}{\downharpoonright}{_s} \circ (s,t))\in {\rm im}\varepsilon.$$
By mathematical induction, we have $a\in {\rm im}\varepsilon$ for all $a\in \textbf{E}(S)$. Thus $\textbf{E}(S)\subseteq {\rm im}\varepsilon$. This implies that $\textbf{E}(S)= {\rm im}\varepsilon$,  and so $S$ is projection-generated if and only if $\varepsilon$ is surjective.
\end{proof}

\begin{lemma}\label{yang2} Let $(C,\circ,\dd,\rr,\leq,P_C, \times, \star,\varepsilon)$ be a chain projection ordered category and $S={\bf S}(C)$. Then $\leq_{S}=\leq$.
\end{lemma}
\begin{proof} Let $x,y\in C$. Then
\begin{eqnarray*}
x \leq_{S} y
&\Longleftrightarrow & x^\clubsuit=y^\clubsuit\bullet x^\clubsuit\bullet y^\clubsuit,x=x^\clubsuit\bullet y\, \, \, \,\,\,\, (\mbox{by (\ref{xiaoyus})})\\
&\Longleftrightarrow& \dd(x)=\dd(y)\bullet \dd(x)\bullet \dd(y) ,x=\dd(x)\bullet y\ \ \ \ (\mbox{by (\ref{zi})})\\
&\Longleftrightarrow&\dd(x)\leq \dd(y),x={_{\dd(x)}}{\downharpoonleft} y. \, \, \, \,\,\,\,(\mbox{by (2) and (5) in Lemma \ref{edd} })\\
&\Longleftrightarrow& x\leq y.  \ \ \ \ \ \ \ \,(\mbox{by (O1), (O3) and Lemma \ref{b} (3)})
\end{eqnarray*}
Thus  $\leq_{S}=\leq$.
\end{proof}

\begin{lemma}\label{ydd}
Let $(C,\circ,\dd,\rr,\leq,P_C, \times, \star,\varepsilon)$ be a chain projection ordered category. Then the restricted product of $S=$${\rm \bf{S}}$$({C})$ is exactly $``\circ"$  and $(P(S),  \times_{S}, \star_{S})=(P_{C},\times, \star)$.
\end{lemma}
\begin{proof}Consider the restricted product of $\textbf{S}({C})$: For all $x,y\in C$,
$$x\odot y=\left\{\begin{array}{cc}
x\bullet y
&\mbox{ if }\ x^\spadesuit=y^\clubsuit,\\
\mbox{undefined} & \mbox{otherwise},
\end{array} \right.
$$
Let $x,y\in C$. Then $x^\spadesuit=\rr(x), y^\clubsuit=\dd(y)$, and so
\begin{center}
$x\odot y$ is defined $ \Longleftrightarrow  \rr(x)=\dd(y) \Longleftrightarrow x\circ y$ is defined,
\end{center}
In this case, we have $x\odot y=x\bullet y=x\circ  y$ by Lemma  \ref{vvvv} (1).
This shows that $\odot$ is equal to $\circ$. On the other hand, we obtain $P(S)=P_{C}$ by (\ref{yd}). Let $e,f\in P_{C}=P(S)$. By Lemma \ref{edd} (3),  $$e\times_{S} f=(e\bullet f)^{\clubsuit}=\dd(e\bullet f)=e\times f,\,\,\, e\star_{S} f=(e\bullet f)^{\spadesuit}=\rr(e\bullet f)=e\star f.$$
Thus $(P(S),  \times_{S}, \star_{S})=(P_{C},\times, \star)$.
\end{proof}

\begin{lemma}\label{wang} Let $\varphi$ be a chain projection ordered functor from $(C_{1},\circ,\dd,\rr,\leq, P_{C_1}\times,\star,\varepsilon_{1})$ to $(C_2, \circ,  \dd,\rr,\leq, P_{C_2}, \times,\star,\varepsilon_2)$. Then  ${\rm\bf{S}}\varphi:C_1 \rightarrow C_2,\ \ x\mapsto x\varphi$  provides a (2,1,1)-homomorphism from ${\rm \bf{S}}(C_1)$ to ${\rm \bf{S}}(C_2)$.
\end{lemma}
\begin{proof}By (F1), $x^{\clubsuit}\varphi=(\dd(x))\varphi=\dd(x\varphi)=(x\varphi)^{\clubsuit}.$
This together with its dual shows that $\textbf{S}\varphi$ preserves $``^\clubsuit"$ and $``^{\spadesuit}$".
Using (F2), Lemma \ref{g}, (F5), (F4) and (F1) in order,
$$(x\bullet y)\varphi=(x{\downharpoonright}{_{\rr(x)\times \dd(y)}} ~\circ~ \varepsilon_{1}( \rr(x)\times \dd(y),\rr(x)\star \dd(y))~\circ~ {_{\rr(x)\star \dd(y)}}{\downharpoonleft} y)\varphi
$$$$=(x{\downharpoonright}{_{\rr(x)\times \dd(y)}})\varphi ~\circ~ (\varepsilon_{1}(\rr(x)\times \dd(y),\rr(x)\star \dd(y)))\varphi ~\circ~ ({_{\rr(x)\star \dd(y)}}{\downharpoonleft} y)\varphi
$$$$=x\varphi{\downharpoonright}{_{(\rr(x)\times \dd(y))\varphi}} ~\circ~ \varepsilon_{2}((\rr(x)\times \dd(y))\varphi,(\rr(x)\star \dd(y))\varphi)~\circ~ {_{(\rr(x)\star \dd(y))\varphi}}{\downharpoonleft} y\varphi
$$$$=x\varphi{\downharpoonright}{_{\rr(x)\varphi\times \dd(y)\varphi}} ~\circ~ \varepsilon_{2}(\rr(x)\varphi\times \dd(y)\varphi,\rr(x)\varphi \star \dd(y)\varphi)~\circ~ {_{\rr(x)\varphi\star \dd(y)\varphi}}{\downharpoonleft} y\varphi
$$$$=x\varphi{\downharpoonright}{ _{\rr(x\varphi)\times \dd(y\varphi)}} ~\circ~ \varepsilon_{2}(\rr(x\varphi)\times \dd(y\varphi),\rr(x\varphi) \star \dd(y\varphi)) ~\circ~ {_{\rr(x\varphi)\star \dd(y\varphi)}}{\downharpoonleft} y\varphi=x\varphi \bullet y\varphi.$$
Thus $\textbf{S}\varphi$ is a (2,1,1)-homomorphism.
\end{proof}

\begin{lemma}\label{zhang3} Let $(S,\cdot,^+,^\ast)$ be a DRC-restriction semigroup. Then we have $\textbf{\rm \bf S}(\textbf{\rm \bf C}(S))=S$.
\end{lemma}
\begin{proof}By Theorem \ref{fii}, $\textbf{C}(S)$ is a chain projection ordered category  and
\begin{equation}\label{5.1}
\dd_{S}(x)=x^+,\, \rr_{S}(x)=x^\ast,\,   e\times_S f=(ef)^+,\,  e\star_S f=(ef)^\ast
\end{equation}
for all $x\in S$ and $e,f\in P_{S}=P(S)$.  Let $x,y\in S$. Then
 $x\circ_{S} y=xy$ if and only if
\begin{equation}\label{5.2}x^\ast=\rr_{S}(x)=\dd_{S}(y)=y^{+}
\end{equation} where $xy$ is the product of $x$ and $y$ in  $(S,\cdot)$. Moreover,  for $x\in S, e,f\in P_S$ with $e\leq_{S} \dd_{S}(x)$ and $f\leq_{S} \rr_{S}(x)$,   have  ${_e}{\downharpoonleft} x=ex$  and $ x{\downharpoonright}{_f}=xf$.
Define three operations on $S$ as follows:
$$x\bullet y=x{\downharpoonright}{_{\rr_{S}(x)\times_S \dd_{S}(y)}}\ \circ  \varepsilon_{S}(\rr_{S}(x)\times_S \dd_{S}(y),\rr_{S}(x)\star_S \dd_{S}(y))  \,   {_{\rr_{S}(x)\star_S \dd_{S}(y)}}{\downharpoonleft}y,$$ and $x^{\clubsuit}=\dd_{S}(x),x^{\spadesuit}=\rr_{S}(x),$ where $\varepsilon_{S}$ is defined as in (\ref{hui}).
By Theorem \ref{xii}, $\textbf{S}(\textbf{C}(S))=(S,\bullet,^{\clubsuit},^{\spadesuit})$ is a DRC-restriction semigroup, and  $$x^{\clubsuit}=\dd_{S}(x)=x^{+},x^{\spadesuit}=\rr_{S}(x)=x^{\ast}$$ by (\ref{5.1}).  Furthermore,
\begin{eqnarray}\label{5.3}
x\bullet y&=&x{\downharpoonright}{_{\rr_{S}(x)\times_S \dd_{S}(y)}}\ \circ \varepsilon_{S}(\rr_{S}(x)\times_S \dd_{S}(y),\rr_{S}(x)\star_S \dd_{S}(y))   \ \circ  {_{\rr_{S}(x)\star_{S} \dd_{S}(y)}}{\downharpoonleft}y\nonumber\\
&=&x{\downharpoonright}{_{(x^\ast y^+)^+}}\ \circ_S    \varepsilon_{S}((x^\ast y^+)^+,(x^\ast y^+)^\ast)\ \circ_S {_{(x^\ast y^+)^\ast}}{\downharpoonleft} y\nonumber\\
&=&x(x^\ast y^+)^+\circ_{S} (x^\ast y^+)^+ (x^\ast y^+)^\ast\circ_{S} (x^\ast y^+)^\ast y\nonumber\\
&=&x(x^\ast y^+)^+ (x^\ast y^+)^\ast y
=xx^\ast y^+  y
=xy.\ \ \ (\mbox{by Lemma \ref{kangwei} (4), (i), (i)$'$})
\end{eqnarray}
Thus $\textbf{\rm \bf S}(\textbf{\rm \bf C}(S))=S$.
\end{proof}
\begin{lemma}\label{zhang31} Let $(C,\circ,\dd,\rr,\leq, P_C, \times,\star,\varepsilon)$ be a chain projection ordered category. Then we have  ${\rm\bf{C}}({\rm \bf{S}}({ C})={ C}$.
\end{lemma}
\begin{proof}
Define three operations on $C$ as follows:
$$x \bullet y=x{\downharpoonright}{_{\rr(x)\times \dd(y)}} \circ\ \varepsilon(\rr(x)\times \dd(y),\,\,\rr(x)\star \dd(y)) \circ\ {_{\rr(x)\star \dd(y)}}{\downharpoonleft}y.$$
\begin{equation}\label{5.4}
x^{\clubsuit}=\dd(x),\,\,  x^{\spadesuit}=\rr(x).
\end{equation}
Then by Theorem \ref{xii}, $S=\textbf{S}({C})=(C,\bullet, ^\clubsuit, ^\spadesuit)$ is a DRC-restriction semigroup and $P_{C}=P(S)$. By Theorem \ref{fii}, we have $\textbf{C}(S)=(C,\circ_{S},\dd_{S},\rr_{S}, \leq_{S},  P_{\textbf{C}(S)}, \times_{S},\star_{S},\varepsilon_{S})$ and $P_{\textbf{C}(S)}=P(S)$. Let $x,y\in C$. Then $\dd_{S}(x)=x^\clubsuit=\dd(x)$ and $\rr_{S}(x)=x^\spadesuit=\rr(x).$ Thus $\dd_{S}=\dd,\rr_{S}=\rr$. Moreover,  $x\circ_{S}y$ is defined if and only if $\rr_{S}(x)=\dd_{S}(y)$, if and only if $\rr(x)=\dd(y)$. In this case, we obtain $x\circ_{S} y=x\bullet y=x\circ y$  by Lemma \ref{vvvv}. This shows that $\circ_{S}=\circ$. By Lemma \ref{ydd}, we get $(P_{C},  \times, \star)=(P_{\textbf{C}(S)},\times_{S}, \star_{S})$. Lemma \ref{yang2} gives that $\leq_{S}=\leq$, and the fact $P_{C}=P_{\textbf{C}(S)}$ gives that $ \mathscr{P}(P_{C})=\mathscr{P}(P_{\textbf{C}(S)})$.

Let $\mathfrak{p}=(p_{1},p_{2},\ldots,p_{k})\in \mathscr{P}(P_{C})$. Then $p_{1}\,\mathcal{F}_{P_C}\,p_{2}\,\mathcal{F}_{P_C}\,\ldots \mathcal{F}_{P_C}\,p_{k}$. If  $k=1$, we have $\varepsilon(\mathfrak{p})=\varepsilon(p_{1})=p_{1}=\varepsilon_{S}(p_{1})=\varepsilon_{S}(\mathfrak{p})$ by (E1).
Denote $\mathfrak{q}=(p_{1},p_{2},\ldots,p_{k-1})$. Then $\mathfrak{p}=\mathfrak{q}\circ (p_{k-1},p_{k})$. Assume that $\varepsilon(\mathfrak{q})=\varepsilon_{S}(\mathfrak{q})=p_{1}\bullet p_{2}\bullet \ldots \bullet p_{k-1}$. Then $$\varepsilon_{S}(\mathfrak{p})=p_{1}\bullet p_{2}\bullet \ldots \bullet p_{k-1}\bullet p_{k}=\varepsilon(\mathfrak{q})\bullet p_{k}.$$ By (E2), we have $\rr(\varepsilon(\mathfrak{q}))=\rr(\mathfrak{q})=p_{k-1}$. Since $p_{k-1} \mathcal{F}_{P_C} p_{k}$, by Lemma \ref{edd} (1) and (E3) we obtain $$\varepsilon_{S}(\mathfrak{p})=\varepsilon(\mathfrak{q})\bullet p_{k}=\varepsilon(\mathfrak{q}) \circ \varepsilon(p_{k-1},p_{k})=\varepsilon(\mathfrak{q} \circ (p_{k-1},p_{k}))=\varepsilon(\mathfrak{p}).$$ Thus $\varepsilon_{S}=\varepsilon$.  By the above discussions, we have ${ C}=\textbf{C}(S)=\textbf{C}(\textbf{S}({  C}))$.
\end{proof}
Now we can give our main theorem in this paper.
\begin{theorem}\label{hgdc}
The category $\mathbb{DRS}$ of DRC-restriction semigroups together with (2,1,1)-homomorphisms is isomorphic to the category $\mathbb{CPOC}$ of chain projection ordered categories together with chain projection ordered functors.
\end{theorem}
\begin{proof}
By Theorems \ref{xii}, \ref{fii} and Lemmas \ref{zhang4},  \ref{wang},  \ref{zhang3} and   \ref{zhang31},
$$\textbf{S}: \mathbb{CPOC} \rightarrow \mathbb{DRS}: (C,\circ,\dd,\rr,\leq, P_C, \times,\star,\varepsilon)\rightarrow \textbf{S}(C),$$$$ \textbf{C}: \mathbb{DRS} \rightarrow \mathbb{CPOC}: (S, \cdot, ^+, ^\ast)\rightarrow \textbf{C}(S)$$
are mutually inverse functors.  Thus the result holds.
\end{proof}
Finally, we also have the following special cases.

\begin{theorem}\label{teshuqingkuang}
The category of generalized regular $\circ$-semigroups (respectively, $P$-restriction semigroups, regular $\circ$-semigroups, restriction semigroups, inverse semigroups) together with (2,1,1)-homomorphisms is isomorphic to the category of chain projection ordered groupoids (respectively, symmetric chain projection ordered categories, symmetric chain projection ordered groupoids, commutative chain projection ordered categories, commutative chain projection ordered groupoids) together with chain projection ordered functors.
\end{theorem}

\section{Chain semigroups}
In this section, we consider the chain semigroups determined by  strong two-sided projection algebras and give a presentation of these semigroups by using Theorem \ref{hgdc}. We first recall some basic results on categories given in \cite{East1}.
Let ${\mathcal C}=(C, \circ, \dd, \rr)$ be a category and $\approx$ be an equivalence on $C$. Then $\approx$ is called a {\em  congruence} on $\mathcal C$  if for all $a,b,u,v\in C$, the following conditions hold:
\begin{itemize}
\item[(V1)] If $a\approx b$, then $\dd(a)=\dd(b)$ and $ \rr(a)=\rr(b)$.
\item[(V2)] If $\rr(u)=\dd(a)=\dd(b)$ and $a\approx b$, then $u\circ a=u\circ b$.
\item[(V3)] If $\dd(v)=\rr(a)=\rr(b)$ and $ a\approx b$, then $a\circ v=b\circ v$.
\end{itemize}
Let ${\mathcal C}=(C, \circ, \dd, \rr)$ be a category and $\approx$ be a  congruence on $\mathcal C$. Let $a\in C$.  Denote the $\approx$-class of $C$ containing $a$ by $[a]$. Moreover, denote $C/\hspace{-1.5mm} \approx=\{[a]\mid a\in C\}$. Define a partial binary operation $``\circ"$ on $C/\hspace{-1.5mm} \approx$ and two maps $\dd,\rr$ from $C/\hspace{-1.5mm} \approx$ to $C/\hspace{-1.5mm} \approx$ as follows: For all $[a],[b]\in C/\hspace{-1.5mm} \approx$,
$$[a]\circ [b]=\left\{\begin{array}{cc}
[a\circ b]\,\,\,\,\,\,\mbox{ if }\ \rr([a])=\dd([b]),\\
\mbox{undefined}\,\,\,\,\,\,\mbox{otherwise},
\end{array} \right.
 $$
$$\dd([a])=[\dd(a)],\rr([a])=[\rr(a)].$$
Then it is easy to see that ${\mathcal C}/\hspace{-1.5mm} \approx=(C/\hspace{-1.5mm} \approx, \circ, \dd, \rr)$ is a category, and is called {\em the quotient category} of ${\mathcal C}$ with respect to~ $\hspace{-1.5mm} \approx$. In this case, the set of objects of ${\mathcal C}/\hspace{-1.5mm} \approx$ is $P_{C/ \approx}=\{[p]\mid p\in P_{C}\}$. Let $p,q\in P_{C}$. By (V1), it follows that $p=q$ if only if $[p]=[q].$  

A congruence $\approx$ on an ordered category ${\mathcal C}=(C, \circ, \dd, \rr, \leq)$ is called an {\em ordered congruence} if for all $a,b\in C$ and $p\in P_{ C}$, the following condition holds:
\begin{itemize}
\item[(V4)] If $a\approx b$ and $p\leq \dd(a)$, then $_{p}{\downharpoonleft}  a\approx   {_{p}{\downharpoonleft}b}$.
\end{itemize}

\begin{lemma}[\cite{East1}]\label{ting1}Let $\approx$ be an ordered  congruence on an ordered category ${\mathcal C}=(C, \circ, \dd, \rr, \leq)$.
 Define the relation ``$\leq_{\approx}$" on $C/\hspace{-1.5mm} \approx$ as follows:
For all $[a], [b]\in C/\hspace{-1.5mm} \approx$,
\begin{equation}\label{tin2}
[a]\leq_\approx [b]\Longleftrightarrow \mbox{ there exist }x\in [a], y\in [b]\mbox{ such that }x\leq y.
\end{equation}
Let $[a], [b]\in C/\hspace{-1.5mm} \approx$.  Then  $[a]\leq_\approx [b]$ if and only if for all $y\in [b]$, there exists $x\in [a]$ such that $x\leq y$.  Furthermore, ${\mathcal C}/\hspace{-1.5mm} \approx=(C/\hspace{-1.5mm} \approx, \cdot, \dd, \rr, \leq_\approx)$ is an ordered category.
\end{lemma}
\begin{lemma}[\cite{East1}]\label{ting2}
Let ${\mathcal C}=(C, \circ, \dd, \rr)$ be a category and $\Omega$ be a subset of $C\times C$. Assume that the following condition holds:
\begin{equation}\label{v-con}(u,v)\in \Omega\Longrightarrow \dd(u)=\dd(v) \mbox{ and } \rr(u)=\rr(v).
\end{equation}
 Denote the least congruence of $\mathcal C$ containing $\Omega$  by $\Omega^{\sharp}$. Then for all $a,b\in C$,  $(a,b)\in \Omega^{\sharp}$ if and only if there exists a sequence
$$a=a_{1}\rightarrow \cdots\rightarrow a_{k}=b,$$ where
$$a_{i}=b_{i}\circ u_{i}\circ c_{i},\, a_{i+1}=b_{i}\circ v_{i}\circ c_{i},\, b_{i},c_{i}\in C,\, (u_{i},v_{i})\in \Omega \cup \Omega^{-1},$$$$ \Omega^{-1}=\{(v,u)\mid (u,v)\in \Omega\},\, i=1,2, \ldots, k.$$
\end{lemma}

\begin{lemma}[\cite{East1}]\label{mis}Let ${\mathcal C}=(C, \circ, \dd, \rr, \leq)$ be an ordered category and $\Omega$ be a subset of $C\times C$ satisfying (\ref{v-con}). Assume that the  following condition holds:
\begin{equation}\label{bca}
(a,b)\in \Omega \Longrightarrow  _{p}{\downharpoonleft}a \approx  {_{p}{\downharpoonleft}b} \mbox{ for all } p\in P_{ C} \mbox{ with } p\leq \dd(a).
\end{equation}
Then $\Omega^{\sharp}$ is an ordered congruence on $\mathcal C$.
\end{lemma}
Let $(P, \times, \star)$ be a strong projection algebra and $p,e,f\in P$. Then $(e,p,f)$ is called an {\em admissible triple } of $P$ if
\begin{equation}\label{ccri}
f=e\theta_{p}\theta_{f}=(e\star p)\star f,\,\,\,\,\,e=f\delta_{p}\delta_{e}=e\times (p\times f),
\end{equation}
that is, $f=(e\star p)\star f,\, e=e\times (p\times f).$

\begin{lemma}\label{xxed}Let $(P, \times, \star)$ be a strong projection algebra and $(e,p, f)$ be an  admissible triple of $P$. Then $$(e,e\theta_{p},f)=(e,e\star p, f), (e,f\delta_{p},f)=(e, p\times f, f)\in {\mathscr{P}}(P),$$ where ${\mathscr{P}}(P)$ is the path category of $P$.
\end{lemma}
\begin{proof}
By Lemma \ref{jiben1} (1), we have $p\times f\leq_{P} p$. This together with (\ref{ccri}) and (5) and (1) in Lemma \ref{jiben1} gives that $e=e\times(p\times f)\leq_{P} e\times p\leq_{P} e$, and so  $e\times p=e$. By Lemma \ref{jiben1} (3), $e\times e\theta_p=e\times(e\star p)=e\times p=e$. In view of (R4) and (R1), we have $$e\star (e\theta_p)=e\star(e\star p)=(e\star p)\star(e\star p)=e\star p= e\theta_{p}.$$ This implies that $e\, \mathcal{F}_{P}\, e\theta_{p}$. On the other hand, denote $t=e\theta_{p}$. By (\ref{ccri}), (R4) and (R1), $$e\theta_{p}\star f=t\star(t\star f)=(t\star f)\star(t\star f)=t\star f=e\theta_{p}\theta_{f}=f,$$ and Lemma \ref{jj} and (\ref{ccri}) imply that $$e\theta_{p}\times f=f\delta_{e\theta_{p}}=f\delta_{p}\delta_{e}\theta_{p}=e\theta_{p}.$$ This shows that $e\theta_{p}\, \mathcal{F}_{P}\, f$. Thus $(e,e\theta_{p},f)\in {\mathscr{P}}(P).$ Dually, we have $(e,f\delta_{p},f)\in {\mathscr{P}}(P)$.
\end{proof}
\begin{lemma}\label{bcv}Let $(P, \times, \star)$ be a strong projection algebra, $(e,p, f)$ be an  admissible triple of $P$ and $e'\in P$, $e'\leq_{P} e$.  Denote $f{'}=e{'}\theta_{p}\theta_{f}$.
Then $(e{'}, p, f{'})$ is also an  admissible triple of $P$ and
$$_{e{'}}{\downharpoonleft\hspace{-1mm}(e, e\theta_{p}, f)}=(e{'}, e' \theta_p, f{'}), \,\,_{e{'}}{\downharpoonleft\hspace{-1mm}(e, f\delta_p,f)}=(e{'}, f'\delta_p,f{'}).$$
\end{lemma}
\begin{proof}Since $(e,p, f)$ is an  admissible triple of $P$, we have $f=e\theta_{p}\theta_{f}$ and $e=f\delta_{p}\delta_{e}$.
Denote $t=e{'}\theta_{p}$. By (R4) and (R1), $$e{'}\theta_{p}\theta_{f{'}}=e{'}\theta_{p}\theta_{e{'}\theta_{p}\theta_{f}}=t\theta_{t\theta_{f}}=t\star(t\star f)=(t\star f)\star(t\star f)=t\star f=e{'}\theta_{p}\theta_{f}=f{'}.$$
On the other hand, using Lemma \ref{jj}, Lemma \ref{jiben1} (3), Lemma \ref{o} , (\ref{ccri}), (\ref{l}) and the fact that $e{'}\leq_{P} e$ in order, we have
$$f{'}\delta_{p}\delta_{e{'}}=e{'}\theta_{p}\theta_{f}\delta_{p}\delta_{e{'}}=e{'}\theta_{f\delta_{p}}\delta_{e{'}}=e'\times (e'\star f\delta_{p})$$$$=e'\times  f\delta_{p}=
f\delta_{p}\delta_{e{'}}
=f\delta_{p}\delta_{e}\delta_{e{'}}=e\delta_{e{'}}=e'\times e=e{'}.$$
Thus $(e{'}, p, f{'})$ is an  admissible triple of $P$. By (\ref{v}),
$$_{e{'}}{\downharpoonleft\hspace{-1mm}(e, e\theta_p, f)}=(e{'},e{'}\theta_{e\theta_{p}},e{'}\theta_{e\theta_{p}}\theta_{f}),\,\, _{e{'}}{\downharpoonleft\hspace{-1mm}(e, f\delta_p, f)}=(e{'},e{'}\theta_{f\delta_{p}},e{'}\theta_{f\delta_{p}}\theta_{f}).$$
According to Lemma \ref{jiben1} (5) and the fact that $e'\leq_{P} e$, we have $e'\star p\leq e\star p$, and so
$$e{'}\theta_{e\theta_{p}}=e{'}\star(e\star p)=(e{'}\star p)\star(e\star p)=e{'}\star p=e{'}\theta_{p}$$ and $e{'}\theta_{e\theta_{p}}\theta_{f}=e{'}\theta_{p}\theta_{f}=f{'}$ by (R4) and Lemma \ref{jiben1} (2).
Moreover, the definition of $f^{'}$, (P3) and Lemma \ref{jj} give
$$e{'}\theta_{f\delta_{p}}\theta_{f}=(e{'}\star(p\times f))\star f=(e{'}\star p)\star f=e{'}\theta_{p}\theta_{f}=f{'}$$  and $e{'}\theta_{f\delta_{p}}=e{'}\theta_{p}\theta_{f}\delta_{p}=f{'}\delta_{p}$. This implies that
$$_{e{'}}{\downharpoonleft\hspace{-1mm}(e, e\theta_p, f)}=(e', e'\theta_p, f'),\,\, _{e{'}}{\downharpoonleft\hspace{-1mm}(e, f\delta_p, f)}=(e', f'\delta_p, f').$$
Thus the result follows.
\end{proof}

Let $(P, \times, \star)$ be a strong projection algebra. Let $\Omega_1(P)=\{((p,p), p)\mid p\in P\}$,
$$\Omega_2(P)=\{((e,e\star p,f), (e,  p\times f, f))\mid (e,p,f)\mbox{ is }\mbox{an admissible triple of } P\}.$$
Write $\Omega(P)=\Omega_1(P)\cup \Omega_2(P)$. Obviously, $\Omega(P)$ satisfies (\ref{v-con}). We use   $\approx_P$ to denote the congruence $\Omega^{\sharp}(P)$ on $\mathscr{P}(P)$ generated by $\Omega(P)$.

\begin{lemma}\label{piabn}Let $(P, \times, \star)$ be a strong projection algebra. Then $\approx_P$ is an ordered congruence on $\mathscr{P}(P)$.
\end{lemma}
\begin{proof}
Let $(\mathfrak{s}, \mathfrak{t})\in \Omega(P)$. If $$p, r\in P, \mathfrak{s}=(p,p), \mathfrak{t}=p, r\leq_{P} p=\dd(\mathfrak{s})=\dd(\mathfrak{t}),$$ then $r\theta_{p}=r\star p=r$ by Lemma \ref{jiben1} (2), and so $_{r}{\downharpoonleft}\mathfrak{s}=(r,r\theta_{p})=(r,r),\,
_{r}{\downharpoonleft}\mathfrak{t}=r$ by (\ref{v}). This implies that $(_{r}{\downharpoonleft}\mathfrak{s},\,  _{r}{\downharpoonleft}\mathfrak{t})\in \Omega(P)$, and hence
$_{r}{\downharpoonleft}\mathfrak{s} \approx_P{} _{r}{\downharpoonleft}\mathfrak{t}$. Assume that $(e, p, f)$ is an admissible triple of $P$ and $$\mathfrak{s}=(e, e\theta_p, f), \mathfrak{t}=(e, f\delta_p, f),\, e'\in P, e{'}\leq_{P} e=\dd(\mathfrak{s})=\dd(\mathfrak{t}),\, f^{'}=e{'}\theta_{p}\theta_{f}.$$ By Lemma \ref{bcv},  $(e{'}, p, f{'})$ is also an admissible triple of $P$ and
$$_{e{'}}{\downharpoonleft\hspace{-1mm}\mathfrak{s}}=(e{'}, e' \theta_p, f{'}), \,\,_{e{'}}{\downharpoonleft\hspace{-1mm}\mathfrak{t}}=(e{'}, f'\delta_p,f{'}).$$
This implies that $(_{e{'}}{\downharpoonleft\hspace{-1mm}\mathfrak{s}},{}  _{e{'}}{\downharpoonleft\hspace{-1mm}\mathfrak{t}})\in \Omega(P)$, and so $_{e{'}}{\downharpoonleft\hspace{-1mm}\mathfrak{s}}\approx_P{}  _{e{'}}{\downharpoonleft\hspace{-1mm}\mathfrak{t}}$. The above statement shows that $\Omega(P)$ satisfies  (\ref{bca}). Since $\Omega(P)$ satisfies (\ref{v-con}), $\approx_P$ is an ordered congruence on $\mathscr{P}(P)$ by Lemma \ref{mis}.
\end{proof}

Let $(P, \times, \star)$ be a strong projection algebra. Denote  $\mathscr{C}(P)=\mathscr{P}(P)/\hspace{-1.5mm}\approx_P$. By Lemmas \ref{ting1} and \ref{piabn}, $(\mathscr{C}(P), \circ, \dd, \rr, \leq_{\approx_P}, P_{\mathscr{C}(P)}, \times, \star)$ is an ordered category, and is called the {\em chain category} of $(P, \times, \star)$. In fact, $\mathscr{C}(P)$ is a weak projection ordered category. Several key points of this ordered category are listed as follows:
\begin{itemize}
\item The elements in $\mathscr{C}(P)$: For any $\mathfrak{p}=(p_{1},p_{2},\ldots,p_{k})\in \mathscr{P}(P)$, denote the $\approx_P$-class containing $\mathfrak{p}$ by $[\mathfrak{p}]=[p_{1}, p_{2}, \ldots, p_{k}].$ Then $\mathscr{C}(P)=\{[\mathfrak{p}]\mid \mathfrak{p}\in \mathscr{P}(P)\}$.
\item For any $[\mathfrak{p}]=[p_{1}, p_{2}, \ldots, p_{k}]\in \mathscr{C}(P)$,  we have $\dd([\mathfrak{p}])=[\dd(\mathfrak{p})]=[p_{1}]$ and $\rr([\mathfrak{p}])=[\rr(\mathfrak{p})]=[p_k]$, and so  $$P_{\mathscr{C}(P)}=\{[p]\mid p\in P\}.$$
\item For all $p,q\in P$, $[p]=[q]$ if and only if $p=q$.  This implies that $(P_{\mathscr{C}(P)}, \times, \star)$ forms a strong projection algebra, where $[p]\times [q]=[p\times q]$ and  $[p]\star [q]=[p\star q]$ for all $p,q\in P$. So we can identify $[p]$ with $p$ for any $p\in P$, and identity  $(P_{\mathscr{C}(P)}, \times, \star)$ with $(P, \times, \star)$.

\item If $\mathfrak{p}=(p_{1},p_{2},\ldots,p_{k}), \mathfrak{q}=(q_{1},q_{2},\ldots,q_{l})  \in \mathscr{P}(P)$ and $\rr[\mathfrak{p}]=[p_{k}]=[q_{1}]=\dd[\mathfrak{q}]$ (i.e. $p_k=q_1$),  then
$$[\mathfrak{p}]\circ [\mathfrak{q}]=[p\circ q]=[p_{1},\ldots,p_{k-1},p_{k}=q_{1},q_{2},\ldots,q_{l}].$$
\item If $\mathfrak{p}, \mathfrak{q}\in \mathscr{P}(P)$, then
$$[\mathfrak{p}]\leq_{\approx_P} [\mathfrak{q}]\Longleftrightarrow \mbox{there exist }\mathfrak{m}, \mathfrak{n}\in \mathscr{P}(P)\mbox{ such that }
\mathfrak{m}\in[\mathfrak{p}],  \mathfrak{n}\in[\mathfrak{q}]\mbox{ and }\mathfrak{m}\leq_{\mathscr{P}} \mathfrak{n}$$$$\Longleftrightarrow \mbox{ there exists }\mathfrak{m} \in \mathscr{P}(P)\mbox{ such that }\mathfrak{m}\in[\mathfrak{p}]\mbox{ and }\mathfrak{m}\leq_{\mathscr{P}} \mathfrak{q}.$$
\item   For any $[\mathfrak{p}]\in \mathscr{C}(P)$ and $r,s\in P$, if $[r]\leq_{\approx_{P}}\dd([\mathfrak{p}])$, $[s]\leq_{\approx_{P}}\rr([\mathfrak{p}])$, then $$_{[r]}{\downharpoonleft}[\mathfrak{p}]=[_{r}{\downharpoonleft}\mathfrak{p}],\,\, [\mathfrak{p}]{\downharpoonright}_{s}=[\mathfrak{p}{\downharpoonright}_{s}].$$
\item For all $p,q\in P$,
$$p\leq_{\mathscr{P}} q  \Longleftrightarrow p\leq_P q \Longleftrightarrow  [p] \leq_{P_{\mathscr{C}(P)}}[q] \Longleftrightarrow [p]\leq_{\approx_P} [q].$$
\end{itemize}
To prove that $(\mathscr{C}(P), \circ, \dd, \rr, \leq_{\approx_P}, P_{(\mathscr{C}(P)}, \times, \star)$ is a projection ordered category, we need some calculations.
Let $\mathfrak{c}=[p_{1},p_{2},\ldots,p_{k}]\in \mathscr{C}(P)$. ({\em Notice that $q$ denotes $[q]$ and $p_i$ denotes $[p_i]$ in the sequel}).
Then $\dd(\mathfrak{c})=p_{1}$ and $\rr(\mathfrak{c})=p_{k}$.
By (\ref{v}) and Lemma \ref{o},
\begin{equation}\label{daxzo}
\nu_{\mathfrak{c}}: p_{1}{\downarrow}\rightarrow p_{k}{\downarrow},\,\ q\mapsto q\nu_{\mathfrak{c}}=\rr(_{q}{\downharpoonleft}\mathfrak{c})=q\theta_{p_{1}}\theta_{p_{2}}\ldots\theta_{p_{k}},
\end{equation}
$$
\Theta_{\mathfrak{c}}: P_{\mathscr{C}(P)}\rightarrow P_{\mathscr{C}(P)}, \, q\mapsto q\Theta_{\mathfrak{c}}=q\theta_{\dd(\mathfrak{c})}\nu_{\mathfrak{c}}=q\theta_{p_{1}} \theta_{p_{1}}\theta_{p_{2}}\ldots\theta_{p_{k}}=q \theta_{p_{1}}\theta_{p_{2}}\ldots\theta_{p_{k}}.
$$
This implies that
\begin{equation}\label{yida}
\Theta_{\mathfrak{c}}=\theta_{p_{1}}\theta_{p_{2}}\ldots\theta_{p_{k}}.
\end{equation}
Dually,
\begin{equation}\label{edsc}
\mu_{\mathfrak{c}}: p_{k}{\downarrow}\rightarrow p_{1}{\downarrow},\,\, q\mu_{\mathfrak{c}}=\dd( \mathfrak{c}{\downharpoonright}_{q})=q\delta_{p_{k}}\delta_{p_{k-1}}\ldots\delta_{p_{1}},
\end{equation}
\begin{equation}\label{mifg}
\Delta_{\mathfrak{c}}=\delta_{p_{k}}\delta_{p_{k-1}}\ldots\delta_{p_{1}}.
\end{equation}
\begin{lemma}\label{drfg}Let $(P, \times, \star)$ be a strong projection algebra. Then $$(\mathscr{C}(P), \circ, \dd, \rr, \leq_{\approx_P}, P_{\mathscr{C}(P)}, \times, \star)$$ is a projection ordered category.
\end{lemma}
\begin{proof}
Let $\mathfrak{c}=[p_{1},p_{2},\ldots,p_{k}]\in \mathscr{C}(P)$ ({\em Notice that $q$ denotes $[q]$ and $p_i$ denotes $[p_i]$ in the sequel}) and  $q\leq_{\approx_P} \dd(\mathfrak{c})=p_{1}$.
By (\ref{daxzo}), Corollary \ref{diav}, (\ref{mifg}) and (\ref{yida}),
$$\delta_{q\nu_{\mathfrak{c}}}=\delta_{q\theta_{p_{1}}\theta_{p_{2}}\ldots\theta_{p_{k}}}=\delta_{p_{k}}\delta_{p_{k-1}}\ldots\delta_{p_{1}}
\delta_{q}\theta_{p_{1}}\theta_{p_{2}}\ldots\theta_{p_{k}}=\Delta_{\mathfrak{c}}\delta_{q}\Theta_{\mathfrak{c}}.$$
Dually, we have $\theta_{q\mu_{\mathfrak{c}}}=\Theta_{\mathfrak{c}}\theta_{q}\Delta_{\mathfrak{c}}$ for all $q\in P$ with $q\leq_{\approx_P}  \rr(\mathfrak{c})=p_{k}$.  Thus (G1a) is true, and so the result follows.
\end{proof}
The following proposition is obvious.

\begin{prop}\label{ahldr}
Let $(P, \times, \star)$ be a strong projection algebra. Then the natural map $$\varepsilon_P=\approx^\natural_P: {\mathscr P}(P_{{\mathscr C}(P)})\rightarrow {\mathscr C}(P),\, ([p_1], [p_2], \ldots, [p_k])\mapsto [p_1, p_2, \ldots, p_k]$$ is an evaluation map on $(\mathscr{C}(P), \circ, \dd, \rr, \leq_{\approx_P}, P_{\mathscr{C}(P)}, \times, \star)$.
\end{prop}

\begin{prop}\label{vnmu}Let $(P, \times, \star)$ be a strong projection algebra. Then $(\mathscr{C}(P), \circ, \dd, \rr, \leq_{\approx_{P}},  P_{\mathscr{C}(P)}, \times, \star, \varepsilon_P)$ is a chain projection ordered category.
\end{prop}
\begin{proof}Let $[\mathfrak{p}]\in \mathscr{C}(P)$ and $[e], [f]\in P_{ \mathscr{C}(P)}$, where $\mathfrak{p}=(p_{1},p_{2},\ldots,p_{k})\in \mathscr{P}(P)$ and $e, f\in P$.  Assume that $([e], [f])$ is a $[\mathfrak{p}]$-linked pair. Then
\begin{equation}\label{ddewq}
[f]=[e]\Theta_{[\mathfrak{p}]}\theta_{[f]},\,[e]=[f]\Delta_{[\mathfrak{p}]}\delta_{[e]}.
\end{equation}
By the definition of $\varepsilon_P$ and (\ref{fff})--(\ref{mmm}),
\begin{equation}\label{qwgh}
\lambda([e], [\mathfrak{p}],[f])=[e,e_{1}]~\circ~_{[e_{1}]}{\downharpoonleft[\mathfrak{p}]}~\circ~[f_{1},f],\,\,\,\rho([e], [\mathfrak{p}],[f])=[e,e_{2}]~\circ~ [\mathfrak{p}]\downharpoonright_{[f_{2}]}~\circ~[f_{2},f],
\end{equation}
where $$[e_{1}]=[e]\theta_{[p_{1}]},\,[e_{2}]=[f]\Delta_{[\mathfrak{p}]},\,[f_{1}]=[e]\Theta_{[\mathfrak{p}]},\,[f_{2}]=[f]\delta_{[p_{k}]},\,\,\, e_1, e_2, f_1, f_2\in P.$$
Denote
$$_{[e_{1}]}{\downharpoonleft[\mathfrak{p}]}=[_{e_1}\downharpoonleft\mathfrak{p}]=[u_{1},\ldots u_{k}],\,\,\,\, {[\mathfrak{p}]\downharpoonright}_{[f_{2}]}=[{\mathfrak{p}\downharpoonright}_{f_{2}}]=[v_{1},\ldots v_{k}].$$
By (\ref{u}), (\ref{v}) and the facts  $e_{1}=e\theta_{p_{1}}$ and $f_{2}=f\delta_{p_{k}}$, for all $1\leq i\leq k$,  we have
$$u_{i}=e\theta_{p_{1}}\ldots\theta_{p_{i}},\,v_{i}=f\delta_{p_{k}}\ldots\delta_{p_{i}}.$$
This together with (\ref{qwgh}) gives that
$$\lambda([e], [\mathfrak{p}],[f])=(e,u_{1},\ldots,u_{k},f),\,\,\,\,\rho([e], [\mathfrak{p}],[f])=(e,v_{1},\ldots,v_{k},f).$$
Denote $u_{0}=e$ and $v_{k+1}=f$. We shall show that $(u_{i-1},p_{i},v_{i+1})$ is an admissible triple of $P$ for all $1\leq i\leq k$, that is,
$$v_{i+1}=u_{i-1}\theta_{p_{i}}\theta_{v_{i+1}},\,\,\,\,\,\,u_{i-1}=v_{i+1}\delta_{p_{i}}\delta_{u_{i-1}}.$$
In view of (\ref{yida}), (\ref{mifg}) and (\ref{ddewq}), we have
$$f=e\theta_{p_{1}}\cdots\theta_{p_{k}}\theta_{f},\,e=f\delta_{p_{k}}\cdots\delta_{p_{1}}\delta_{e}.$$
By Lemma \ref{diav},
$$u_{i-1}\theta_{p_{i}}\theta_{v_{i+1}}=e\theta_{p_{1}}\cdots\theta_{p_{i-1}}\theta_{p_{i}}\theta_{f\delta_{p_{k}}\cdots\delta_{p_{i+1}}}
$$$$=e\theta_{p_{1}}\cdots\theta_{p_{i-1}}\theta_{p_{i}}\theta_{p_{i+1}}\cdots\theta_{p_{k}}\theta_{f}\delta_{p_{k}}\cdots\delta_{p_{i+1}}
=f\delta_{p_{k}}\cdots\delta_{p_{i+1}}=v_{i+1}.$$
Dually,
$$v_{i+1}\delta_{p_{i}}\delta_{u_{i-1}}=f\delta_{p_{k}}\ldots\delta_{p_{i+1}}\delta_{p_{i}}\delta_{e\theta_{p_{1}}\cdots\theta_{p_{i-1}}}
$$$$=f\delta_{p_{k}}\cdots\delta_{p_{i+1}}\delta_{p_{i}}\delta_{p_{i-1}}\cdots\delta_{p_{i}}\delta_{e}\theta_{p_{1}}\cdots\theta_{p_{i-1}}
=e\theta_{p_{1}}\cdots\theta_{p_{i-1}}=u_{i-1}.$$
Thus $(u_{i-1},p_{i},v_{i+1})$ is an admissible triple for all $1\leq i\leq k$, and so
$(u_{i-1},u_{i-1}\theta_{p_{i}},\nu_{i+1})=(u_{i-1},u_{i},\nu_{i+1})$ is $\approx^\natural_P$-related to
$(u_{i-1},\nu_{i+1}\delta_{p_{i}},\nu_{i+1})=(u_{i-1},v_{i},\nu_{i+1})$.
Therefore
\begin{equation}\label{ffds}
[u_{i-1},u_{i},\nu_{i+1}]=[u_{i-1},v_{i},\nu_{i+1}] \mbox{ for all }1\leq i\leq k.
\end{equation}
This gives that
$$\rho([e], [\mathfrak{p}],[f])=[e,v_{1},\ldots,v_{k},f]=[u_{0},v_{1},v_{2},v_{3},v_{4},\ldots,v_{k},v_{k+1}]$$$$=
[u_{0},u_{1},v_{2},v_{3},v_{4},\ldots,v_{k},v_{k+1}]=[u_{0},u_{1},u_{2},v_{3},v_{4},\ldots,v_{k},v_{k+1}]$$$$=\cdots
[u_{0},u_{1},u_{2},u_{3},u_{4},\ldots,u_{k},v_{k+1}]=[e,u_{1},\ldots,u_{k},f]=\lambda([e], [\mathfrak{p}],[f])$$
Hence (G2) holds, and the desired result now follows.
\end{proof}

\begin{defn} Let $(P, \times \star)$ be a strong projection algebra. By Proposition \ref{vnmu} and Theorem \ref{xii}, we have the DRC-restriction semigroup $\textbf{S}(\mathscr{C}(P))$ and call it the chain semigroup determined by $(P, \times, \star)$.
\end{defn}
Let $$[\mathfrak{p}]=[p_{1},p_{2}, \ldots, p_{k}], \, [\mathfrak{q}]=[q_{1},q_{2},\ldots,q_{l}]\in \mathscr{C}(P).$$  Then $\rr([\mathfrak{p}])=[p_{k}], \dd([\mathfrak{q}])=[q_{1}]$.
Denote $[p']=[p_k\times q_1], [q']=[p_k\star q_1]$.
By the definition of $\varepsilon_P$, we have
$$[\mathfrak{p}]\bullet [\mathfrak{q}]=[\mathfrak{p}]{\downharpoonright}{_{[p']}} \circ~ [p', q']\circ~{ _{[q']}}{\downharpoonleft}[\mathfrak{q}].$$ By (\ref{v}) and (\ref{win}) we let $[\mathfrak{p}]{\downharpoonright}{_{[p']}}=[p_{1}{'}, \ldots, p_{k}{'}],\,\,\, {_{[q{'}]}}{\downharpoonleft} [\mathfrak{q}]=[q_{1}{'},\ldots, q_{l}{'}],$ where $p{'}=p_{k}{'},  q{'}=q_{1}{'}$ and  $$p_{i}{'}=p{'}\delta_{p_{k}}\cdots\delta_{p_{i}},\, q_{j}{'}=q{'}\theta_{q_{1}}\cdots\theta_{q_{j}}, i=1,2, \ldots, t, j=1,2, \ldots, l.$$
This implies that
$$[\mathfrak{p}]\bullet [\mathfrak{q}]=[p'_{1},\ldots,p'_{k}]\circ [p'_{k},q'_{1}]\circ [q_{1}{'},\ldots,q_{l}{'}]=[p_{1}{'},\ldots,p_{k}{'},q_{1}{'},\ldots,q_{l}{'}],$$
$$[\mathfrak{p}]^+=\dd([\mathfrak{p}])=[p_{1}],\,\, [\mathfrak{p}]^{\ast}=\rr([\mathfrak{p}])=[p_{k}].$$
Moreover, $P(\textbf{S}(\mathscr{C}(P)))=P_{\mathscr{C}(P)}=\{[p]\mid p\in P\}.$

\begin{prop}\label{lope}Let $(C,\circ,\dd,\rr,\leq, P_C, \times,\star,\varepsilon)$ be a chain projection ordered category. Then
$$\overline{\varepsilon}: \mathscr{C}(P_{C})\rightarrow C,\,\,  [\mathfrak{p}]\mapsto \varepsilon (\mathfrak{p})$$
is a chain projection ordered functor form $\mathscr{C}(P_{ C})$ to $\mathcal C$, and so is a (2,1,1)-homomorphism from to ${\bf S}(\mathscr{C}(P_{ C}))$ to ${\bf S}( C)$ such that $[p]\overline{\varepsilon}=p.$
\end{prop}
\begin{proof}
We first show that $\overline{\varepsilon}$ is well-defined. Let $\mathfrak{p},\mathfrak{q}\in \mathscr{P}(P_{C})$ and $\mathfrak{p}=\mathfrak{b}\circ \mathfrak{s} \circ \mathfrak{c},\mathfrak{q}=\mathfrak{b}\circ \mathfrak{t} \circ \mathfrak{c}$, where $\mathfrak{b},\mathfrak{c}\in \mathscr{P}(P_{ C}),(\mathfrak{s},\mathfrak{t})\in \Omega(P_{ C})\cup  \Omega^{-1}(P_{ C})$. If $\mathfrak{s}=(p,p),\mathfrak{t}=p$ or $\mathfrak{s}=p,\mathfrak{t}=(p,p)$, where $p\in P_{ C} $, then by (E1) we have $\varepsilon(p,p)=p=\varepsilon(p)$, and so $\varepsilon(\mathfrak{s})=\varepsilon(\mathfrak{t})$. Assume that $(e,p,f)$ is an admissible triple in $P_{ C}$. Then
\begin{equation}\label{ccsgh}
f=e\theta_{p}\theta_{f},e=f\delta_{p}\delta_{e}
\end{equation}
Let $\mathfrak{s}=(e,e\theta_{p},f),\mathfrak{t}=(e,f\delta_{p},f)$ or $\mathfrak{s}=(e,f\delta_{p},f),\mathfrak{t}=(e,e\theta_{p},f)$. We assume the former case holds without loss of generality. By (\ref{pp}),
\begin{equation}\label{xbuy}
\varepsilon(\mathfrak{s})=\varepsilon(e,e\theta_{p})\circ \varepsilon(e\theta_{p},f),\varepsilon(\mathfrak{t})=\varepsilon(e,f\delta_{p})\circ \varepsilon(f\delta_{p},f).
\end{equation}
By using the fact that $p\in P_{ C}$ and (\ref{ccwio}), we have $\Theta_{p}=\theta_{p},\Delta_{p}=\delta_{p},$ and hence $f=e\Theta_{p}\theta_{f}$ and $e=f\Delta_{p}\delta_{e}$ by (\ref{ccsgh}). This implies that $(e,f)$ is (in the sense of Definition \ref{lihfg}) $p$-linked pair. By (G2), (\ref{nnn}), (\ref{mmm}), (\ref{fff}) and the fact that $\dd(p)=p$, we have $$\varepsilon(e,e_{1})\circ{_{e_{1}}}{\downharpoonleft}p \circ \varepsilon(f_{1},f)=\varepsilon(e,e_{2})\circ{_{e_{2}}}{\downharpoonleft}p \circ \varepsilon(f_{2},f),$$
where $$e_{1}=e\theta_{\dd(p)}=e\theta_{p},e_{2}=f\Delta_{p}=f\delta_{p},f_{1}=e\Theta_{p}=e\theta_{p},f_{2}=f\delta_{\rr(p)}=
f\delta_{p}.$$ This gives that
\begin{equation}\label{erfui}
\varepsilon(e,e\theta_{p})\circ{_{e\theta_{p}}}{\downharpoonleft}p \circ \varepsilon(e\theta_{p},f)=\varepsilon(e,f\delta_{p})\circ{_{f\delta_{p}}}{\downharpoonleft}p \circ \varepsilon(f\delta_{p},f).
\end{equation}
Using the fact that $e\theta_{p}\leq p,f\delta_{p}\leq p$ and Lemma \ref{b} (11), we get ${_{e\theta_{p}}}{\downharpoonleft}p=e\theta_{p}$ and ${_{f\delta_{p}}}{\downharpoonleft}p=f\delta_{p}.$ Observe that $\rr(\varepsilon(e,e\theta_{p}))=e\theta_{p}$ and $\rr(\varepsilon(e,f\delta_{p}))=f\delta_{p}$ by Lemma \ref{oo}. This together with  (C3) provides
\begin{equation}\label{fhkwe}
\varepsilon(e,e\theta_{p})\circ{_{e\theta_{p}}}{\downharpoonleft}p=\varepsilon(e,e\theta_{p}),\,\,
\varepsilon(e,f\delta_{p})\circ{_{f\delta_{p}}}{\downharpoonleft}p=\varepsilon(e,f\delta_{p}).
\end{equation}
By (\ref{xbuy}), (\ref{erfui}) and (\ref{fhkwe}), we have $\varepsilon(\mathfrak{s})=\varepsilon(\mathfrak{t}).$ The above discussion shows that  $\varepsilon(\mathfrak{s})=\varepsilon(\mathfrak{t})$ in either case. Thus by (E3), $$\varepsilon(\mathfrak{p})=\varepsilon(\mathfrak{b}\circ \mathfrak{s} \circ \mathfrak{c})=\varepsilon(\mathfrak{b})\circ \varepsilon(\mathfrak{s}) \circ \varepsilon(\mathfrak{c})=\varepsilon(\mathfrak{b})\circ \varepsilon(\mathfrak{t}) \circ \varepsilon(\mathfrak{c})=\varepsilon(\mathfrak{b}\circ \mathfrak{t} \circ \mathfrak{c})=\varepsilon(\mathfrak{q}).$$
Let $\mathfrak{p},\mathfrak{q}\in \mathscr{P}(P_{ C})$ and $[\mathfrak{p}]=[\mathfrak{q}]$. Then $\mathfrak{p}\approx_{P_{ C}}\mathfrak{q}.$ By Lemma \ref{ting2} and the above statements, we have $\varepsilon(\mathfrak{p})=\varepsilon(\mathfrak{q})$. This shows that $\overline{\varepsilon}$ is well-defined.

In the sequel, we show that $\overline{\varepsilon}$ is a chain projection ordered functor. Let $$[\mathfrak{p}],[\mathfrak{q}]\in \mathscr{C}(P_{ C}),\, e,f\in P_{C},\,\mathfrak{p}=(p_{1},p_{2},\ldots,p_{k}),\mathfrak{q}=(q_{1},q_{2},\ldots,q_{k})\in \mathscr{P}(P_{ C}).$$

(1) Firstly, $$\overline{\varepsilon}(\dd([\mathfrak{p}]))=\overline{\varepsilon}([\dd(\mathfrak{p})])=\overline{\varepsilon}([p_{1}]
)=\varepsilon(p_{1})=p_{1}=\dd(\mathfrak{p})=\dd(\varepsilon(\mathfrak{p}))=\dd(\overline{\varepsilon}([\mathfrak{p}])).$$
Dually, $\overline{\varepsilon}(\rr([\mathfrak{p}]))=p_{k}=\rr(\overline{\varepsilon}([\mathfrak{p}]))$.

(2) If $\rr([\mathfrak{p}])=\dd([\mathfrak{q}])$, then $p_{k}=q_{1}$ and $$\overline{\varepsilon}([\mathfrak{p}]\circ[\mathfrak{q}])=\overline{\varepsilon}([\mathfrak{p}\circ \mathfrak{q}])=\varepsilon(\mathfrak{p} \circ \mathfrak{q})=\varepsilon(\mathfrak{p})\circ \varepsilon(\mathfrak{q})=\overline{\varepsilon}([\mathfrak{p} ])\circ \overline{\varepsilon}([\mathfrak{q} ]).$$

(3) Let $[\mathfrak{p}]\leq_{\approx_{P}} [\mathfrak{q}]$. Then there exist $\mathfrak{m},\mathfrak{n}\in \mathscr{P}(P_{ C})$ such that $[\mathfrak{p}]=[\mathfrak{m}],[\mathfrak{q}]=[\mathfrak{n}]$ and $\mathfrak{m}\leq_{\mathscr{P}} \mathfrak{n}$ in $\mathscr{P}(P_{\mathcal C})$. Since $\varepsilon$ is order-preserving, we have $$\overline{\varepsilon}([\mathfrak{p} ])=\overline{\varepsilon}([\mathfrak{m} ])=\varepsilon(\mathfrak{m})\leq \varepsilon(\mathfrak{n})=\overline{\varepsilon}([\mathfrak{n} ])=\overline{\varepsilon}([\mathfrak{q} ]).$$

(4) Let $[p],[q]\in P_{\mathscr{C}(P_{ C})}$, where $p,q\in P_{C}.$ Then $$\overline{\varepsilon}([p]\times [q])=\overline{\varepsilon}([p\times q])=\varepsilon(p\times q)=\varepsilon(p)\times \varepsilon(q)=\overline{\varepsilon}([p])\times \overline{\varepsilon}([q]).$$
    Dually, $\overline{\varepsilon}([p]\star [q])=\overline{\varepsilon}([p])\star \overline{\varepsilon}([q]).$

(5) Let $([p_{1}],[p_{2}],\ldots,[p_{k}])\in \mathscr{P}(P_{\mathscr{C}(P_{ C})})$. Then $$(\varepsilon_{P_{ C}}(([p_{1}],[p_{2}],\ldots,[p_{k}])))\overline{\varepsilon}=([p_{1},p_{2},\ldots,p_{k}])\overline{\varepsilon}$$$$=\varepsilon((p_{1},p_{2},\ldots,p_{k}))
=\varepsilon(\varepsilon(p_{1}),\varepsilon(p_{2}),\ldots,\varepsilon(p_{k}))=\varepsilon(([p_{1}]\overline{\varepsilon},[p_{2}]\overline{\varepsilon},\ldots,[p_{k}]\overline{\varepsilon})).$$
By Lemma \ref{wang}, $\overline{\varepsilon}$ is a (2,1,1)-homomorphism from ${\bf S}(\mathscr{C}(P))$ to ${\bf S}( C)$. Moreover, since $\varepsilon$ is an evaluation map, it follows that $([p])\overline{\varepsilon}=\varepsilon(p)=p$ for all $p\in P_{C}$ by the definition of $\overline{\varepsilon}$.
\end{proof}

\begin{theorem}\label{llpg} If $(P,\times,\star)$ be a strong projection algebra, then ${\bf S}(\mathscr{C}(P))$ is a projection-generated DRC-restriction semigroup with projection algebra $P({\bf S}(\mathscr{C}(P)))=\{[p]\mid p\in P\}$.   Conversely, if $(S,\cdot,^+,^\ast)$ is a projection-generated DRC-restriction semigroup whose projection algebra is $(P,\times,\star)$, then there exists a surjective (2,1,1)-homomorphism $\phi$ from  ${\bf S}(\mathscr{C}(P))$ onto $S$ such that $[p]\phi=p$ for all $p\in P$.
\end{theorem}
\begin{proof}
Since $\varepsilon_{P}$ (see Proposition \ref{ahldr}) is an evaluation map on $\mathscr{C}(P)$ and is also surjective, by  Proposition \ref{vnmu}, Theorem \ref{xii} and  Proposition \ref{fgpu},  ${\bf S}(\mathscr{C}(P))$ is a projection-generated DRC-restriction semigroup  with projection algebra $P({\bf S}(\mathscr{C}(P)))=P_{ \mathscr{C}(P)}=\{[p]\mid p\in P\}$.

Conversely, by Proposition \ref{lope} and Theorem \ref{hgdc}, $\overline{\varepsilon_{S}}$ is a (2,1,1)-homomorphism from ${\bf S}(\mathscr{C}(P_{{\bf C}(S)})={\bf S}(\mathscr{C}(P))$ to ${\bf S}({\bf C}(S))=S$ satisfying $[p]\overline{\varepsilon_{S}}=p$ for all $p\in P$. Since $S$ is projection-generated, it follows that $\varepsilon_{S}$ is surjective by Proposition \ref{fgpu}, and hence $\overline{\varepsilon_{S}}$ is also surjective.
\end{proof}

\begin{theorem}\label{cc}
Let $\psi$ be a homomorphism from the strong projection algebra  $(P,\times,\star)$ to the strong projection algebra $(P',\times,\star)$. Then
$$\Phi:\mathscr{C}(P)\rightarrow \mathscr{C}(P^{'}),\,\,\,[p_{1},p_{2},\ldots,p_{k}]\rightarrow [p_{1}\psi,p_{2}\psi,\ldots,p_{k}\psi]$$ is a chain projection ordered functor, and so is a (2,1,1)-homomorphism from ${\rm \bf{S}}(\mathscr{C}(P))$ to ${\rm \bf{S}}(\mathscr{C}(P'))$.
\end{theorem}

\begin{proof}
Let $$\xi:\mathscr{P}(P)\rightarrow \mathscr{P}(P'),\,\,\,(p_{1},p_{2},\ldots,p_{k})\mapsto (p_{1}\psi,p_{2}\psi,\ldots,p_{k}\psi).$$ Then $\xi$ is well-defined by
Lemma \ref{pin}. It is easy to prove that
\begin{equation}\label{xfgji}
(\dd(\mathfrak{p}))\xi=\dd(\mathfrak{p}\xi),(\rr(\mathfrak{p}))\xi=\rr(\mathfrak{p}\xi),(\mathfrak{p}\circ \mathfrak{q})\xi=(\mathfrak{p}\xi)\circ(\mathfrak{q}\xi)
\end{equation}
for all $\mathfrak{p},\mathfrak{q}\in \mathscr{P}(P)$. Let $(\mathfrak{s}, \mathfrak{t})\in \Omega(P)\cup  \Omega^{-1}(P)$. If $\mathfrak{s}=(p,p), \mathfrak{t}=p, p\in P,$ then $\mathfrak{s}\xi=(p\psi,p\psi)$ and $\mathfrak{t}\xi=p\psi$, and so $(\mathfrak{s}\xi,\mathfrak{t}\xi),(\mathfrak{t}\xi,\mathfrak{s}\xi)\in \Omega(P')\cup  \Omega^{-1}(P')$. Now let $(e,p,f)$ be an admissible triple of $P$ and $\mathfrak{s}=(e,e\star p,f),\mathfrak{t}=(e,p\times f,f)$. Then $(e\star p)\star f=f$ and $e=e\times(p\times f)$. Since $\psi$ is a homomorphism, it follows that $(e\psi,p\psi,f\psi)$ is an admissible triple of $P'$ and $$\mathfrak{s}\xi=(e\psi,(e\psi)\star (p\psi),f\psi),\,\,\,\mathfrak{t}\xi=(e\psi,(p\psi)\times (e\psi),f\psi).$$ This implies that $(\mathfrak{s}\xi,\mathfrak{t}\xi),(\mathfrak{t}\xi,\mathfrak{s}\xi)\in \Omega(P')\cup  \Omega^{-1}(P')$.

Let $\mathfrak{p},\mathfrak{q}\in \mathscr{P}(P)$ and $\mathfrak{p}=\mathfrak{b}\circ \mathfrak{s} \circ \mathfrak{c},\mathfrak{q}=\mathfrak{b}\circ \mathfrak{t} \circ \mathfrak{c}$, where $\mathfrak{b},\mathfrak{c}\in \mathscr{P}(P),(\mathfrak{s},\mathfrak{t})\in \Omega(P)\cup  \Omega^{-1}(P)$. Then by the above discussions, $(\mathfrak{s}\xi,\mathfrak{t}\xi)\in \Omega(P')\cup  \Omega^{-1}(P')$. By (\ref{xfgji}), $$\mathfrak{p}\xi=(\mathfrak{b}\circ \mathfrak{s} \circ \mathfrak{c})\xi=(\mathfrak{b}\xi)\circ (\mathfrak{s}\xi) \circ (\mathfrak{c}\xi),\mathfrak{q}\xi=(\mathfrak{b}\circ \mathfrak{t} \circ \mathfrak{c})\xi=(\mathfrak{b}\xi)\circ (\mathfrak{t}\xi) \circ (\mathfrak{c}\xi).$$ This implies that $\mathfrak{p}\xi\approx_{P'}\mathfrak{q}\xi$, i.e. $[\mathfrak{p}\xi]=[\mathfrak{q}\xi].$

Let $\mathfrak{p},\mathfrak{q}\in \mathscr{P}(P)$ and $[\mathfrak{p}]=[\mathfrak{q}]$. Then $\mathfrak{p}\approx_{P} \mathfrak{q}$. By Lemma \ref{ting2} and the statements in previous paragraph, we have $[\mathfrak{p}]\Phi=[\mathfrak{p}\xi]=[\mathfrak{q}\xi]=[\mathfrak{q}]\Phi$. Thus $\Phi$ is well-defined.

Let $\mathfrak{p}, \mathfrak{q}\in \mathscr{P}(P)$ and $\mathfrak{p}=(p_{1},p_{2},\ldots,p_{k}),\mathfrak{q}=(q_{1},q_{2},\ldots,q_{l})$. Then $$(\dd([\mathfrak{p}]))\Phi=[p_{1}]\Phi=[p_{1}\psi]=\dd([\mathfrak{p}]\Phi).$$ Dually, $(\rr([\mathfrak{p}]))\Phi=\rr([\mathfrak{p}]\Phi)$. If $\rr([\mathfrak{p}])=\dd([\mathfrak{p}])$, i.e. $p_{k}=q_{1}$, by (\ref{xfgji}) we have $$([\mathfrak{p}]\circ[\mathfrak{q}])\Phi=[\mathfrak{p} \circ \mathfrak{q}]\Phi=[(\mathfrak{p} \circ \mathfrak{q})\xi]=[(\mathfrak{p}\xi)\circ(\mathfrak{q}\xi)]=[\mathfrak{p}\xi] \circ [q\xi]=([\mathfrak{p}]\Phi)\circ([\mathfrak{q}]\Phi).$$
This shows that (F1) and (F2) hold.

Assume that $[\mathfrak{p}]\leq_{\approx_P} [\mathfrak{q}]$. Then there exist $\mathfrak{m}\in [\mathfrak{p}]$ and $\mathfrak{n}\in [\mathfrak{q}]$ such that $\mathfrak{m}\leq_{\mathscr{P}} \mathfrak{n}$, and so there exists $r\in P$ such that $r\leq_{P} \dd(\mathfrak{n})$ and $\mathfrak{m}={_r}{\downharpoonleft}\mathfrak{n}$. Denote $\mathfrak{m}=(m_{1},m_{2},\ldots,m_{u})$ and $\mathfrak{n}=(n_{1},n_{2},\ldots,n_{v})$. Then by (\ref{v}), $$\mathfrak{m}=(r,r\theta_{n_{2}},\ldots,r\theta_{n_{2}}\ldots\theta_{n_{v}}),\,\, \mathfrak{n}\xi=(n_{1}\psi,n_{2}\psi,\ldots,n_{v}\psi).$$
Since $\psi$ is a homomorphism between projection algebras, we have $r\psi\leq_{P'}(\dd(\mathfrak{n}))\psi=n_{1}\psi=\dd(\mathfrak{n}\xi)$ and $$\mathfrak{m}\xi=(r\psi,(r\theta_{n_{2}})\psi,\ldots,(r\theta_{n_{2}}\theta_{n_{3}}\ldots\theta_{n_{v}})\psi) =(r\psi,r\psi\theta_{n_{2}\psi},\ldots,r\psi\theta_{n_{2}\psi}\ldots\theta_{n_{v}\psi})
$$$$={_{r\psi}}{\downharpoonleft}(n_{1}\psi,n_{1}\psi,\ldots,n_{v}\psi) ={_{r\psi}}{\downharpoonleft}\mathfrak{n}\xi.$$
This shows that $\mathfrak{m}\xi\leq_{\mathscr{P}} \mathfrak{n}\xi,$  and so $[\mathfrak{m}]\Phi=[\mathfrak{m}\xi]\leq_{\approx_{P'}} [\mathfrak{n}\xi]=[\mathfrak{n}]\Phi$. As $\mathfrak{m}\in [\mathfrak{p}]$ and $\mathfrak{n}\in [\mathfrak{q}]$,  we have $[\mathfrak{m}]=[\mathfrak{p}]$ and $[\mathfrak{n}]=[\mathfrak{q}]$. Hence $[\mathfrak{p}]\Phi=[\mathfrak{m}]\Phi \leq_{\approx_{P'}}[\mathfrak{n}]\Phi=[\mathfrak{q}]\Phi.$ This shows that $\Phi$ satisfies (F3).

Take $[p],[q]\in P_{\mathscr{C}(P)},p,q\in P$ arbitrarily.  Then $$([p]\times[q])\Phi=([p\times q])\Phi=[(p\times q)\psi]=[(p\psi)\times (q\psi) ]=[p\psi]\times[q\psi].$$ Dually, $([p]\star[q])\Phi=[p\psi]\star[q\psi].$ This gives (F4).
Let $([p_{1}],[p_{2}],\ldots,[p_{k}])\in \mathscr{P}(P_{\mathscr{C}(P)})$. Then $$(\varepsilon_{P}([p_{1}],[p_{2}],\ldots,[p_{k}]))\Phi=[p_{1},p_{2},\ldots,p_{k}]\Phi
 =[p_{1}\psi,p_{2}\psi,\ldots,p_{k}\psi]
$$$$=\varepsilon_{P'}(([p_{1}\psi],[p_{2}\psi],\ldots,[p_{k}\psi]))
 =\varepsilon_{P'}([p_{1}]\Phi,[p_{2}]\Phi,\ldots,[p_{k}]\Phi).$$
Thus (F5) holds. This implies that $\Phi$ is a chain projection ordered functor. By Lemma \ref{wang}, $\Phi$ is a (2,1,1)-homomorphism from $\textbf{S}(\mathscr{C}(P))$ to $\textbf{S}(\mathscr{C}(P'))$.
\end{proof}

In the remainder of this section, we consider the presentations of chain semigroups. We first recall the notion of presentations of semigroups.  Let $X$ be a nonempty set and $X^{+}$ be the free semigroup on $X$. Let $R\subseteq X^{+}\times X^{+}$ and $R^{\sharp}$ be the congruence on $X^{+}$  generated by $R$. Let $S$ be a semigroup.  If $S\cong X^{+}/R^{\sharp}$, then  we say that $S$ has presentation $\langle X: R \rangle$.  In this case, we shall identify $S$ with $X^{+}/R^{\sharp}$, and call the elements in $X$ and $R$ the {\em generators and generations} of $S$, respectively. Obviously, $S$ has presentation $\langle X:R \rangle$ if and only if there exists an surjective semigroup homomorphism $\pi: X^{+}\rightarrow S$  such that the kernel of $\pi$  is exactly $R^{\sharp}$.

Let $(P,\times, \star)$ be a strong projection algebra. The following theorem gives a presentation of the chain semigroup $\textbf{S}(\mathscr{C}(P))$.
\begin{theorem}\label{scab}
Let $(P,\times, \star)$ be a strong projection algebra. Then the chain semigroup $\textbf{S}(\mathscr{C}(P))$ has a presentation $\textbf{S}(\mathscr{C}(P))\cong \langle X_{P}:R_{P}\rangle$, where $X_{P}=\{x_{p}\mid p\in P\}$ is an alphabet with a correspondence to $P$ and  $R_{P}$ is the set of the following relations: For all $p,q\in P$,
\begin{itemize}
\item[\rm(R1)]$x_{p}^{2}=x_{p}.$
\item[\rm(R2)]$x_{p\times q}x_{q}=x_{p}x_{q}.$
\item[\rm(R3)]$x_{p\times q}x_{p\ast q}=x_{p}x_{q}.$
\item[\rm(R4)]$x_{p}x_{p\ast q}=x_{p}x_{q}.$
\end{itemize}
\end{theorem}
To prove Theorem \ref{scab}, we need a series of lemmas. For convenience, we denote  $\sim=R^{\sharp}_{P}$, the congruence of $X^{+}_{P}$ generated by $R_{P}$.

\begin{lemma}\label{dhji}
Let $(P,\times,\star)$ be a strong projection algebra and $k$ be a positive integer. Let $$p_{1},p_{2}, \ldots, p_{k}, q\in P,\,\, p_{1}\mathcal{F}_{P} p_{2}\mathcal{F}_{P} \ldots \mathcal{F}_{P} p_{k},\,\, q\leq_P p_{k}.$$ Then there exist $q_{1},q_{2},\ldots,q_{k-1}\in P$ such that $$q_{1}\,\mathcal{F}_{P} \,q_{2}\,\mathcal{F}_{P} \ldots \mathcal{F}_{P}\, q_{k-1}\,\mathcal{F}_{P}\, q,\,\,\, x_{p_{1}}x_{p_{2}}\ldots x_{p_{k-1}}x_{q}\sim x_{q_{1}}x_{q_{2}}\ldots x_{q_{k-1}}x_{q}.$$
\end{lemma}
\begin{proof}If $k=1$, then $x_{q}\sim x_{q}$ obviously. Let $k\geq 2$. Denote $q_{k-1}=p_{k-1}\times q$. Then $q_{k-1}\leq_P p_{k-1}$ by Lemma \ref{jiben1} (1) .
By mathematical induction, there exist $q_{1},q_{2},\ldots,q_{k-2}\in P$ such that
\begin{equation}\label{xlwfa}
q_{1}\,\mathcal{F}_{P}\, q_{2}\,\mathcal{F}_{P}\, \ldots \mathcal{F}_{P}\, q_{k-2}\,\mathcal{F}_{P}\, q_{k-1},\,\,\,\, x_{p_{1}}x_{p_{2}}\ldots x_{p_{k-2}}x_{q_{k-1}}\sim x_{q_{1}}x_{q_{2}}\ldots x_{q_{k-2}}x_{q_{k-1}}.
\end{equation}
By (R2),
\begin{equation}\label{wiop}
x_{p_{k-1}}x_{q} {~\sim~} x_{p_{k-1}\times q}x_{q}=x_{q_{k-1}}x_{q}.
\end{equation}
By (\ref{xlwfa}) and (\ref{wiop}),  we have
$$x_{p_{1}}x_{p_{2}}\ldots x_{p_{k-2}}x_{p_{k-1}}x_{q}\sim x_{p_{1}}x_{p_{2}}\ldots x_{p_{k-2}}x_{q_{k-1}}x_{q} \sim x_{q_{1}}x_{q_{2}}\ldots x_{q_{k-2}}x_{q_{k-1}}x_{q}.$$
We shall prove $q_{k-1}\,\mathcal{F}_{P}\, q$ in the sequel. On one hand, by the fact that $q_{k-1}=p_{k-1}\times q$, (L4) and (L1), we obtain
$$q_{k-1}\times q=(p_{k-1}\times q)\times q=(p_{k-1}\times q)\times(p_{k-1}\times q)=p_{k-1}\times q=q_{k-1}.$$
On the other hand, the fact $p_{k-1}\mathcal{F}_{P} p_{k}$ gives that $p_{k}=p_{k-1}\star p_{k}$, and the fact $q\leq p_{k}=p_{k-1}\star p_{k}$ and (\ref{wang2}) imply $q=(p_{k-1}\times q)\star p_{k}$. By (R4), (R1) and the above statements, $$q_{k-1}\star q=(p_{k-1}\times q)\star ((p_{k-1}\times q)\star p_{k})= $$$$((p_{k-1}\times q)\star p_{k})\star  ((p_{k-1}\times q)\star p_{k})= ((p_{k-1}\times q)\star p_{k})=q.$$
By (\ref{s}), $q_{k-1}\,\mathcal{F}_{P}\, q.$
\end{proof}

\begin{lemma}\label{abgq}
Let $(P,\times,\star)$ be a strong projection algebra and $w=x_{q_{1}}x_{q_{2}}\cdots x_{q_{k}}\in X_{P}^{+}$. Then there exist $p_{1},p_{2},\ldots,p_{k}\in P$ such that $p_{1}\mathcal{F}_{P} p_{2}\mathcal{F}_{P} \cdots \mathcal{F}_{P} p_{k}$ and $w\sim x_{p_{1}}x_{p_{2}}\cdots x_{p_{k}}.$
\end{lemma}
\begin{proof}
When $k=1$, it is obvious that $x_{q_{1}}\sim x_{q_{1}}$. Let $k\geq 2$ and $w=x_{q_{1}}x_{q_{2}}\cdots x_{q_{k}}$. By mathematical induction, there exist $r_{1},r_{2},\ldots,r_{k-1}\in P$ such that $$r_{1}\mathcal{F}_{P} r_{2}\mathcal{F}_{P} \cdots \mathcal{F}_{P} r_{k-1},\,\,\, x_{q_{1}}x_{q_{2}}\cdots x_{q_{k-1}}\sim  x_{r_{1}}x_{r_{2}}\cdots x_{r_{k-1}}.$$
Denote $p_{k-1}=r_{k-1}\times q_{k}\,\, p_{k}=r_{k-1}\star {q_{k}}.$ By Lemma \ref{p},  we have $p_{k-1}\,\mathcal{F}_{P}\, p_{k}$, and so
$$x_{r_{k-1}}x_{q_{k}}\overset{\rm(R3)}{\sim}x_{r_{k-1}\times q_{k}}x_{r_{k-1}\star {q_{k}}}=x_{p_{k-1}}x_{p_{k}}.$$
Moreover, the fact $p_{k-1}=r_{k-1}\times q_{k}$ gives that $p_{k-1}\leq_P r_{k-1}$ by Lemma \ref{jiben1} (1). By Lemma \ref{dhji}, there exist $p_{1},p_{2},,\ldots,p_{k-2}\in P$
such that $$p_{1}\mathcal{F}_{P} p_{2}\mathcal{F}_{P} \cdots \mathcal{F}_{P} p_{k-2}\mathcal{F}_{P} p_{k-1},\,\, x_{r_{1}}x_{r_{2}}\cdots x_{r_{k-2}}x_{p_{k-1}}\sim  x_{p_{1}}x_{p_{2}}\cdots x_{p_{k-2}}x_{p_{k-1}}.$$
By the above statements, $$p_{1}\mathcal{F}_{P} p_{2}\mathcal{F}_{P} \cdots \mathcal{F}_{P} p_{k-2}\mathcal{F}_{P} p_{k-1}\mathcal{F}_{P} p_{k},$$
$$w=x_{q_{1}}x_{q_{2}}\cdots x_{q_{k-1}} x_{q_{k}}\sim x_{r_{1}}x_{r_{2}}\cdots x_{r_{k-2}}x_{r_{k-1}}x_{q_{k}}$$$$\sim x_{r_{1}}x_{r_{2}}\cdots x_{r_{k-2}}x_{p_{k-1}}x_{p_{k}}\sim x_{p_{1}}x_{p_{2}}\cdots x_{p_{k-1}} x_{p_{k}}.$$
By mathematical induction, the result follows.
\end{proof}

Let $(P,\times,\star)$ be a strong projection algebra and $\mathfrak{p}=(p_{1},p_{2},\ldots,p_{k})\in \mathscr{P}(P)$. Define
$$w_{\mathfrak{p}}= x_{p_{1}}x_{p_{2}}\cdots x_{p_{k}}\in X_{P}^{+}.$$
Let $w\in X_{P}^{+}.$ By Lemma \ref{abgq}, there exists $\mathfrak{p}\in \mathscr{P}(P)$ such that $w\sim w_{\mathfrak{p}}$. Let $\mathfrak{p},\mathfrak{q}\in \mathscr{P}(P)$ and  $\rr(\mathfrak{p})=\dd(\mathfrak{q})$. By (R1),
\begin{equation}\label{qkla}
w_{\mathfrak{p}}w_{\mathfrak{q}}\sim w_{\mathfrak{p}\circ \mathfrak{q}}.
\end{equation}
\begin{lemma}\label{slrgf}
Let $(P,\times,\star)$ be a strong projection algebra, $\mathfrak{p}, \mathfrak{q}\in \mathscr{P}(P)$ and $\mathfrak{p}\approx_{P}\mathfrak{q}$. Then $w_{\mathfrak{p}}\sim w_{\mathfrak{q}}$.
\end{lemma}
\begin{proof}
By Lemma \ref{ting2}, we assume that $\mathfrak{p},\mathfrak{q}\in \mathscr{P}(P)$ and $\mathfrak{p}=\mathfrak{b}\circ \mathfrak{s} \circ \mathfrak{c},\mathfrak{q}=\mathfrak{b}\circ \mathfrak{t} \circ \mathfrak{c}$, where $$\mathfrak{b},\mathfrak{c}\in \mathscr{P}(P),(\mathfrak{s},\mathfrak{t})\in \Omega(P)\cup  \Omega^{-1}(P).$$ By (\ref{qkla}), we obtain that $w_{\mathfrak{p}}\sim w_{\mathfrak{b}}w_{\mathfrak{s}}w_{\mathfrak{c}}$ and $ w_{\mathfrak{p}}\sim w_{\mathfrak{b}} w_{\mathfrak{t}}w_{\mathfrak{c}}.$  We shall prove $w_{\mathfrak{s}}\sim w_{\mathfrak{t}}$ in the sequel.
Let $\mathfrak{s}=(p,p), \mathfrak{t}=p, p\in P.$ Then $w_{\mathfrak{s}}=x_{p}x_{p}\overset{\rm(R1)}{\sim}x_{p}=w_{\mathfrak{t}}.$
Let $(e,p,f)$ be an admissible triple in $P$ and  $\mathfrak{s}=(e,e\star p,f),\mathfrak{t}=(e,p\times f,f)$. Then $$w_{\mathfrak{s}}=x_{e}x_{e\star {p}}x_{f}\overset{\rm(R4)}{\sim}x_{e}x_{p}x_{f}\overset{\rm(R2)}{\sim}x_{e}x_{p\times f}x_{f}=w_{\mathfrak{t}}.$$
This shows that $w_{\mathfrak{s}}\sim w_{\mathfrak{t}}$, and so $w_{\mathfrak{p}}\sim w_{\mathfrak{q}}.$
\end{proof}

Now we can give a proof of Theorem \ref{scab}.

{\bf A proof of Theorem \ref{scab}} By the freeness of $X_{P}^{+}$ we can assume that the unique semigroup homomorphism determined by the map
$X_P\rightarrow  \textbf{S}(\mathscr{C}(P)),\,\, x_p\mapsto [p]$ is
$$\Psi:X_{P}^{+}\rightarrow \textbf{S}(\mathscr{C}(P)).$$
We first show that $\Psi$ is surjective. Let $\mathfrak{p}=(p_{1},p_{2},\ldots,p_{k})\in \mathscr{P}(P),\mathfrak{c}=[\mathfrak{p}]\in \textbf{S}(\mathscr{C}(P)).$ By the proof of Proposition \ref{fgpu}, we have
\begin{equation}\label{vbkdg}
\begin{array}{cc}
\mathfrak{c}=[p_{1},p_{2},\ldots,p_{k}]=\varepsilon_{P}([p_{1}],[p_{2}],\ldots,[p_{k}])\\[3mm]
=[p_{1}]\bullet [p_{2}]\bullet\ldots\bullet [p_{k}]=(x_{p_{1}}\Psi)\bullet (x_{p_{2}}\Psi)\bullet\ldots \bullet (x_{p_{k}}\Psi)=(x_{p_{1}}x_{p_{2}}\ldots x_{p_{k}})\Psi=w_{\mathfrak{p}}\Psi.
\end{array}
\end{equation}
We next prove that $R_{P}^{\sharp}=\ker\Psi$. Let $(u,v)\in R_{P}$.
\begin{itemize}
\item[(1)]In the case that $u=x_{p}^{2},v=x_{p}$, we have $$u\Psi=(x_{p}x_{p})\Psi=(x_{p}\Psi)\bullet (x_{p}\Psi)=[p]\bullet [p]=[p]=x_{p}\Psi=v\Psi.$$
\item[(2)]Let $u=x_{p\times q}x_{q}$ and $v=x_{p}x_{q}$ for some $p,q \in P$. By (L4) and (L1), $$(p\times q)\times q=(p\times q)\times(p\times q)=p\times q,$$ which together with
Lemma \ref{jiben1} (3) gives that $(p\times q)\star q=p\star q$. By Lemma \ref{edd} (3),
    $$u\Psi=(x_{p\times q}x_{q})\Psi=(x_{p\times q}\Psi)\bullet(x_{q}\Psi)=[p\times q]\bullet[q]=\varepsilon_{P}([p\times q]\times[q],[p\times q]\star[q])$$$$=\varepsilon_{P}([(p\times q)\times q],[(p\times q)\star q])=\varepsilon_{P}([p\times q],[p\star q])=\varepsilon_{P}([p]\times[q],[p]\star[q])$$$$=[p]\bullet [q]=x_{p}\Psi\bullet x_{q}\Psi=(x_{p}x_{q})\Psi=v\Psi.$$
\item[(3)]Let $u=x_{p\times q}x_{p\star q}$ and $v=x_{p}x_{q}$ for some $p,q \in P$. By Lemma \ref{jiben1} (4), $$(p\times q)\times(p\star q)=p\times q,(p\times q)\star(p\star q)=p\star q.$$ Using similar calculations as in (2), we have
$$u\Psi=[p\times q]\bullet[p\star q]=\varepsilon_{P}([p\times q]\times[p\star q],[p\times q]\star[p\star q])=\varepsilon_{P}([p\times q],[p\star q])=v\Psi.$$
\item[(4)] This is the dual of (2).
\end{itemize}
From items (1)--(4) above, we can obtain $R_{P}^{\sharp}\subseteq \ker\Psi.$ On the other hand, let $(u,v)\in \ker\Psi$. Then $u,v\in X_{P}^{+}$, $u\Psi=v\Psi$. By Lemma \ref{abgq} we have $u\sim w_{\mathfrak{p}}$ and $v\sim w_{\mathfrak{q}}$ for some $\mathfrak{p},\mathfrak{q}\in \mathscr{P}(P)$. Moreover,   (\ref{vbkdg}) and the fact that $R_{P}^{\sharp}\subseteq \ker\Psi$ give that
$$[\mathfrak{p}]=w_{\mathfrak{p}}\Psi=u\Psi=v\Psi=w_{\mathfrak{q}}\Psi=[\mathfrak{q}].$$
Thus $\mathfrak{p}\approx_{P}\mathfrak{q}$. By Lemma \ref{slrgf}, $w_{\mathfrak{p}}\sim w_{\mathfrak{q}}$. So $u\sim w_{\mathfrak{p}}\sim w_{\mathfrak{q}}\sim v$. This implies that  $\ker(\Psi)\subseteq R_{P}^{\sharp}.$  Therefore,  $R_{P}^{\sharp}=\ker\Psi$.

Denote the $R_{P}^{\sharp}$-class containing $w\in  X_{P}^{+}$ by $\overline{w}$. Then
\begin{equation}\label{qyh}
\overline{\Psi}:  X_{P}^{+}/R_{P}^{\sharp} =\{\overline{w}\mid  w\in X_{P}^{+}\}=\{\overline{w_{\mathfrak{p}}}\mid  \mathfrak{p}\in \mathscr{P}(P)\}\rightarrow \textbf{S}(\mathscr{C}(P)), \,\, \overline{w}\mapsto w\Psi
\end{equation}
is a semigroup isomorphism. We observe that the  multiplication on
$X_{P}^{+}/R_{P}^{\sharp}$ is as follows: For all $w_{\mathfrak{p}}, w_{\mathfrak{q}}\in  X_{P}^{+}/R_{P}^{\sharp}$, $\overline{w_{\mathfrak{p}}}\,\,\overline{w_{\mathfrak{q}}}=\overline{w_{\mathfrak{p}} w_{\mathfrak{q}}}$. Moreover,
let $\mathfrak{p}=(p_1,p_2,\ldots, p_k)\in \mathscr{P}(P)$. Define
\begin{equation}\label{wangyin1}
\overline{w_{\mathfrak{p}}}^+=(\overline{w_{\mathfrak{p}}}\overline{\Psi})^+\overline{\Psi}^{-1}
=(w_{\mathfrak{p}}\Psi)^+\overline{\Psi}^{-1}=[p_1]\overline{\Psi}^{-1}=\overline{w_{{p_1}}}=\overline{x_{p_1}},\,\,\,\, \overline{w_{\mathfrak{p}}}^\ast=\overline{x_{p_k}}.
\end{equation}
Then $(X_{P}^{+}/R_{P}^{\sharp}, \cdot, ^+, ^\ast)$ forms a DRC-restriction semigroup with respect to the above operations with the set of projections $\{\overline{x_{p}}\mid p\in P\}$, and $\overline{\Psi}$ is a (2,1,1)-isomorphism from $X_{P}^{+}/R_{P}^{\sharp}$ onto $\textbf{S}(\mathscr{C}(P))$. Thus the strong projection algebra of $\langle X_{P}: R_{P}\rangle$ is isomorphic to that of  $\textbf{S}(\mathscr{C}(P))$, and so is isomorphic to  $(P,\times,\star).$

\begin{theorem}\label{ahyj}
Let $(S,\cdot,^+,^\ast)$ be a DRC-restriction semigroup and $(P(S),\times_{S},\star_{S})$ be the strong projection algebra of $S$. Assume that $(P,\times,\star)$ is a strong projection algebra and $\phi: P\rightarrow P(S)$  is a homomorphism  between strong projection algebras. Then there exists a unique (2,1,1)-homomorphism
$$\Phi:{\bf S}(\mathscr{C}(P))\rightarrow S,\,\,\,[p_1,p_2,\ldots, p_k]\mapsto (p_1\phi)(p_2\phi)\ldots (p_k\phi).$$
such that the following diagram
\begin{equation}\label{wangyin3}
\CD
  P@>\phi>> P(S)\\
  @V\iota_{1} VV @V\iota_{2} VV \\
  {\bf S}(\mathscr{C}(P))@>\Phi>> S
\endCD
\end{equation}
commutes, where $$\iota_{1}:P\rightarrow {\bf S}(\mathscr{C}(P)),\,\,p\mapsto [p],\,\,\iota_{2}:P(S)\rightarrow S,\,\,p\mapsto p.$$
\end{theorem}
\begin{proof} Bt Theorem \ref{scab},
it suffices to show that there exists a unique  (2,1,1)-homomorphism $\Gamma: X_{P}^{+}/R_{P}^{\sharp}\rightarrow S$ such that the diagram
$$\CD
  P@>\phi>> P(S)\\
  @V\tau_{1} VV @V\tau_{2} VV \\
X_{P}^{+}/R_{P}^{\sharp} @>\Gamma>> S
\endCD$$
commutes, where $$\tau_{1}:P\rightarrow X_{P}^{+}/R_{P}^{\sharp} ,\,\,p\mapsto \overline{x_{p}},\,\,\tau_{2}:P(S)\rightarrow S,\,\,x\mapsto x.$$
If this is the case, it is easy to prove that
\begin{equation}\label{wangyin2}
\Phi=\overline{\Psi}^{-1}\Gamma:{\bf S}(\mathscr{C}(P))\rightarrow S
\end{equation}
is the unique (2,1,1)-homomorphism satisfying  $\iota_{1}\Phi=\phi \iota_{2}$, where $\overline{\Psi}$ is define in (\ref{qyh}).

By the freeness of $X^+_{P}$, we can assume that the unique semigroup homomorphism determined by the map $X_P\rightarrow S, x_p\mapsto p\phi$ is as follows:
$$\varphi:X_{P}^{+}\rightarrow S,\,\,\,x_{p}\rightarrow p\phi.$$
We shall show that $R_{P}\subseteq \ker{\varphi}$ in the sequel. Let $p,q\in P$ and $(u,v)\in R_{P}$. If  $u=x_{p}^{2}$ and $v=x_{p}$, then $u\varphi=x_{p}^{2}\varphi=(p\phi)(p\phi)=p\phi=v\varphi.$
If $u=x_{p\times q}x_{q}$ and $v=x_{p}x_{q}$, then by using the fact that $\phi$ is homomorphism, Lemmas \ref{touyingdaishu} and \ref{kangwei} (3), we have
$$u\varphi=(x_{p\times q}x_{q})\varphi=((p\times q)\phi) (q\phi)=(p\phi\times_S q\phi)(q\phi)$$$$=((p\phi)(q\phi))^{+}(q\phi)=(p\phi)(q\phi)=(x_{p}\varphi)(x_{q}\varphi)=(x_{p}x_{q})\varphi=v\varphi.$$
Dually, we also have $u\varphi=v\varphi$ if $u=x_{p}x_{p\star q}$ and $v=x_{p}x_{q}$.
If $u=x_{p\times q}x_{p\star{q}}, v=x_{p}x_{q}$, then by using the fact that $\phi$ is homomorphism, Lemmas \ref{touyingdaishu} and \ref{kangwei} (4), we obtain
$$u\varphi=(x_{p\times q}x_{p\star{q}})\varphi=(p\times q)\phi(p\star q)\phi
=(p\phi \times_S q\phi) (p\phi\star_S q\phi)\phi$$$$=((p\phi)(q\phi))^{+}((p\phi)(q\phi))^{\star}=(p\phi)(q\phi)=(x_{p}x_{q})\varphi=v\varphi.$$
This implies that $R_{P}\subseteq \ker\varphi$, and so $R_{P}^{\sharp}\subseteq \ker\varphi$. Thus we can define the semigroup homomorphism
$$\Gamma: X_{P}^{+}/R_{P}^{\sharp}=\langle X_{P}:R_{P}\rangle\rightarrow S,\,\,\, \overline{w}\rightarrow w\varphi.$$
By the statements after Lemma \ref{abgq}, we have $X_{P}^{+}/R_{P}^{\sharp}=\{\overline{w_{\mathfrak{p}}}\mid  \mathfrak{p}\in \mathscr{P}(P)\}$. Let $$\overline{w_{\mathfrak{p}}} \in X_{P}^{+}/R_{P}^{\sharp},\,\, \mathfrak{p}=(p_{1},p_{2},\ldots,p_{k})\in \mathscr{P}(P).$$ By (\ref{wangyin1}), $\overline{w_{\mathfrak{p}}}^{+}=\overline{x_{p_1}}$ and $ \overline{w_{\mathfrak{p}}}^\ast=\overline{x_{p_k}}$.  This shows that $\overline{w_{\mathfrak{p}}}^{+}\Gamma=\overline{x_{p_1}}\Phi=x_{p_1}\varphi=p_1\phi$.
Observe $p_{1}\mathcal{F}_{P} p_{2}\mathcal{F}_{P} \ldots \mathcal{F}_{P} p_{k}$  as $\mathfrak{p}=(p_{1},p_{2},\ldots,p_{k})\in \mathscr{P}(P)$. Hence $$p_{1}\phi\mathcal{F}_{P(S)} p_{2}\phi\mathcal{F}_{P(S)} \ldots \mathcal{F}_{P(S)} p_{k}\phi.$$
In view of Lemma \ref{zjiy}, in $(P(S),\times_{S},\star_{S})$ we have
\begin{equation*}
\begin{aligned}
p_{1}\phi&=p_{k}\phi\delta_{p_{k-1}\phi}\delta_{p_{k-2}\phi}\ldots \delta_{{p_{1}}\phi}\\
&=(p_{1}\phi\times_{S}(\ldots(p_{k-2}\phi\times_{S}(p_{k-1}\phi\times_{S}p_{k}\phi))))\\
&=(p_{1}\phi(\ldots(p_{k-2}\phi(p_{k-1}\phi p_{k}\phi)^{+})^{+})^{+})^{+}\\
&=((p_{1}\phi)(p_{2}\phi)\ldots (p_{k}\phi))^{+}\,\,\,\,\,\,\,\,(\mbox{by (ii) in DRC Conditions})\\
&=((x_{p_{1}}x_{p_{2}}\cdots x_{p_{k}})\varphi)^{+}
=(w_{\mathfrak{p}}\varphi)^{+}=(\overline{w_{\mathfrak{p}}}\Gamma)^{+}.
\end{aligned}
\end{equation*}
Thus $(\overline{w_{\mathfrak{p}}}\Gamma)^{+}=\overline{w_{\mathfrak{p}}}^{+}\Gamma.$ Dually, we have $(\overline{w_{\mathfrak{p}}}\Gamma)^{\ast}=\overline{w_{\mathfrak{p}}}^{\ast}\Gamma.$
Let $p\in P$. Then $$p(\tau_{1}\Gamma)=\overline{x_{p}}\Gamma=x_{p}\varphi=p\phi=(p\phi)\tau_{2}=p(\phi\tau_{2}).$$
Finally, since $$\overline{w_{\mathfrak{p}}}=\overline{x_{p_{1}}x_{p_{2}}\ldots x_{p_{k}}}=\overline{x_{p_{1}}}\,\,\overline{x_{p_{2}}}\ldots\overline{x_{p_{k}}},$$
it follows that $X_{P}^{+}/R_{P}^{\sharp}$ is generated by $\{\overline{x_{p}}\mid  p\in P\}$. So $\Gamma$ is necessarily unique. Moreover, combining  (\ref{vbkdg}), (\ref{qyh})  and (\ref{wangyin2}), we have $$\Phi:{\bf S}(\mathscr{C}(P))\rightarrow S,\,\,\,[p_1,p_2,\ldots, p_k]\mapsto (p_1\phi)(p_2\phi)\ldots (p_k\phi),$$ which  is the unique (2,1,1)-homomorphism such that the diagram (\ref{wangyin3}) commutes.
\end{proof}

Denote the category of DRC-restriction semigroups together with (2,1,1)-homomorphisms by $\mathbb{DRS}$ and the category of strong projection algebras together with projection algebra homomorphisms by $\mathbb{SPA}$, respectively.
Define a functor ${\mathcal U}: \mathbb{DRS}\rightarrow \mathbb{SPA}$ as follows: For any two DRC-semigroups $(S, \cdot, ^+, ^\ast)$, $(T, \cdot, ^+, ^\ast)$
and a (2,1,1)-homomorphism $\phi: S\rightarrow T$,
\begin{itemize}
\item[(1)]${\mathcal U}(S)=(P(S), \times_S, \star_S)$.
\item[(2)]${\mathcal U}(\phi): {\mathcal U}(S)=P(S)\rightarrow {\mathcal U}(T)=P(T), x\mapsto x\phi$.
\end{itemize}
Moreover, define a functor ${\mathcal V}: \mathbb{SPA} \rightarrow \mathbb{DRS}$ as follows: For any two strong  projection algebras $(P, \times, \star)$ and $(Q, \times, \star)$ and a projection algebra homomorphism $\psi: P\rightarrow Q$,
\begin{itemize}
\item[(1)]${\mathcal V}(P)= \textbf{S}(\mathscr{C}(P)).$
\item[(2)]${\mathcal V}(\psi): {\mathcal V}(P)=\textbf{S}(\mathscr{C}(P))\rightarrow {\mathcal V}(Q)=\textbf{S}(\mathscr{C}(Q))$ is the  (2,1,1) homomorphism induced in Theorem \ref{cc}.
\end{itemize}
To give further properties of the functors ${\mathcal U}$ and ${\mathcal V}$ above, we need some notions.
\begin{defn}\label{yqzyq}
Let ${\rm \bf C}$ and $ {\rm \bf D}$ be two (large) category and $F$ and $G$ be two functors from ${\rm \bf C}$ to ${\rm \bf D}$. Assume that for each object $C$ in ${\rm \bf C}$, $\eta_{C}$ is a morphism from $F(C)$ to $G(C)$ in ${\rm \bf D}$. The family of morphisms $\eta=\{\eta_{C}\mid {C\mbox{ is }\mbox{ an object in }{\rm \bf C}}\}$ is called a {\em natural transformation } from $F$ to $G$ if for any two objects $C,C'$ in ${\rm \bf C}$ and
any morphism $\phi: C\rightarrow C'$ in ${\rm \bf C}$, we have $\eta_CG(\phi)=F(\phi)\eta_{C'}$, that is, the following diagram
$$\CD
  F(C)@>\eta_{C}>> G(C)\\
  @VF(\phi) VV @VG(\phi) VV  \\
 F(C')@>\eta_{ C'}>>G(C')
\endCD$$
commutes.
\end{defn}
\begin{defn}\label{axdr}
Let ${\rm \bf C}$ and $ {\rm \bf D}$ be two (large) category, ${\mathcal V}: {\rm \bf C} \rightarrow {\rm \bf D}$ and ${\mathcal U}: {\rm \bf D} \rightarrow {\rm \bf C}$ be two functors, and $\eta: {\rm id}_{{\rm \bf C}}\rightarrow {\mathcal UV}$ be a natural transformation. Then $({\mathcal V}, {\mathcal U},\eta)$ is called an {\em adjunction} if for any object $C$ in ${\rm \bf C}$ , any object $D$ in ${\rm \bf D}$ and any morphism $\phi: C\rightarrow  {\mathcal U}(D)$, there exists a unique morphism $\overline{\phi}: {\mathcal V}(C)\rightarrow D$ such that $\eta_C{\mathcal U}(\overline{\phi})=\phi$, that is, the following diagram
$$
\xymatrix{
  C \ar[r]^{\phi} \ar[d]_{\eta_C} & {\mathcal U}(D)\\
  {\mathcal U}({\mathcal V}(C)) \ar[ur]_{{\mathcal U}(\overline{\phi})}
}
$$
commutes. In this case, $\mathcal V$ and $\mathcal U$ are called {\em left and right adjoints}, respectively, and $\eta$ is called the {\em unit} of the adjunction $({\mathcal V}, {\mathcal U},\eta)$. Moreover, elements in $\{{\mathcal V}(C)\mid C\mbox{ is an object in {\rm \bf C}}\}$ are called the {\em {\bf C}-free objects in {\rm \bf D}}.
\end{defn}

\begin{theorem}\label{dlsrb}
Define $$\eta=\{\eta_{P}\mid P \mbox{ is a strong projection algebra}\}: {\rm id}_{\mathbb {SPA}}\rightarrow {\mathcal U}{\mathcal V}$$ as follows:
For any strong projection algebra $P$,
$$\eta_{P}:P={\rm id}_{\mathbb {SPA}}(P)\rightarrow {\mathcal U}{\mathcal V}(P)=\{[p]\mid p\in P\},\,p\mapsto [p].$$
Then $({\mathcal V}, {\mathcal U},\eta)$ is an adjunction, and so the $\mathbb{SPA}$-objects in $\mathbb{DRS}$ are exactly chain semigroups.

\end{theorem}
\begin{proof}
We first show that $\eta$ is a natural transformation. Let $P$ and $Q$ be two strong projection algebras and $\phi:P\rightarrow Q$ be a homomorphism. Let $p\in P$. Then $$p(\eta_{P}({\mathcal U}{\mathcal V}(\phi)))=[p]({\mathcal U}{\mathcal V}(\phi))=[p\phi]=(p\phi)\eta_{Q}=p(\phi\eta_{Q})=p({\rm id}_{\mathbb {SPA}}(\phi)\eta_{Q}).$$  This shows that the following diagram
$$\CD
   P={\rm id}_{\mathbb{SPA}}(P)@>\eta_{P}>> {\mathcal U}{\mathcal V}(P)=P(\textbf{S}(\mathscr{C}(P)))\\
  @V{\rm id}_{\mathbb{SPA}}(\phi) VV @V {\mathcal U}{\mathcal V}(\phi)VV  \\
 Q={\rm id}_{\mathbb{SPA}}(Q)@>\eta_{Q}>> {\mathcal U}{\mathcal V}(Q)=P(\textbf{S}(\mathscr{C}(Q)))
\endCD$$
commutes. Thus $\eta$ is anatural transformation.
Let $(P,\times,\star)$ be a strong projection algebra, $(S,\cdot, ^+, ^\ast)$ be a DRC-restriction semigroup whose strong projection algebra is $(P(S),\times_{S},\star_{S})$  and $\phi:P\rightarrow {\mathcal U}(S)=P(S)$ be a projection algebra homomorphism. By Theorem \ref{ahyj}, there exists a unique  (2,1,1)-homomorphism
$$\Phi:{\mathcal V}(P)=\textbf{S}(\mathscr{C}(P))\rightarrow S,\,\,[p_{1},p_{2},\ldots,p_{k}]\mapsto (p_{1}\phi)(p_{2}\phi)\ldots(p_{k}\phi).$$
such that $[p]\Phi=p\phi$. Thus
$${\mathcal U}(\Phi):P(\textbf{S}(\mathscr{C}(P)))={\mathcal U}{\mathcal V}(P)\rightarrow {\mathcal U}(S)=P(S),\,\,[p]\mapsto p\phi.$$ Let $p\in P$. Then $p(\eta_{P}{\mathcal U}(\Phi))=(p\eta_{P}){\mathcal U}(\Phi)=[p]{\mathcal U}(\Phi)=p\phi,$ that is, the following diagram
$$
\xymatrix{
  P \ar[r]^{\phi} \ar[d]_{\eta_P} & {\mathcal U}(S)=P(S)\\
 P(\textbf{S}(\mathscr{C}(P))) ={\mathcal U}{\mathcal V}(P) \ar[ur]_{{\mathcal U}(\Phi)}
}
$$
commutes.
Assume that $\sigma:{\mathcal V}(P)=\textbf{S}(\mathscr{C}(P))\rightarrow S$ is a (2,1,1)-homomorphism such that the following diagram
$$
\xymatrix{
  P \ar[r]^{\phi} \ar[d]_{\eta_P} & {\mathcal U}(S)=P(S)\\
 P(\textbf{S}(\mathscr{C}(P))) ={\mathcal U}{\mathcal V}(P) \ar[ur]_{{\mathcal U}(\sigma)}
}
$$
commutes. Then $[p]\sigma=[p]{\mathcal U}(\sigma)=p(\eta_{P}{\mathcal U}(\sigma))=p\phi$ for any $p\in P$, and so the following diagram
$$\CD
  P@>\phi>> P(S)\\
  @V\iota_1 VV @V\iota_2 VV  \\
 \textbf{S}(\mathscr{C}(P))@>\sigma >>S
\endCD$$
commutes.
By Theorem \ref{ahyj}, we have $\sigma=\Phi$. This gives the uniqueness of $\Phi$. Thus $({\mathcal V}, {\mathcal U},\eta)$ is an adjunction.
\end{proof}

\section{Projection-fundamental DRC-restriction semigroups}
In this section, by using Theorem \ref{hgdc} we reobtain the constructions of projection-fundamental DRC-restriction semigroups which originally given by Wang \cite{Wang6}.  We first give some preliminary notions and results.

Let $(S, \cdot, ^+, ^\ast)$ be a DRC-restriction semigroup and $\sigma$ be an equivalence on $S$. Then $\sigma$ is called {\em a (2,1,1)-congruence} on $S$ if for all $a,b,c\in S$,
$$a~\sigma~b\Longrightarrow ac~\sigma~bc,\,c a~\sigma~c b,\,a^{+}\sigma b^{+},\, a^{\ast}\sigma b^{\ast}.$$
A (2,1,1)-congruence $\sigma$ on $S$ is called {\em projection-separating} if for all $p,q\in P(S),$  the fact that $p~\sigma~q$ implies that $p=q$.

\begin{lemma}[\cite{Jones1}]\label{dvhi}
Let $(S, \cdot, ^+, ^\ast)$ be a DRC-restriction semigroup. Then
\begin{equation}\label{skfvm}
\mu_{S}=\{(a,b)\in S\times S\mid a^+=b^+, a^\ast=b^\ast\mbox{ and }  $$$$  (ap)^{+}=(bp)^{+},(pa)^{\ast}=(pb)^{\ast} \mbox{ for all } p\in P(S)\}
\end{equation}
is the maximum projection-separating (2,1,1)-congruence on $S$.
\end{lemma}
\begin{defn}\label{shcyk}
A  DRC-restriction semigroup $(S, \cdot, ^+, ^\ast)$ is called {\em projection-fundamental} if the unique (2,1,1)-projection-separating congruence on $S$ is the identity relation, that is,  $\mu_S$ is the identity relation on $S$.
\end{defn}

\begin{prop}\label{fscpl}
Let $(S, \cdot, ^+, ^\ast)$ be a DRC-restriction semigroup and $a,b\in S$. Assume that $\dd_{S}$,\,$\rr_{S}$, $\nu_{a},\,\mu_{a},\nu_{b},\,\mu_{b}$ and $\Theta_{a}, \Delta_{a}, \Theta_{b}, \Delta_{b}$ are given in (\ref{qqyyaa}) ,(\ref{qqyyee}) ,(\ref{qqyyvv}), (\ref{yyqqaa}) and (\ref{yyqqee}), respectively. Then the following statements are equivalent:
\begin{itemize}
\item[(1)]$(a,b)\in \mu_{S}.$

\item[(2)]$\Theta_{a}=\Theta_{b}$ and $\Delta_{a}=\Delta_{b}.$

\item[(3)]$\dd_{S}(a)=\dd_{S}(b)$ and $\Theta_{a}=\Theta_{b}.$

\item[(4)]$\rr_{S}(a)=\rr_{S}(b)$ and $\Delta_{a}=\Delta_{b}.$

\item[(5)]$\dd_{S}(a)=\dd_{S}(b)$ and $\Theta_{a}{\mid}{_{\dd_{S}(a)^{\downarrow}}}=\Theta_{b}{\mid}{_{\dd_{S}(b)^{\downarrow}}}.$

\item[(6)]$\rr_{S}(a)=\rr_{S}(b)$ and $\Delta_{a}{\mid}{_{\rr_{S}(a)^{\downarrow}}}=\Delta_{b}{\mid}{_{\rr_{S}(b)^{\downarrow}}}.$

\item[(7)]$\nu_{a}=\nu_{b}.$

\item[(8)]$\mu_{a}=\mu_{b}.$
\end{itemize}
\end{prop}
\begin{proof}We first observe that $\dd_{S}(a)=a^{+},\,\rr_{S}(a)=a^{\ast}$, and
\begin{center}$p\Theta_{a}=p\theta_{\dd_{S}(a)}\nu_{a}=(pa)^{\ast}$ and $p\Delta_{a}=p\delta_{\rr_{S}(a)}\mu_{a}=(ap)^{+}$ for all $p\in P(S)$.
\end{center}
By symmetry, we only need to show that (1), (2), (3), (5) and (7) are mutually equivalent.   Lemma \ref{dvhi} gives that (1) implies that (2).  That (3) implies (5) is trivial. The fact (\ref{xpha}) provides that (5) implies (7). In the followings, we prove that (2) implies (3) and (7) implies (1).  Assume that (2) holds. Then  $\Theta_{a}=\Theta_{b}$ and $\Delta_{a}=\Delta_{b}.$  This implies that $$\dd_{S}(a)=a^{+}=(aa^{\ast})^{+}=a^{\ast}\Delta_{a}=a^{\ast}\Delta_{b}=(ba^{\ast})^{+}\leq_{S}b^{+}=\dd_{S}(b)$$ by (i)$'$, (\ref{xiaoyu}) and Lemma \ref{xiaoyuslw}.
Dually, we have  $\dd_{S}(b)\leq_{S} \dd_{S}(a)$. This implies that  $\dd_{S}(b)=\dd_{S}(a)$. This together with the known fact that $\Theta_{a}=\Theta_{b}$ gives (3).

Now assume that (7) is true. Then $\nu_{a}=\nu_{b}.$  This yields that $\dd_{S}(a)^{\downarrow}=\dm\nu_{a}=\dm\nu_{b}=\dd_{S}(b)^{\downarrow}$, and so  $\dd_{S}(a)=\dd_{S}(b)$. Moreover,  $\theta_{\dd_{S}(a)}=\theta_{\dd_{S}(b)}$. Thus $\Theta_{a}=\theta_{\dd_{S}(a)}\nu_{a}=\theta_{\dd_{S}(b)}\nu_{b}=\Theta_{b}.$ This gives  (3). On the other hand, by Lemma \ref{f}, we have $\mu_{a}=\nu_{a}^{-1}=\nu_{b}^{-1}=\mu_{b}$. By the dual statements, we can obtain (4). Now Lemma \ref{dvhi} and (3), (4) together imply (1).
\end{proof}

Let $(P,\times,\star)$ be a strong projection algebra and $p\in P$. Then $$p^{\downarrow}=\{x\in P\mid x\leq_P p\}$$ is a subalgebra of $P$. Let $p,q\in P$. Denote
$$M(p,q)=\{\alpha\mid  \alpha\mbox{ is } \mbox{ a projection algebra isomorphism from } p^{\downarrow}\mbox{ to }q^{\downarrow}\},$$
$$M(P)=\underset{{p,q\in P}}{\bigcup}M(p,q).$$
For all $\alpha,\beta \in M(P)$, define
$$\alpha\circ \beta=\left\{\begin{array}{cc}
\alpha\beta\,\,\,\,\,\,\mbox{if }\ra \alpha=\dm\beta,\\
\mbox{undefined }\,\,\,\,\,\,\mbox{otherwise},
\end{array} \right.
$$
and $\dd(\alpha)={\rm id}_{\dm\alpha},\,\,\rr(\alpha)={\rm id}_{\ra\alpha},$
where $\alpha\beta$ denotes the composition of $\alpha$ and $\beta$. Denote the inverse map of $\alpha$ by $\alpha^{-1}$ for all $\alpha \in M(P)$. Then it is easy to prove that  $(M(P),\circ,\dd,\rr,^{-1})$ is a groupoid with the set of objects $P_{M(P)}=\{{\rm id}_{p^{\downarrow}}\mid p\in P\}.$ We also observe the following fact:
\begin{equation}\label{baoxu}
p\alpha =q \mbox{ and } s\alpha\leq_{P}t\alpha \mbox{ for all }  \alpha\in M(p,q)  \mbox{ and } s,t\in p^{\downarrow} \mbox{ with } s\leq_{P}t,
\end{equation}
Let $\alpha\in M(s,t)$ and $X$ be a subalgebra of $\dm\alpha=s^\downarrow$. Then $\alpha{\mid}_{X}:X\rightarrow X\alpha,\,\,\,x\mapsto x\alpha$ is a projection algebra isomorphism from $X$ to $X\alpha$, and is called {\em the restriction of $\alpha$ on $X$}.
\begin{lemma}\label{yinwang}
Let $\alpha\in M(s,t)$ and $p\in \dm\alpha$. Then $\alpha{\mid}_{p^{\downarrow}}\in M(p,p\alpha).$
\end{lemma}
\begin{proof}
If $p\in \dm\alpha$, then $p^{\downarrow}\subseteq \dm\alpha=s^{\downarrow}$, and so
$\alpha{\mid}_{p^{\downarrow}}:p^{\downarrow}\rightarrow p^{\downarrow}\alpha,\,\,\,x\mapsto x\alpha.$
Let $x\alpha \in p^{\downarrow}\alpha$, where $x\leq_{P}p$. By (\ref{baoxu}), we have $x\alpha\leq_{P}p\alpha$ and so $x\alpha \in (p\alpha)^{\downarrow}$. This shows that $ p^{\downarrow}\alpha \subseteq (p\alpha)^{\downarrow}.$ Now let $y\in (p\alpha)^{\downarrow}$. Then $y\leq_{P} p\alpha\leq_{P} t$, and so $y\in t^{\downarrow}$. Denote  $y=x\alpha$. Then $y=x\alpha\leq_{P}p\alpha.$ Thus $x\leq_{P}p$ as $\alpha^{-1}$ is  also order-preserving. This implies that $(p\alpha)^{\downarrow}\subseteq  p^{\downarrow}\alpha$. Therefore $(p\alpha)^{\downarrow}= p^{\downarrow}\alpha.$ We have shown that $\alpha{\mid}_{p^{\downarrow}}\in M(p,p\alpha).$
\end{proof}
Define a relation $\leq_{M(P)}$ on $M(P)$ as follows: For all $\alpha,\beta \in M(P)$,
$$\alpha \leq_{M(P)}\beta \Longleftrightarrow \alpha=\beta{\mid}_{\dm\alpha}.$$
That is, $\alpha \leq_{M(P)}\beta$ if and only if $\dm\alpha\subseteq \dm\beta$, and $x\alpha=x\beta$ for all $x\in \dm\alpha$. It is easy to see that $``\leq_{M(P)}"$ is a partial order on  $M(P)$.

\begin{lemma}\label{yind2}
Let $(P,\times,\star)$ be a strong projection algebra. Then $(M(P),\circ,\dd,\rr, \leq_{M(P)}, P_{M(P)}, \times, \star)$ is a  projection ordered category and
$${_{{\rm id}_{p^{\downarrow}}}}{\downharpoonleft}\alpha=\alpha{\mid}_{p^{\downarrow}},\,\,\,\,\alpha{\downharpoonright}{_{{\rm id}_{q^{\downarrow}}}}=\alpha{\mid}{_{(q\alpha^{-1})^{\downarrow}}} $$ for all $ p\in P $ with  ${\rm id}_{p^{\downarrow}}\leq_{M(P)} \dd(\alpha)$ and $q\in P$ with  ${\rm id}_{q^{\downarrow}}\leq_{M(P)} \rr(\alpha)$, respectively.
\end{lemma}
\begin{proof}Define $\times$ and $\star$ on $P_{M(P)}$ as follows:
For all $p,q\in P$,
$${\rm id}_{p^{\downarrow}}\times{\rm id}_{q^{\downarrow}}={\rm id}_{(p\times q)^{\downarrow}},\,\,{\rm id}_{p^{\downarrow}}\star {\rm id}_{q^{\downarrow}}={\rm id}_{(p\star q)^{\downarrow}}.$$
Then $(P_{M(P)},\times,\star)$ forms a strong projection algebra isomorphic to $(P,\times,\star)$ via the map ${\rm id}_{p^{\downarrow}}\mapsto p$. Thus for all $p,q\in P$, $p\leq_{P}q$ if and only if ${\rm id}_{p^{\downarrow}}\leq_{P_{M(P)}}{\rm id}_{q^{\downarrow}}.$

Now let $p,q\in P$.  Then $${\rm id}_{p^{\downarrow}}\leq_{M(P)} {\rm id}_{q^{\downarrow}}\Longleftrightarrow p^{\downarrow}\subseteq q^{\downarrow}\Longleftrightarrow p\leq_{P}q \Longleftrightarrow {\rm id}_{p^{\downarrow}}\leq_{P_{M(P)}} {\rm id}_{q^{\downarrow}}.$$  This implies that $\leq_{M(P)}$ and $\leq_{P_{M(P)}}$ is the same on $P_{M(P)}$.

(O1) Let $\alpha \leq_{M(P)}\beta$. Then $\alpha=\beta{\mid}_{\dm\alpha}$, which implies that  $\dm\alpha\subseteq \dm\beta$ and $x\alpha=x\beta$ for all $x\in \dm\alpha$. Hence $\dd(\alpha)={\rm id}_{\dm\alpha}\leq_{M(P)} {\rm id}_{\dm\beta}=\dd(\beta)$. Dually, we have $\rr(\alpha )\leq_{M(P)}\rr(\beta)$.

(O2) Let $\alpha \leq_{M(P)}\beta$,$\gamma \leq_{M(P)}\delta$ and $\alpha\circ \gamma$ and $\beta\circ \delta$ be defined. Then $$\alpha=\beta{\mid}_{\dm\alpha},\, \gamma=\delta{\mid}_{\dm\gamma}, \,\dm\alpha\subseteq \dm\beta, \, \dm\gamma\subseteq \dm\delta,$$$$\ra\alpha=\dm\gamma,\, \ra\beta=\dm\delta,\,\, \alpha\circ \gamma=\alpha\gamma,\,\beta\circ \delta=\beta\delta.$$ Thus $$\dm(\alpha \circ\gamma)=\dm(\alpha\gamma)=\dm(\alpha)\subseteq \dm(\beta)=\dm(\beta\circ\delta).$$
Let $x\in \dm(\alpha\circ \gamma)=\dm\alpha$. Then $$x(\alpha \circ \gamma)=x(\alpha\gamma)=(x\alpha)\gamma=(x\beta)\gamma=(x\beta)\delta=x\beta\delta=x(\beta\circ \delta).$$ This implies  that $\alpha\circ \gamma\leq_{M(P)}\beta \circ \delta.$

(O3) Let ${\rm id}_{p^{\downarrow}}\in P_{M(P)}, p\in P $ and ${\rm id}_{p^{\downarrow}}\leq_{M(P)} \dd(\alpha)={\rm id}_{\dm\alpha}$.  Then $p^{\downarrow}\subseteq \dm\alpha.$ By Lemma \ref{yinwang}, we have $\alpha{\mid}_{p^{\downarrow}}\in M(p,p\alpha)\subseteq M(P)$. Obviously, $\alpha{\mid}_{p^{\downarrow}}\leq_{M(P)}\alpha$.  Moreover, $\dd(\alpha{\mid}_{p^{\downarrow}})={\rm id}_{\dm(\alpha{\mid}_{p^{\downarrow}})}={\rm id}_{p^{\downarrow}}$.
Assume that $\beta\in M(P),\beta\leq_{M(P)} \alpha$ and $\dd(\beta)={\rm id}_{p^{\downarrow}}$. Then
${\rm id}_{\dm\beta}=\dd(\beta)={\rm id}_{p^{\downarrow}}$ and $\beta=\alpha{\mid}_{\dm\beta}$, and so $\dm\beta=p^{\downarrow}$ and $\beta=\alpha{\mid}_{\dm\beta}=\alpha{\mid}_{p^{\downarrow}}.$ Thus ${_{{\rm id}_{p^{\downarrow}}}}{\downharpoonleft}\alpha=\alpha{\mid}_{p^{\downarrow}}$.

(O4) Let ${\rm id}_{q^{\downarrow}}\in P_{M(P)}, q\in P$ and ${\rm id}_{q^{\downarrow}}\leq_{M(P)} \rr(\alpha)={\rm id}_{\ra\alpha}$. Then $$q^{\downarrow}\subseteq \ra\alpha=\dm\alpha^{-1},q\alpha^{-1}\in \ra\alpha^{-1}=\dm\alpha, (q\alpha^{-1})^{\downarrow}\subseteq \dm\alpha.$$
By Lemma \ref{yinwang}, we have $\alpha{\mid}{_{(q\alpha^{-1})^{\downarrow}}}\in M(q\alpha^{-1},q)$. Obviously, $\alpha{\mid}{_{(q\alpha^{-1})^{\downarrow}}} \leq_{M(P)} \alpha$. Moreover, $\rr(\alpha{\mid}{_{(q\alpha^{-1})^{\downarrow}}} )={\rm id}_{\ra(\alpha{\mid}{_{(q\alpha^{-1})^{\downarrow}}} )}={\rm id}_{q^{\downarrow}}$.
Assume that $\beta\in M(P), \beta\leq_{M(P)} \alpha$ and $\rr(\beta)={\rm id}_{q^{\downarrow}}$. Then  ${\rm id}_{\ra\beta}=\rr(\beta)={\rm id}_{q^{\downarrow}}$
and $\beta=\alpha{\mid}_{\dm\beta}$, and so $\ra\beta=q^{\downarrow}$ and $$\beta=\alpha{\mid}_{\dm\beta} =\alpha{\mid}{_{(q^{\downarrow})\beta^{-1}}}=\alpha{\mid}{_{(q\beta^{-1})^{\downarrow}}}
=\alpha{\mid}{_{(q\alpha^{-1})^{\downarrow}}}.$$ Thus $\alpha{\downharpoonright}{_{{\rm id}_{q^{\downarrow}}}}=\alpha{\mid}{_{(q\alpha^{-1})^{\downarrow}}}$.
We have shown that $(M(P),\circ,\dd,\rr, \leq_{M(P)}, P_{M(P)}, \times,\star)$ is a weak projection ordered category.

Now, let $p,q\in P$ and $\alpha\in M(p,q)$.  Then
$$\dd(\alpha)^{\downarrow}=({\rm id}_{\dm\alpha})^{\downarrow}=({\rm id}_{p^{\downarrow}})^{\downarrow}=\{{\rm id}_{x^{\downarrow}}\mid x\in P,  x\leq_{P}p\}=\{{\rm id}_{x^{\downarrow}}\mid x\in P, x\in p^{\downarrow}\},$$$$\,\,\rr(\alpha)^{\downarrow}=({\rm id}_{\ra\alpha})^{\downarrow}=({\rm id}_{q^{\downarrow}})^{\downarrow} =\{{\rm id}_{y^{\downarrow}}\mid y\in P, y\leq_{P}q\}$$$$=\{{\rm id}_{y^{\downarrow}}\mid y\in P, y\in q^{\downarrow}\}=\{{\rm id}_{(x\alpha)^{\downarrow}}\mid x\in P, x\in p^{\downarrow}\}.$$
Let ${\rm id}_{x^{\downarrow}}\leq_{M(P)} \dd(\alpha)$. By (\ref{c}), $${\rm id}_{x^{\downarrow}}\nu_{\alpha}=\rr({_{{\rm id}_{x^{\downarrow}}}}{\downharpoonleft}\alpha)=\rr(\alpha{\mid}{_{x^{\downarrow}}})
={\rm id}_{\ra(\alpha{\mid}{_{x^{\downarrow}}})}={\rm id}_{(x\alpha)^{\downarrow}}.$$
Thus we have
\begin{equation}\label{frig}
\nu_{\alpha}:\dd(\alpha)^{\downarrow}\rightarrow \rr(\alpha)^{\downarrow},\,\,{\rm id}_{x^{\downarrow}}\mapsto {\rm id}_{(x\alpha)^{\downarrow}}.
\end{equation}
Lemma \ref{f} gives that $\mu_{\alpha}=\nu_{\alpha^{-1}}$. Hence
\begin{equation}\label{yind4}
\mu_{\alpha}:\rr(\alpha)^{\downarrow}\rightarrow \dd(\alpha)^{\downarrow},\,\,{\rm id}_{y^{\downarrow}}\mapsto {\rm id}_{(y\alpha^{-1})^{\downarrow}}.
\end{equation}
Since $\alpha$ is isomorphism,  $\nu_{\alpha}$ is an isomorphism. Thus $(M(P),\circ,\dd,\rr,
\leq_{M(P)}, P_{M(P)}, \times, \star)$ is a  projection ordered category. 
\end{proof}

Let $(P,\times,\star)$ be a strong projection algebra, $p,q\in P$ and $\alpha\in M(p,q)$. Then $\dd(\alpha)={\rm id}_{p^{\downarrow}}$ and $\rr(\alpha)={\rm id}_{q^{\downarrow}}$. Let ${\rm id}_{x^{\downarrow}}\in P_{M(P)}$. By (\ref{aa}) and (\ref{frig}), $$({\rm id}_{x^{\downarrow}})\Theta_{\alpha}=({\rm id}_{x^{\downarrow}})\theta_{\dd(\alpha)}\nu_{\alpha}=({\rm id}_{x^{\downarrow}}\star {\rm id}_{p^{\downarrow}})\nu_{\alpha}=({\rm id}_{(x\star p)^{\downarrow}})\nu_{\alpha}={\rm id}_{((x\star p)\alpha)^{\downarrow}}.$$
Thus we have
\begin{equation}\label{ahdy}
\Theta_{\alpha}:P_{M(P)}\rightarrow\rr(\alpha)^{\downarrow},\,\,\,{\rm id}_{x^{\downarrow}}\mapsto {\rm id}_{((x\star p)\alpha)^{\downarrow}}.
\end{equation}
Dually, we obtain
\begin{equation}\label{yind5}
\Delta_{\alpha}:P_{M(P)}\rightarrow\dd(\alpha)^{\downarrow},\,\,\,{\rm id}_{x^{\downarrow}}\mapsto {\rm id}_{((q\times x)\alpha^{-1})^{\downarrow}}.
\end{equation}
In particular, for all ${\rm id}_{x^{\downarrow}}\in P_{M(P)}$,
\begin{equation}\label{yind6}
\Theta_{{\rm id}_{x^{\downarrow}}}=\theta_{{\rm id}_{x^{\downarrow}}},\,\,\Delta_{{\rm id}_{x^{\downarrow}}}=\delta_{{\rm id}_{x^{\downarrow}}}
\end{equation}

Let $(P,\times,\star)$ be a strong projection algebra, and $p,q \in P$ with $p \mathcal{F}_{P} \,q$. Then
$$\theta_{q}:P\rightarrow P,\,\,x\mapsto x\star q,\,\,\delta_{p}:P\rightarrow P,\,\,y\mapsto p\times y.$$
By (\ref{m}), we have $\ra\theta_{q}=q^{\downarrow}$ and $\ra\delta_{p}=p^{\downarrow}$. Denote
\begin{equation}\label{avtki}
\gamma_{p,q}=\theta_{q}{\mid}{_{p^{\downarrow}}},\,\,\,\beta_{q,p}=\delta_{p}{\mid}{_{q^{\downarrow}}}
\end{equation}
Then $\gamma_{p,q}$  and $\beta_{q,p}$ are maps from $p^{\downarrow}$ to $q^{\downarrow}$ and $q^{\downarrow}$ to $p^{\downarrow}$, respectively.
According to \cite[(4.8)]{Wang6}, \cite[Lemma 4.5]{Wang6} and the fact  that $p \mathcal{F}_{P} \,q$, we have the following lemma.

\begin{lemma}\label{qodvk}
Let $(P,\times,\star)$ be a strong projection algebra and $p,q \in P$ with $p \mathcal{F}_{P}\, q$. Then $\gamma_{p, q}\in M(p,q)$ and $\gamma_{p,q}^{-1}=\beta_{q,p}.$
\end{lemma}

Let $f: A\rightarrow B$ and $g: B\rightarrow C$ be two bijections and $X\subseteq A$. Then
\begin{equation}\label{tiocn}
(fg){\mid}_{X}=(f{\mid}_{X})( g{\mid}{_{Xf}}).
\end{equation}

\begin{lemma}\label{aufrp}
Let $(P,\times,\star)$ be a strong projection algebra and $(M(P),\circ,\dd,\rr, \leq_{M(P)}, P_{M(P)},\\ \times,\star)$ be the projection ordered category determined by $P$ and ${\mathscr P}(P_{M(P)})$ be the path category of $P_{M(P)}$. Define  $$\varepsilon: {\mathscr P}(P_{M(P)})\rightarrow M(P),\,\,\, {\rm id}_{p_1^{\downarrow}}\mapsto {\rm id}_{p_1^{\downarrow}},\,\,\,
\,\, ({\rm id}_{p_{1}^{\downarrow}},{\rm id}_{p_{2}^{\downarrow}},\ldots,{\rm id}_{p_{k}^{\downarrow}})\rightarrow \gamma_{p_{1},p_{2}}\gamma_{p_{2},p_{3}}\ldots\gamma_{p_{k-1},p_{k}}.$$
Then $\varepsilon$ is an evaluation map.
\end{lemma}
\begin{proof}
Let $$\mathfrak{p}=({\rm id}_{p_{1}^{\downarrow}},{\rm id}_{p_{2}^{\downarrow}},\ldots,{\rm id}_{p_{k}^{\downarrow}}),\mathfrak{s}=({\rm id}_{s_{1}^{\downarrow}},{\rm id}_{s_{2}^{\downarrow}},\ldots,{\rm id}_{s_{l}^{\downarrow}})\in \mathscr{P}(P_{M(P)})$$ and $p\in P$. Then $\varepsilon({\rm id}_{p^{\downarrow}})={\rm id}_{p^{\downarrow}}$. By (\ref{avtki}), we have $\varepsilon(({\rm id}_{p^{\downarrow}},{\rm id}_{p^{\downarrow}}))=\gamma_{p, p}=\theta_{p}{\mid}{_{p^{\downarrow}}}={\rm id}_{p^{\downarrow}}.$ This gives (E1). By (\ref{avtki}) and (E1), $$\dd(\varepsilon(\mathfrak{p}))=\dd(\gamma_{p_{1},p_{2}}\gamma_{p_{2},p_{3}}\ldots\gamma_{p_{k-1},p_{k}})
$$$$=\dd(\theta_{p_{2}}{\mid}{_{p_{1}^{\downarrow}}}\theta_{p_{3}}{\mid}{_{p_{2}^{\downarrow}}}\ldots
\theta_{p_{k}}{\mid}{_{p_{k-1}^{\downarrow}}})={\rm id}_{p_{1}^{\downarrow}}=\varepsilon({\rm id}_{p_{1}^{\downarrow}})=\varepsilon(\dd(\mathfrak{p})).$$
Dually, we have $\rr(\varepsilon(\mathfrak{p}))=\varepsilon(\rr(\mathfrak{p}))$. This proves (E2). If $\rr(\mathfrak{p})=\dd(\mathfrak{s})$, then $$\varepsilon(\mathfrak{p}\circ \mathfrak{s})=\varepsilon({\rm id}_{p_{1}^{\downarrow}},{\rm id}_{p_{2}^{\downarrow}},\ldots,{\rm id}_{p_{k}^{\downarrow}}={\rm id}_{s_{1}^{\downarrow}},{\rm id}_{s_{2}^{\downarrow}},\ldots,{\rm id}_{s_{l}^{\downarrow}})
$$$$=(\gamma_{p_{1},p_{2}}\gamma_{p_{2},p_{3}}\ldots\gamma_{p_{k-1},p_{k}})(\gamma_{s_{1},s_{2}}\gamma_{s_{2},s_{3}}\ldots\gamma_{s_{l-1},s_{l}})
=\varepsilon(\mathfrak{p})\circ \varepsilon(\mathfrak{s}).$$ Thus (E3) is true.

Let ${\rm id}_{q^{\downarrow}}\in P_{M(P)}$ and ${\rm id}_{q^{\downarrow}}\leq_{M(P)} \dd(\mathfrak{p})$. In the case that $k=1$, by the fact that ${\rm id}_{q^{\downarrow}}\leq_{M(P)} \dd(\mathfrak{p})={\rm id}_{p_{1}^{\downarrow}}$, Lemma \ref{b} (11)  and (E1), $$\varepsilon({_{{\rm id}_{q^{\downarrow}}}}{\downharpoonleft}\mathfrak{p})=\varepsilon({_{{\rm id}_{q^{\downarrow}}}}{\downharpoonleft}{\rm id}_{p_{1}^{\downarrow}})=\varepsilon({\rm id}_{q^{\downarrow}})
={\rm id}_{q^{\downarrow}}
={_{{\rm id}_{q^{\downarrow}}}}{\downharpoonleft}{\rm id}_{p_{1}^{\downarrow}}={_{{\rm id}_{q^{\downarrow}}}}{\downharpoonleft}\varepsilon({\rm id}_{p_{1}^{\downarrow}})={_{{\rm id}_{q^{\downarrow}}}}{\downharpoonleft}\varepsilon(\mathfrak{p}).$$
Let ${_{{\rm id}_{q^{\downarrow}}}}{\downharpoonleft}\mathfrak{p}=({\rm id}_{q_{1}^{\downarrow}},{\rm id}_{q_{2}^{\downarrow}},\ldots,{\rm id}_{q_{k}^{\downarrow}})$ and $$\mathfrak{p}'=({\rm id}_{p_{1}^{\downarrow}},{\rm id}_{p_{2}^{\downarrow}},\ldots,{\rm id}_{p_{k-1}^{\downarrow}})
,\,\,\,\,{_{{\rm id}_{q^{\downarrow}}}}{\downharpoonleft}\mathfrak{p}'=({\rm id}_{q_{1}^{\downarrow}},{\rm id}_{q_{2}^{\downarrow}},\ldots,{\rm id}_{q_{k-1}^{\downarrow}}).$$ Then by induction hypothesis,
$$
\varepsilon({_{{\rm id}_{q^{\downarrow}}}}{\downharpoonleft}\mathfrak{p})=\varepsilon({\rm id}_{q_{1}^{\downarrow}},{\rm id}_{q_{2}^{\downarrow}},\ldots,{\rm id}_{q_{k}^{\downarrow}})
=\gamma_{q_{1},q_{2}}\gamma_{q_{2},q_{3}}\ldots\gamma_{q_{k-1},q_{k}}$$$$
=\varepsilon({_{{\rm id}_{q^{\downarrow}}}}{\downharpoonleft}\mathfrak{p}')\circ \gamma_{q_{k-1},q_{k}}
={_{{\rm id}_{q^{\downarrow}}}}{\downharpoonleft}\varepsilon(\mathfrak{p}')\circ \gamma_{q_{k-1},q_{k}}.
$$
Denote $Y=q^{\downarrow}(\gamma_{p_{1},p_{2}}\gamma_{p_{2},p_{3}}\ldots\gamma_{p_{k-2},p_{k-1}})$.  By (\ref{tiocn}),  $${_{{\rm id}_{q^{\downarrow}}}}{\downharpoonleft}\varepsilon(\mathfrak{p})=(\gamma_{p_{1},p_{2}}\gamma_{p_{2},p_{3}}\ldots\gamma_{p_{k-1},p_{k}}){\mid}{_{q^{\downarrow}}}
$$$$=(\gamma_{p_{1},p_{2}}\gamma_{p_{2},p_{3}}\ldots\gamma_{p_{k-2},p_{k-1}}){\mid}{_{q^{\downarrow}}}\circ\gamma_{p_{k-1},p_{k}}{\mid}_{Y}
=\varepsilon(\mathfrak{p}'){\mid}{_{q^{\downarrow}}}\circ \gamma_{p_{k-1},p_{k}}{\mid}_{Y}
={_{{\rm id}_{q^{\downarrow}}}}{\downharpoonleft}\varepsilon(\mathfrak{p}')\circ \gamma_{p_{k-1},p_{k}}{\mid}_{Y}.$$
We shall prove that $\gamma_{p_{k-1},p_{k}}{\mid}_{Y}=\gamma_{q_{k-1},q_{k}}$. By (\ref{v}),
$$Y=q^{\downarrow}(\gamma_{p_{1},p_{2}}\gamma_{p_{2},p_{3}}\ldots\gamma_{p_{k-2},p_{k-1}})
$$$$=(q\gamma_{p_{1},p_{2}}\gamma_{p_{2},p_{3}}\ldots\gamma_{p_{k-2},p_{k-1}})^{\downarrow}
=(q\theta_{p_{2}}\ldots\theta_{p_{k-1}})^{\downarrow}=q_{k-1}^{\downarrow}.$$ This together with the fact that $q_{k-1}\leq_P p_{k-1}$ implies that $$\gamma_{p_{k-1},p_{k}}{\mid}{_{Y}}=(\theta_{p_{k}}{\mid}{_{{p_{k-1}^{\downarrow}})}}){\mid}{_{q_{k-1}^{\downarrow}}}
=\theta_{p_{k}}{\mid}{_{q_{k-1}^{\downarrow}}}.$$ Let $t\in q_{k-1}^{\downarrow}$. By (\ref{l}) and (\ref{v}), we have $t=t\theta_{q_{k-1}}=t\star q_{k-1}$ and $q_{k}=q_{k-1}\theta_{p_{k}}=q_{k-1}\star p_{k}$. This together  with (R3) implies that
$$t\theta_{q_{k}}=t\star q_{k}=(t\star q_{k-1})\star(q_{k-1}\star p_{k})=(t\star q_{k-1})\star p_{k}=t\theta_{q_{k-1}}\theta_{p_{k}}=t\theta_{p_{k}}.$$  Thus $\gamma_{p_{k-1},p_{k}}{\mid}_{Y}=\theta_{p_{k}}{\mid}{_{q_{k-1}^{\downarrow}}}=\theta_{q_{k}}{\mid}{_{q_{k-1}^{\downarrow}}}=\gamma_{q_{k-1},q_{k}},$ and so  $\varepsilon({_{{\rm id}_{q^{\downarrow}}}}{\downharpoonleft}\mathfrak{p})={_{{\rm id}_{q^{\downarrow}}}}{\downharpoonleft}\varepsilon(\mathfrak{p}).$  By mathematical induction, for all ${\rm id}_{q^{\downarrow}}\in P_{M(P)}$ and $\mathfrak{p}\in \mathscr{P}(P_{M(P)})$, the fact that  ${\rm id}_{q^{\downarrow}}\leq_{M(P)} \dd(\mathfrak{p})$ implies that $\varepsilon({_{{\rm id}_{q^{\downarrow}}}}{\downharpoonleft}\mathfrak{p})={_{{\rm id}_{q^{\downarrow}}}}{\downharpoonleft}\varepsilon(\mathfrak{p}).$  This gives (E5). By Proposition
\ref{fder}, $\varepsilon$ is an evaluation map.
\end{proof}

\begin{prop}\label{sozr}
Let $(P,\times,\star)$ be a strong projection algebra. Then $$(M(P),\circ,\dd,\rr, \leq_{M(P)},  P_{M(P)}, \times, \star,\varepsilon)$$ is a chain projection ordered category.
\end{prop}
\begin{proof}
Let $\beta\in M(q,r), {\rm id}_{e^{\downarrow}},{\rm id}_{f^{\downarrow}}\in P_{M(P)}, e,f\in P$  and $( {\rm id}_{e^{\downarrow}},{\rm id}_{f^{\downarrow}})$ be a $\beta$-linked pair.  Then
\begin{equation}\label{ddew}
{\rm id}_{f^{\downarrow}}={\rm id}_{e^{\downarrow}}\Theta_{\beta}\theta_{{\rm id}_{f^{\downarrow}}},\,{\rm id}_{e^{\downarrow}}={\rm id}_{f^{\downarrow}}\Delta_{\beta}\delta_{{\rm id}_{e^{\downarrow}}}.
\end{equation}
By the definition of $\varepsilon$ (see Lemma \ref{aufrp}), (\ref{fff})--(\ref{mmm}), Lemma \ref{yind2}, (\ref{ahdy}) and (\ref{yind5}),
$$\lambda({\rm id}_{e^{\downarrow}},\beta,{\rm id}_{f^{\downarrow}})=\varepsilon({\rm id}_{e^{\downarrow}},{\rm id}_{e_{1}^{\downarrow}})\circ {_{{\rm id}_{e_{1}^{\downarrow}}}}{\downharpoonleft}\beta\circ \varepsilon({\rm id}_{f_{1}^{\downarrow}},{\rm id}_{f^{\downarrow}})=\gamma_{ee_{1}}\beta{\mid}{_{e_{1}^{\downarrow}}} \gamma_{f_{1}f}=\theta_{e_{1}}{\mid}{_{e^{\downarrow}}}\beta{\mid}{_{e_{1}^{\downarrow}}} \theta_{f}{\mid}{_{f_{1}^{\downarrow}}},$$
$$\rho({\rm id}_{e^{\downarrow}},\beta,{\rm id}_{f^{\downarrow}})=\varepsilon({\rm id}_{e^{\downarrow}},{\rm id}_{e_{2}^{\downarrow}})\circ \beta{\downharpoonright}{_{{\rm id}_{f_{2}^{\downarrow}}}}\circ \varepsilon({\rm id}_{f_{2}^{\downarrow}},{\rm id}_{f^{\downarrow}})=\gamma_{ee_{2}}\beta{\mid}{_{e_{2}^{\downarrow}}} \gamma_{f_{2}f}=\theta_{e_{2}}{\mid}{_{e^{\downarrow}}}\beta{\mid}{_{e_{2}^{\downarrow}}} \theta_{f}{\mid}{_{f_{2}^{\downarrow}}},$$
where $${\rm id}_{e_{1}^{\downarrow}}=({\rm id}_{e^{\downarrow}})\theta_{{\rm id}_{q^{\downarrow}}}={\rm id}_{e^{\downarrow}}\star {\rm id}_{q^{\downarrow}}={\rm id}_{(e\star q)^{\downarrow}},\,\,{\rm id}_{e_{2}^{\downarrow}}={\rm id}_{f^{\downarrow}}\Delta_{\beta}={\rm id}_{((r\times f)\beta^{-1})^{\downarrow}},$$$$ {\rm id}_{f_{1}^{\downarrow}}={\rm id}_{e^{\downarrow}}\Theta_{\beta}={\rm id}_{((e\star q)\beta)^{\downarrow}},\,\,{\rm id}_{f_{2}^{\downarrow}}={\rm id}_{f_{1}^{\downarrow}}\delta_{{\rm id}_{r^{\downarrow}}}={\rm id}_{r^{\downarrow}}\times {\rm id}_{f^{\downarrow}}={\rm id}_{(r\times f)^{\downarrow}}.$$
So $$e_{1}=e\star q,\,\,e_{2}=(r\times f)\beta^{-1},\,\,f_{1}=(e\star q)\beta,\,\,f_{2}=r\times f.$$
We shall prove that $\lambda({\rm id}_{e^{\downarrow}},\beta,{\rm id}_{f^{\downarrow}})=\rho({\rm id}_{e^{\downarrow}},\beta,{\rm id}_{f^{\downarrow}})$. This is equivalent to prove that   $t\theta_{e_{1}}\beta\theta_{f}=t\theta_{e_{2}}\beta\theta_{f}$ for $t\in P$ with $t\in e^{\downarrow}$.
Since $t\leq_{P} e$,  by  (\ref{l}) we have $t=t\theta_{e}=t\star e$. This together with the fact $e_{1}=e\theta_{q}$ and (R3) implies that
$$t\theta_{e_{1}}\beta\theta_{f}=t\theta_{e}\theta_{e\theta_{q}}\beta\theta_{f}=((t\star e)\star(e\star q))\beta\theta_{f}=((t\star e)\star q)\beta\theta_{f}=t\theta_{e}\theta_{q}\beta\theta_{f}=t\theta_{q}\beta\theta_{f}.$$
On the other hand, let $p\in P$.  By Lemma \ref{ccc} (2), $${\rm id}_{f^{\downarrow}}=\rr(\beta)\star {\rm id}_{f^{\downarrow}}={\rm id}_{r^{\downarrow}}\star {\rm id}_{f^{\downarrow}}={\rm id}_{(r\star f)^{\downarrow}},$$  and so  $f=r\star f.$  Since $p\star f\leq_{P} f=r\star f$ by Lemma \ref{jiben1} (2), we obtain  $(r\times(p\star f))\star f=p\star f$ by (\ref{wang2}). This implies that $$p\theta_{f}\delta_{r}\theta_{f}=(r\times(p\star f))\star f=p\star f=p\theta_{f}.$$ Thus $\theta_{f}\delta_{r}\theta_{f}=\theta_{f}.$
In view of (G1b), (\ref{ahdy}) and (\ref{yind5}), we have
$${\rm id}_{(t\star e_{2})^{\downarrow}}={\rm id}_{t^{\downarrow}}\star {\rm id}_{e_{2}^{\downarrow}}={\rm id}_{t^{\downarrow}}\theta_{{\rm id}_{e_{2}^{\downarrow}}}={\rm id}_{t^{\downarrow}}\theta_{{\rm id}_{f^{\downarrow}}\Delta_{\beta}}={\rm id}_{t^{\downarrow}}
\Theta_{\beta}\theta_{{\rm id}_{f^{\downarrow}}}\Delta_{\beta}={\rm id}_{((r\times((t\star q)\beta \star f))\beta^{-1})^{\downarrow}},$$ and so $t\theta_{e_{2}}=t\star e_{2}=(r\times((t\star q)\beta \star f))\beta^{-1}.$ By the above statements,
$$t\theta_{e_{2}}\beta\theta_{f}=(r\times((t\star q)\beta \star f))\beta^{-1}\beta\theta_{f}=(r\times((t\star q)\beta \star f))\theta_{f}=t\theta_{q}\beta\theta_{f}\delta_{r}\theta_{f}=t\theta_{q}\beta\theta_{f}.$$
This gives that $t\theta_{e_{1}}\beta\theta_{f}=t\theta_{e_{2}}\beta\theta_{f}.$
Therefore $\lambda({\rm id}_{e^{\downarrow}},\beta,{\rm id}_{f^{\downarrow}})=\rho({\rm id}_{e^{\downarrow}},\beta,{\rm id}_{f^{\downarrow}}).$
\end{proof}

Let $(P,\times,\star)$ be a strong projection algebra. Define a binary operation $``\bullet"$ and two unary operations $``^{+}"$ and $``^{\ast}"$ on $M(P)$ as follows: For all $\alpha\in M(p,q),\beta\in M(s,t)$,
$$\alpha\bullet \beta=\alpha{\downharpoonright}{_{\rr(\alpha)\times \dd(\beta)}} \circ\, \varepsilon(\rr(\alpha)\times \dd(\beta),\rr(\alpha)\star \dd(\beta))\,\circ{ _{\rr(\alpha)\star \dd(\beta)}}{\downharpoonleft} \beta,\,\,\alpha^{+}=\dd(\alpha),\,\,\alpha^{\ast}=\rr(\alpha).$$
By (\ref{kkk}), Theorem \ref{xii} and Pproposition \ref{sozr}, $\textbf{S}(M(P))=(M(P),\bullet,^{+},^{\ast})$ forms a DRC-restriction semigroup ([cf. Theorem 4.7 in \cite{Wang6}]). Since $\dd(\alpha)={\rm id}_{p^{\downarrow}}$ and $\rr(\alpha)={\rm id}_{q^{\downarrow}}$, we have $\alpha^{+}={\rm id}_{p^{\downarrow}}$ and  $\alpha^{\ast}={\rm id}_{q^{\downarrow}}$. Moreover, $$\alpha{\downharpoonright}{_{\rr(\alpha)\times \dd(\beta)}}=\alpha{\downharpoonright}{_{{\rm id}_{q^{\downarrow}}\times {\rm id}_{s^{\downarrow}}}}=\alpha{\downharpoonright}{_{{\rm id}_{(q\times s)^{\downarrow}}}}=\alpha{\mid}{_{((q\times s)\alpha^{-1})^{\downarrow}}}.$$ Similarly, we have $${ _{\rr(\alpha)\star \dd(\beta)}}{\downharpoonleft} \beta=\beta{\mid}{_{(q\star s)^{\downarrow}}},\varepsilon(\rr(\alpha)\times \dd(\beta),\rr(\alpha)\star \dd(\beta))=\varepsilon({\rm id}_{(q\times s)^{\downarrow}},{\rm id}_{(q\star s)^{\downarrow}})=\gamma_{q\times s,q\star s}.$$ Thus
\begin{equation}\label{gpyin6}
\alpha \bullet \beta=\alpha{\mid}{_{((q\times s)\alpha^{-1})^{\downarrow}}}\circ\gamma_{q\times s,q\star s}\circ
\beta{\mid}{_{(q\star s)^{\downarrow}}}=\alpha\theta_{q\star s}\beta{\mid}{_{((q\times s)\alpha^{-1})^{\downarrow}}}.
\end{equation}
If $\rr(\alpha)=\dd(\beta)$, i.e. $q=s$, we have $q\times s=q\star s=q=s$,  and so $(q\times s)\alpha^{-1}=q\alpha^{-1}=p$. This yields that $\alpha \bullet \beta=\alpha\theta_{q}\beta{\mid}{_{p^{\downarrow}}}=\alpha\beta.$
In particular, for any $\alpha\in M(p,q)$, we have $\alpha^{-1}\in M(q,p)$, and so $$\alpha \bullet \alpha^{-1}=\alpha\alpha^{-1}={\rm id}_{p^{\downarrow}}=\alpha^{+}=(\alpha^{-1})^{\ast},\,\,\alpha^{-1} \bullet \alpha=\alpha^{-1}\alpha={\rm id}_{q^{\downarrow}}=\alpha^{\ast}=(\alpha^{-1})^{+}.$$
By Lemma \ref{aisdm}, $\textbf{S}(M(P))$ is a generalized regular ${\circ}$-semigroup with set of projections $P(\textbf{S}(M(P)))=P_{M(P)}=\{{\rm id}_{p^{\downarrow}}\mid p\in P\}.$  Let $p,q\in P$. By Lemma \ref{edd},  $${\rm id}_{p^{\downarrow}}\bullet {\rm id}_{q^{\downarrow}}=\varepsilon({\rm id}_{p^{\downarrow}}\times {\rm id}_{q^{\downarrow}},{\rm id}_{p^{\downarrow}}\star {\rm id}_{q^{\downarrow}})=\varepsilon({\rm id}_{(p\times q)^{\downarrow}},{\rm id}_{(p\star q)^{\downarrow}})=\gamma_{p\times q,p\star q}.$$  In particular, if $p\,\mathcal{F}_{P}\,q,$  then ${\rm id}_{p^{\downarrow}}\bullet {\rm id}_{q^{\downarrow}}=\gamma_{p,q}.$

\begin{theorem}[cf. Theorem 5.3 in \cite{Wang6}]\label{dpwgl0}
Let $(S,\cdot,^{+},^{\ast})$ be a DRC-restriction semigroup and $(P(S),\times_{S},\star_{S})$ be the strong projection algebra of $S$. For all $a\in S$, denote $$\nu_{a}:(a^{+})^{\downarrow}\rightarrow (a^{\ast})^{\downarrow},\,\,x\mapsto (xa)^{\ast},$$ where $$(a^{+})^{\downarrow}=\{x\in P(S)\mid x\leq_{S}a^{+}\},\,\,(a^{\ast})^{\downarrow}=\{x\in P(S)\mid x\leq_{S}a^{\ast}\}.$$ Then $\phi_{S}:S\rightarrow {\rm \bf{S}}(M(P(S))),\,\,\,a\mapsto \nu_{a}$ is a (2,1,1)-homomorphism with $\ker\phi_{S}=\mu_{S}$ and $p\phi_{S}={\rm id}_{p^{\downarrow}}$ for all $p\in P(S)$, and so $S$ is fundamental if and only if $\phi$ is injective. Moreover, ${\rm im}(\phi_{S})\cong S/\mu_{S}$ is fundamental.
\end{theorem}
\begin{proof}
By Theorem \ref{fii}, we can consider the chain projection ordered category $$\textbf{\rm \bf C}(S)=(S, \circ_S, \dd_S, \rr_S, \leq_S, P_{S},\times_S, \star_S, \varepsilon_S).$$ Let $a\in S$. By (G1d) and (\ref{bope}),
$$\nu_{a}:(a^{+})^{\downarrow}=\dd_{S}(a)^{\downarrow}\rightarrow \rr_{S}(a)^{\downarrow}=(a^{\ast})^{\downarrow},\,\,\,x\mapsto \rr_{S}({_x}{\downharpoonleft} a)=(xa)^{\ast}$$ is an isomorphism, and so  $\nu_{a}\in M(a^{+},a^{\ast})$. So we can define
$$\phi_{S}:\textbf{\rm \bf C}(S)\rightarrow M(P(S)) ,\,\,a\rightarrow \nu_{a}.$$
We shall prove that $\phi_{S}$ is a chain projection ordered functor. Let $a,b\in S$.

(F1) Firstly,  we have $\dd(a\phi_{S})=\dd(\nu_{a})={\rm id}_{(a^{+})^{\downarrow}}.$ Since $a^{+}\in P(S)=P_{S}$, it follows that $\nu_{a^{+}}={\rm id}_{(a^{+})^{\downarrow}}$ by Lemma \ref{f}.  This implies that $\dd(a\phi_{S})=\nu_{a^{+}}=a^{+}\phi_{S}=(\dd_{S}(a))\phi_{S}.$ Dually, $\rr(a\phi_{S})=(\rr_{S}(a))\phi_{S}.$

(F2) Let  $\rr_{S}(a)=\dd_{S}(b)$. Then $a^{\ast}=b^{+}$ and $\rr(\nu_{a})={\rm id}_{(a^{\ast})^{\downarrow}}={\rm id}_{(b^{+})^{\downarrow}}=\dd(\nu_{b})$.
By Lemma \ref{ee}, $$(a\circ_{S} b)\phi_{S}=\nu_{a\circ_{S} b}=\nu_{a}\nu_{b}=\nu_{a}\circ \nu_{b}=(a\phi_{S})\circ(b\phi_{S}).$$

(F3) Let $a\leq_{S}b$. By Lemma \ref{b} (3), we obtain  $a={_{\dd_{S}(a)}}{\downharpoonleft}b={_{a^{+}}}{\downharpoonleft}b$, this together with Lemma \ref{apdxn} gives that $\nu_{a}=\nu_{{_{a^{+}}}{\downharpoonleft}b}=\nu_{b}{\mid}{_{(a^{+})^{\downarrow}}}
=\nu_{b}{\mid}{_{\dm \nu_{a}}}.$ This shows that  $a\phi_{S}=\nu_{a}\leq_{M(P)} \nu_{b}=b\phi_{S}$ in $(M(P),\circ,\dd,\rr,\leq)$.

(F4) Let $e,f\in P_{S}=P(S)$. By Lemma \ref{f}, $$(e\times_{S} f)\phi_{S}=\nu_{e\times_{S} f}={\rm id}_{(e\times_{S} f)^{\downarrow}}={\rm id}_{e^{\downarrow}}\times {\rm id}_{f^{\downarrow}}=\nu_{e}\times \nu_{f}=(e\phi_{S})\times (f\phi_{S}).$$

(F5) Let $\mathfrak{c}=(p_{1},p_{2},\ldots,p_{k})\in \mathscr{P}(P(S)).$ Using (\ref{pp}), (F2), Lemma \ref{oo} (2), (\ref{avtki}) and Lemma \ref{aufrp},  we have
\begin{equation*}
\begin{aligned}
(\varepsilon_{S}(\mathfrak{c}))\phi_{S}&=(\varepsilon_{S}(p_{1},p_{2})\circ \ldots \circ \varepsilon_{S}(p_{k-1},p_{k}))\phi_{S} =(\varepsilon_{S}(p_{1},p_{2}))\phi_{S}\circ \ldots \circ (\varepsilon_{S}(p_{k-1},p_{k}))\phi_{S}\\
&=\nu_{\varepsilon_{S}(p_{1},p_{2})}\circ \ldots \circ \nu_{\varepsilon_{S}(p_{k-1},p_{k})}=\theta_{p_{2}}{\mid}{_{p_{1}^{\downarrow}}}\circ \ldots \circ \theta_{p_{k}}{\mid}{_{p_{k-1}^{\downarrow}}}\\
&=\gamma_{p_{1},p_{2}}\circ \ldots \circ \gamma_{p_{k-1},p_{k}}=\varepsilon({\rm id}_{p_{1}^{\downarrow}},{\rm id}_{p_{2}^{\downarrow}})\circ \ldots \circ \varepsilon({\rm id}_{p_{k-1}^{\downarrow}},{\rm id}_{p_{k}^{\downarrow}})=\varepsilon({\rm id}_{p_{1}^{\downarrow}},{\rm id}_{p_{2}^{\downarrow}},\ldots,{\rm id}_{p_{k}^{\downarrow}}).
\end{aligned}
\end{equation*}
Since  $p_{i}\in P(S)$ for all $i$, it follows that $p_{i}\phi_{S}=\nu_{p_{i}}={\rm id}_{p_{i}^{\downarrow}}$ by Lemma \ref{f}. Thus $(\varepsilon_{S}(\mathfrak{c}))\phi_{S}=\varepsilon((p_{1}\phi_{S},p_{2}\phi_{S},\ldots,p_{k}\phi_{S})).$
By items (F1)--(F5), $\phi_{S}$ is a chain projection ordered functor. By Lemma \ref{wang} and Theorem \ref{hgdc}, $\phi_{S}$ is a (2,1,1)-homomorphism from
$S=\textbf{S}(\textbf{\rm \bf C}(S))$ to $\textbf{S}(M(P(S)))$,  and $p\phi_{S}=\nu_{p}={\rm id}_{p^{\downarrow}}$ for all $p\in P(S)$. Finally, $$\ker\phi_{S}=\{(a,b)\in S\times S\mid a\phi_{S}=b\phi_{S}\}=\{(a,b)\in S\times S\mid \nu_{a}=\nu_{b}\}=\mu_{S}$$ by Proposition \ref{fscpl}. The remaining part  of the theorem  is obvious.
\end{proof}
\begin{theorem}[cf. Theorem 5.4 in \cite{Wang6}]\label{dpwgl}Let $(P,\times,\star)$ be a strong projection algebra and $(S,\cdot,^{+},^{\ast})$ be a DRC-restriction semigroup whose strong projection algebra is isomorphic to $(P,\times,\star)$. Then $S$ is fundamental if and only if $S$ is isomorphic to some full (2,1,1)-subalgebra of ${\rm \bf{S}}(M(P))$. In particular, ${\rm \bf{S}}(M(P))$ is projection-fundamental.
\end{theorem}
\begin{proof}
Assume that $S$ is projection-fundamental. By Theorem \ref{dpwgl0}, $$\phi_{S}:S\rightarrow \textbf{S}(M(P(S)))(\cong\textbf{S}(M(P)))$$ is an injective (2,1,1)-homomorphism and  $p\phi_{S}={\rm id}_{p^{\downarrow}}$ for all $p\in P(S)$. Thus
$$\im\phi_{S}\supseteq P(S)\phi_{S}=\{p\phi_{S}\mid p\in P(S)\}=\{{\rm id}_{p^{\downarrow}}\mid p\in P(S)\}=P(\textbf{S}(M(P(S)))).$$
This  implies that  $S\cong \im\phi_{S} $ and  $\im\phi_{S}$  is a full (2,1,1)-subalgebra of $\textbf{S}(M(P))$.

Conversely, assume that $T$  is a full (2,1,1)-subalgebra of $\textbf{S}(M(P))$.  Let $\alpha,\beta\in T$ and $(\alpha,\beta)\in \mu_{T}$. By Lemma \ref{dvhi}, we have ${\rm id}_{\dm\alpha}=\alpha^{+}=\beta^{+}={\rm id}_{\dm\beta}$, and so we can denote $e^{\downarrow}=\dm\alpha=\dm\beta$ for some $e\in P$.
Let $x\in e^{\downarrow}$. Then $x\in P$ and $x\leq_{P}e$, and so ${\rm id}_{ x^{\downarrow}}\in P(\textbf{S}(M(P)))\subseteq T.$ This gives that
$({\rm id}_{ x^{\downarrow}}\bullet \alpha,{\rm id}_{ x^{\downarrow}}\bullet \beta)\in \mu_{T}$. Similar discussion gives that $\ra({\rm id}_{ x^{\downarrow}}\bullet \alpha)=\ra({\rm id}_{ x^{\downarrow}}\bullet \beta)$. By $(\ref{gpyin6})$, we have  $x^{\downarrow}\alpha=x^{\downarrow}\beta$, i.e $(x\alpha)^{\downarrow}=(x\beta)^{\downarrow}$. Thus $x\alpha=x\beta$. we have shown that $\alpha=\beta$. Therefore $\mu_{T}$ is the identity relation and so $T$ is projection-fundamental.
\end{proof}

\begin{theorem}\label{yin12}
Let $(P,\times,\star)$ be a strong projection algebra, $\mu=\mu_{\mbox{{\bf S}}(\mathscr{C}(P))}$ and  $(S, \cdot, ^+, ^\ast)$ be a projection- fundamental DRC-restriction semigroup whose strong projection algebra is isomorphic to $(P,\times,\star)$. Then $\mbox{{\bf E}}(S)\cong \mbox{{\bf S}}(\mathscr{C}(P))/\mu.$ In particular if $(S, \cdot, ^+, ^\ast)$ is projection-generated, then $S\cong\mbox{{\bf S}}(\mathscr{C}(P))/\mu.$
\end{theorem}
\begin{proof}
Assume that $(S, \cdot, ^+, ^\ast)$ is a projection-fundamental DRC-restriction semigroup whose strong projection algebra is  $(P,\times,\star)$. By Theorem \ref{fii}, $\textbf{\rm \bf C}(S)$ is a chain projection ordered category.  By Proposition \ref{lope}, $\overline{\varepsilon}: \mathscr{C}(P_{S})\rightarrow \textbf{\rm \bf C}(S)$ is a chain projection ordered functor from
$\mathscr{C}(P_{S})$ to $\textbf{\rm \bf C}(S)$,  and so is a (2,1,1)-homomorphism from ${\bf S}(\mathscr{C}(P_{S}))$ to ${\bf S}(\textbf{\rm \bf C}(S))=S$ (by Theorem \ref{hgdc}) such that $[p]\overline{\varepsilon}=p$ for all $p\in P$. According to Theorem \ref{dpwgl0}, there exists a (2,1,1)-homomorphism $$\phi_{S}:S\rightarrow \textbf{S}(M(P)),\,\,\,a\rightarrow \nu_{a}$$ such that $p\mapsto {\rm id}_{p^\downarrow}$ for all $p\in P$.
This implies that
$$\overline{\varepsilon}\phi_{S}: {\bf S}(\mathscr{C}(P_{S}))\rightarrow \textbf{S}(M(P))$$ is also a (2,1,1)-homomorphism such that $[p]\overline{\varepsilon}\phi_{S}={\rm id}_{p^\downarrow}$ for all $p\in P$.
In view of Theorem \ref{dpwgl0}, we have a (2,1,1)-homomorphism $$\phi_{{\bf S}(\mathscr{C}(P_{S}))}: {\bf S}(\mathscr{C}(P_{S}))\longrightarrow \textbf{S}(M(P))$$ such that $[p]={\rm id}_{p^\downarrow}$ for all $p\in P$.
By Theorem \ref{llpg}, ${\bf S}(\mathscr{C}(P_{S}))$ is projection-generated, and so  $\overline{\varepsilon}\phi_{S}=\phi_{{\bf S}(\mathscr{C}(P_{S}))}$.
Since $S$ is projection-fundamental, it follows that $\phi_{S}$ is injective by Theorem \ref{dpwgl0}. This together with Theorem \ref{dpwgl0} gives that
$$\ker(\overline{\varepsilon})=\ker(\overline{\varepsilon}\phi_{S})=\ker(\phi_{{\bf S}(\mathscr{C}(P_{S}))})=\mu_{{\bf S}(\mathscr{C}(P_{S}))}.$$
By Proposition \ref{fgpu}, $$\textbf{E}(S)=im(\overline{\varepsilon})\cong {\bf S}(\mathscr{C}(P_{\textbf{\rm \bf C}(S)}))/\mu_{{\bf S}(\mathscr{C}(P_{\textbf{\rm \bf C}(S)}))}\cong {\bf S}(\mathscr{C}(P))/\mu.$$
In particular,  if $(S, \cdot, ^+, ^\ast)$ is projection-generated, then $S=\textbf{E}(S)\cong {\bf S}(\mathscr{C}(P))/\mu$ by Proposition \ref{fgpu} again.
\end{proof}
Let $(P,\times,\star)$ be a strong projection algebra. To end this section, we shall give some remarks on the semigroup $\mbox{{\bf S}}(\mathscr{C}(P))/\mu_{{\bf S}(\mathscr{C}(P))}$.
By Theorem \ref{yin12}, up to isomorphism, there exists a unique projection-generated projection-fundamental  DRC-restriction semigroup whose strong projection algebra is isomorphic to   $(P,\times,\star)$.   This semigroup is $\mbox{{\bf S}}(\mathscr{C}(P))/\mu_{\mbox{{\bf S}}(\mathscr{C}(P))}$. Moreover, if $(S, \cdot, ^+, ^\ast)$ is a  projection-fundamental DRC-restriction semigroup with $P(S)\cong P$, then $\mbox{{\bf E}}(S)\cong \mbox{{\bf S}}(\mathscr{C}(P))/\mu_{{\bf S}(\mathscr{C}(P))}.$
To obtain $\mbox{{\bf S}}(\mathscr{C}(P))/\mu_{{\bf S}(\mathscr{C}(P))}$, we only need to find a  projection-fundamental DRC-restriction semigroup $(S,\cdot,^{+},^{\ast})$ with $P(S)\cong P$ and compute $\mbox{{\bf E}}(S)$. So we have the following three approaches to describe $\mbox{{\bf S}}(\mathscr{C}(P))/\mu_{{\bf S}(\mathscr{C}(P))}$.

{\bf Approach 1}: The elements in ${\bf S}(\mathscr{C}(P))/\mu_{{\bf S}(\mathscr{C}(P))}$ are all $\mu_{{\bf S}(\mathscr{C}(P))}$-classes in ${\bf S}(\mathscr{C}(P))$. Let $$\mathfrak{p}=[p_{1},p_{2},\ldots,p_{k}],\mathfrak{q}=[q_{1},q_{2},\ldots,q_{l}]\in {\bf S}(\mathscr{C}(P)).$$ By Proposition \ref{fscpl} and (\ref{daxzo}), (\ref{yida}), (\ref{edsc}), (\ref{mifg}),
$$(\mathfrak{p},\mathfrak{q})\in \mu_{{\bf S}(\mathscr{C}(P))}\Longleftrightarrow \theta_{p_{1}}\theta_{p_{2}}\ldots\theta_{p_{k}}=\theta_{q_{1}}\theta_{q_{2}}\ldots\theta_{q_{l}},\,\,\delta_{p_{k}}\delta_{p_{k-1}}\ldots\delta_{p_{1}}=\delta_{q_{l}}\delta_{q_{l-1}}\ldots\theta_{q_{1}}
$$$$\Longleftrightarrow \theta_{p_{1}}\theta_{p_{2}}\ldots\theta_{p_{k}}=\theta_{q_{1}}\theta_{q_{2}}\ldots\theta_{q_{l}},\,\,p_{1}=q_{1}$$$$\Longleftrightarrow p_{1}=q_{1},\,\, \theta_{p_{2}}\theta_{p_{3}}\ldots\theta_{p_{k}}|_{p_1^{\downarrow}}=\theta_{q_{2}}\theta_{q_{3}}\ldots\theta_{q_{l}}|_{q_1^{\downarrow}}.$$

{\bf Approach 2}: Since $\textbf{S}(M(P))$ is projection-fundamental and its strong projection algebra is isomorphic to $(P, \times, \star)$. By Theorem \ref{yin12},
${\bf S}(\mathscr{C}(P))/\mu_{{\bf S}(\mathscr{C}(P))}$ is isomorphic to $\textbf{E}(\textbf{S}(M(P))$. The elements in $\textbf{E}(\textbf{S}(M(P))$ are:
$${\rm id}_{p_{1}^{\downarrow}} \bullet {\rm id}_{p_{2}^{\downarrow}}\bullet \ldots \bullet {\rm id}_{p_{k}^{\downarrow}}, \,\,\, p_{1},p_{2},\ldots,p_{k}\in P,\,\, k \mbox{ is an integer}. $$

{\bf Approach 3}: Let $(P,\times,\star)$ be a strong projection algebra. Then we have the generalized regular ${\circ}$-semigroup $\mbox{{\bf S}}(M(P))=(M(P),\bullet,^{+},^{\ast},^{-1})$. Denote $$\overline{\textbf{S}(M(P))}=\{(\theta_{p}\alpha,\delta_{q}\alpha^{-1})\mid\alpha\in M(p,q),p,q\in P\}$$  and define $$\varphi:\mbox{{\bf S}}(M(P))\rightarrow \overline{\textbf{S}(M(P)},\,\,\alpha\in M(p,q)\mapsto (\theta_{p}\alpha,\delta_{q}\alpha^{-1}).$$ We assert that $\varphi$ is a bijection. Obviously, $\varphi$ is surjective. Let $\alpha\in M(p,q),\beta\in M(s,t)$ and $(\theta_{p}\alpha,\delta_{q}\alpha^{-1})=(\theta_{s}\beta,\delta_{t}\beta^{-1})$.
Then $\theta_{p}\alpha=\theta_{s}\beta,\,\,\delta_{q}\alpha^{-1}=\delta_{t}\beta^{-1}.$
Let $x \in P$. Then  $x(\theta_{p}\alpha)=x(\theta_{s}\beta)$, i.e. $(x\star p)\alpha=(x\star s)\beta$. Take $x=p$. Then $q=p\alpha=(p\star p)\alpha=(p\star s)\beta\leq_{P} s\beta=t.$ Similarly, take $x=s$. Then we have $t\leq_{P} q$. Thus $q=t$. Dually, we can obtain $p=s$ form the fact that $\delta_{q}\alpha^{-1}=\delta_{t}\beta^{-1}$. Thus $\alpha,\beta\in M(p,q)$. Take $x\in \dm\alpha=p^{\downarrow}$ arbitrarily. Then $x\leq_{P} p$,  and so $x\star p=x$ by Lemma \ref{jiben1} (1). Thus $$x\alpha=(x\star  p)\alpha=x\theta_{p}\alpha=x\theta_{p}\beta=(x\star p)\beta=x\beta.$$  This gives $\alpha=\beta$. Therefore $\varphi$ is injective.

Since $\varphi$ is bijective, we can transfer the operations on $\mbox{{\bf S}}(M(P))$ to $\overline{\textbf{S}(M(P)}$ as follows. Let $\alpha \in M(p,q),\beta\in M(s,t)$. By (\ref{gpyin6}),  $\alpha\bullet \beta \in M(((q\times s)\alpha^{-1},(q\star s)\beta))$. So we define
\begin{equation*}
\begin{aligned}
(\theta_{p}\alpha,\delta_{q}\alpha^{-1})(\theta_{s}\beta,\delta_{t}\beta^{-1})&=(\alpha\bullet \beta )\varphi=(\alpha\gamma_{q\times s,q\star s}\beta)\varphi\\
&=(\theta_{(q\times s)\alpha^{-1}}\alpha\gamma_{q\times s,q\star s}\beta,\delta_{(q\star s)\beta}\beta^{-1}\gamma_{q\times s,q\star s}^{-1}\alpha^{-1}),
\end{aligned}
\end{equation*}
 $$(\theta_{p}\alpha,\delta_{q}\alpha^{-1})^{+}=\alpha^{+}\varphi={\rm id}_{p^{\downarrow}}\varphi=(\theta_{p}{\rm id}_{p^{\downarrow}},\delta_{p}{\rm id}_{p^{\downarrow}}^{-1})=(\theta_{p},\delta_{p}).$$
$$(\theta_{p}\alpha,\delta_{q}\alpha^{-1})^{\ast}=\alpha^{\ast}\varphi={\rm id}_{q^{\downarrow}}\varphi=(\theta_{q}{\rm id}_{q^{\downarrow}},\delta_{q}{\rm id}_{q^{\downarrow}}^{-1})=(\theta_{q},\delta_{q}).$$
Let $\alpha \in M(p,q)$ and $\beta\in M(s,t)$. Then $$\alpha\bullet \beta=\alpha\gamma_{q\times s,q\star s}\beta \in M((q\times s)\alpha^{-1},(q\star s)\beta).$$ Assume that $x\in P$. Then we have
\begin{equation*}
\begin{aligned}
x\theta_{(q\times s)\alpha^{-1}}&=x\theta_{p}\theta_{(q\times s)\alpha^{-1}}\,\,\,\,\,\,\,\,(\mbox{by }(q\times s)\alpha^{-1}\leq_{P} p \mbox{ and Lemma \ref{o}})\\
&=(x\theta_{p})\alpha\alpha^{-1}\theta_{(q\times s)\alpha^{-1}}\\
&=(x\theta_{p}\alpha\theta_{q\times s})\alpha^{-1}\,\,\,\,\,\,\,\,(\mbox{since }\alpha^{-1}\mbox{ is an isomorphism})\\
&=x\theta_{p}\alpha\theta_{q\times s}\delta_{q}\alpha^{-1}\,\,\,\,\,\,\,\,(\mbox{by} q\times s\leq_{P} q,\mbox{and Lemma \ref{o}})\\
&=x\theta_{p}\alpha\theta_{(q\star s)\delta_{q}}\delta_{q}\alpha^{-1}\,\,\,\,\,\,\,\,(\mbox{by } (q\star s)\delta_{q}=q\times(q\star s)=q\times s \,\,\mbox{and Lemma \ref{jiben1} (3)} )\\
&=x\theta_{p}\alpha\theta_{q}\theta_{q\star s}\delta_{q}\delta_{q}\alpha^{-1}\,\,\,\,\,\,\,\,(\mbox{by Lemma \ref{jj}})\\
&=x\theta_{p}\alpha\theta_{q\star s}\delta_{q}\alpha^{-1}\,\,\,\,\,\,\,\,(\mbox{by }(x\theta_{p})\alpha\in q^{\downarrow}\,\mbox{and Lemma \ref{o}})
\end{aligned}
\end{equation*}
(R4) gives that $q\star(q\star s)=(q\star s)\star(q\star s)=q\star s$ and Lemma \ref{jiben1} (2) gives that $$(x\theta_{p}\alpha)\star(q\star s)\leq_{P} q\star s=q\star(q\star s).$$ This together with (\ref{wang2}) implies that $$(q\times((x\theta_{p}\alpha)\star(q\star s)))\star(q\star s)=(x\theta_{p}\alpha)\star(q\star s).$$ Thus
\begin{equation*}
\begin{aligned}
x\theta_{(q\times s)\alpha^{-1}}\alpha\gamma_{q\times s,q\star s}\beta&=x\theta_{(q\times s)\alpha^{-1}}\alpha\theta_{q\star s}\beta
=x\theta_{p}\alpha\theta_{q\star s}\delta_{q}\alpha^{-1}\alpha\theta_{q\star s}\beta
=x\theta_{p}\alpha\theta_{q\star s}\delta_{q}\theta_{q\star s}\beta\\
&=((q\times((x\theta_{p}\alpha)\star(q\star s)))\star(q\star s))\beta
=((x\theta_{p}\alpha)\star(q\star s))\beta\,\,\,\,\,\,\,\,(\mbox{(\ref{wang2})})\\
&=((x\theta_{p}\alpha\star q)\star(q\star s))\beta\,\,\,\,\,(x\theta_{p}\alpha\leq_{P} q,\,\,\mbox{Lemma \ref{jiben1} (2)})\\
&=((x\theta_{p}\alpha\star q)\star s)\beta\,\,\,\,\,\,\,\,(\mbox{(R3)})\\
&=(x\theta_{p}\alpha\star s)\beta
=x\theta_{p}\alpha\theta_{s}\beta.
\end{aligned}
\end{equation*}
Dually,
$\delta_{(q\star s)\beta}\beta^{-1}\gamma_{q\times s,q\star s}^{-1}\alpha^{-1}=\delta_{t}\beta^{-1}\delta_{q}\alpha^{-1}.$ Thus $$(\theta_{p}\alpha,\delta_{q}\alpha^{-1})(\theta_{s}\beta,\delta_{t}\beta^{-1})=(\theta_{p}\alpha\theta_{s}\beta,\delta_{t}\beta^{-1}\delta_{q}\alpha^{-1}).$$
So  $\varphi$ is a (2,1,1)-isomorphism from $\mbox{{\bf S}}(M(P))$ onto $\overline{\textbf{S}(M(P)}$. Since ${\bf S}(\mathscr{C}(P))/\mu_{{\bf S}(\mathscr{C}(P))}$ is isomorphic to $\textbf{E}(\textbf{S}(M(P))$, ${\bf S}(\mathscr{C}(P))/\mu_{{\bf S}(\mathscr{C}(P))}$ is isomorphic to $\textbf{E}(\overline{\textbf{S}(M(P)})$. Observe that the projections in $\overline{\textbf{S}(M(P)}$ are:
$(\theta_{p},\delta_{p}), p\in P,$  it follows that the elements in
$\textbf{E}(\overline{\textbf{S}(M(P)})$ are: $$(\theta_{p_{1}}\theta_{p_{2}}\ldots\theta_{p_{k}},\delta_{p_{k}}\delta_{p_{k-1}}\ldots\delta_{p_{1}}),\,\, p_{1},p_{2},\ldots,p_{k}\in P,\,\, k \mbox{ is an integer}.$$

\vspace{3mm}
\noindent {\bf Acknowledgements:} The paper is supported jointly  by the Nature Science Foundations of  China (12271442, 11661082).

\end{document}